%% file: RelaxationCamassaHolm_final2.tex
\documentclass[a4paper]{amsart}

\newif\ifcomments

\usepackage{amsmath, amsthm, amsfonts,stmaryrd,amssymb}

\usepackage{mathtools}                     
\usepackage[utf8]{inputenc}
\usepackage{dsfont}                         
\usepackage{graphicx}   
\usepackage[top=1in, bottom=1.25in, left=1.25in, right=1.25in]{geometry}
\usepackage{hyperref} 
\usepackage{xcolor}
\hypersetup{
    colorlinks=true,
    linkcolor=blue,  
    urlcolor=cyan,
    pdftitle={},
    pdfauthor={},
}
\usepackage{theoremref}
\usepackage{subfigure}

\numberwithin{equation}{section}
\theoremstyle{plain}
\newtheorem{theorem}{Theorem}[section]
\newtheorem{remark}[theorem]{Remark}
\newtheorem{lemma}[theorem]{Lemma}
\newtheorem{corollary}[theorem]{Corollary}
\newtheorem{proposition}[theorem]{Proposition}
\newtheorem{problem}[theorem]{Problem}

\theoremstyle{definition}
\newtheorem{definition}[theorem]{Definition}

\def\ed{\mathrm{d}}
\def\cone{{\mc{C}}}  
\def\fmap{h}  

\newcommand{\corr}[1]{#1}

\newcommand{\mres}{\mathbin{\vrule height 1.6ex depth 0pt width
0.13ex\vrule height 0.13ex depth 0pt width 1.3ex}}

\let\bs=\boldsymbol
\let\mb=\mathbb
\let\mc=\mathcal
\let\mf=\mathfrak
\let\mr=\mathrm

\begin{document}
\date{\today}

\title[Generalized solutions of the $H(\mr{div})$ geodesic problem]{Generalized compressible flows and solutions of the $H(\mathrm{div})$ geodesic problem}

\author{Thomas Gallou{\"e}t,  Andrea Natale and Fran\c{c}ois-Xavier Vialard}
\maketitle
\begin{abstract}
We study the geodesic problem on the group of diffeomorphism of a domain $M\subset \mathbb{R}^d$, equipped with the $H(\mr{div})$ metric. The geodesic equations coincide with the Camassa-Holm equation when $d=1$, and represent one of its possible multi-dimensional generalizations when $d>1$. We propose a relaxation \`a la Brenier of this problem, in which solutions are represented as probability measures on the space of continuous paths on the cone over $M$.  
We use this relaxation to prove that smooth $H(\mr{div})$ geodesics are globally length minimizing for short times. We also prove that there exists a unique pressure field associated to solutions of our relaxation.   
Finally, we propose a numerical scheme to construct generalized solutions on the cone 
and present some numerical results illustrating the relation between the generalized Camassa-Holm and incompressible Euler solutions.
\end{abstract}

\section{Introduction}

\corr{

The $H(\mr{div})$ minimizing geodesic problem on the group of diffeomorphisms of a compact domain in $\mathbb{R}^d$ can be stated as follows:
\begin{problem}[$H(\mr{div})$ geodesic problem]\label{prob:det}
Let $M$ be a compact domain in $\mathbb{R}^d$ and let $\mr{Diff}(M)$ be the group of smooth diffeomorphisms of $M$. Denote by $\rho_0$ the Lebesgue measure on $M$. Given  $\fmap\in \mr{Diff}(M)$, find a smooth curve $t\in[0,T] \mapsto \varphi_t \in \mr{Diff}(M)$ satisfying 
\begin{equation}\label{eq:boundaries}
\varphi_0 = \mr{Id}\,, \quad \varphi_T= h\,,
\end{equation}
and minimizing the action $\int_0^T l(u) \,\ed t$, with Lagrangian given by 
\begin{equation}\label{eq:ablag}
l(u)  =  \int_M  | u |^2\, \ed \rho_0 + \frac{1}{4} \int_M  |\mr{div}\,u|^2\,\ed \rho_0\,,
\end{equation}
where $u: [0,T] \times M \rightarrow \mathbb{R}^d$ is the Eulerian velocity field defined by the equation $\partial_t {\varphi}_t(\cdot) = u(t,\varphi_t(\cdot))$.
\end{problem} 


Michor and Mumford proved in \cite{michor2005vanishing} that the $H(\mr{div})$ Lagrangian $\eqref{eq:ablag}$ defines a non-vanishing distance on the diffeomorphism group, in contrast to the $L^2$ case for which the metric is degenerate (i.e., there exist non trivial maps $h$  for which the infimum of the action vanishes). 
Note also that in dimension $d\geq 2$ the distance induced by the $H^s$ metric vanishes if and only if $s<1$, as recently proved by Jerrard and Maor \cite{jerrard2018vanishing}. Local well-posedness and existence of $H(\mr{div})$ geodesics is guaranteed if $h$ is close to the identity $\mr{Id}$ in a sufficiently strong topology, due to Ebin and Marsden \cite{ebin1970groups}. 

In one dimension, the Lagrangian in \eqref{eq:ablag} is equivalent to the square of the $H^1$ norm, and if we replace $M$ by the real line, the Euler-Lagrange equations coincide with the Camassa-Holm (CH) equation. For the choice of coefficients in \eqref{eq:ablag} the CH equation reads as follows:
\begin{equation} \label{eq:ch1d}
\partial_{t} u - \frac{1}{4} \partial_{txx} u + 3  u\partial_x u - \frac{1}{2} \partial_{xx} u \partial_x u - \frac{1}{4}  u \partial_{xxx}u = 0\,.
\end{equation}
This equation was shown to model shallow water waves \cite{camassa1993integrable}, and in this context, the Lagrangian in \eqref{eq:ablag} yields the appropriate generalization to a higher dimensional domain in $\mathbb{R}^d$ \cite{kruse2001two}. The CH equation has been intensively studied in literature, mostly because it represents an example of bi-Hamiltonian and integrable equation \cite{fuchssteiner1981symplectic}, and its smooth solution blow up in finite time in a process known as wave breaking. Furthermore, even weak solutions cannot be defined globally \cite{molinet2004well}, their blow up being related to the emergence of non-injective Lagrangian maps.

Problem \ref{prob:det} was recentently reinterpreted as an $L^2$ geodesic problem on the cone over $M$ \cite{gallouet2017camassa}, establishing therefore a link with the incompressible Euler equations which share a similar structure, as shown in the pioneering work of Arnold \cite{arnold1966geometrie}. 

\subsection{Contributions}
In this paper we construct a relaxation of problem \ref{prob:det} inspired by Brenier's relaxation of the incompressible Euler equation.  We call the minimizers of such a relaxation \emph{generalized solutions}. This approach allows us to obtain several results on the $H(\mr{div})$ geodesic problem. In particular, we show that:
\begin{itemize}
\item if $M$ is convex, smooth $H(\mr{div})$ geodesics are globally length-minimizing for short times and in any dimension (theorem \ref{th:correspondence}). This result generalizes the one in  \cite{gallouet2017camassa}, which was only valid on the circle of unit radius $S^1_1$ and it was local otherwise;
\item on the torus $S^1\times S^1$, there exists $h\in\mr{Diff}(S^1\times S^1)$ such that the infimum of the action in problem \ref{prob:det} cannot be attained by any smooth flow (theorem \ref{th:rott2}); on the contrary, for the same $h$ there exists a generalized solution that arises as the limit of a minimizing sequence of smooth flows (theorem \ref{th:approxt2}); 
\item there exists a unique pressure field in the sense of distribution associated with generalized solutions (theorem \ref{th:existpres}). To the best of the authors' knowledge, the pressure field we consider is a variable that has not been studied before in the literature of the CH equation or the $H(\mr{div})$ geodesic problem (see remark \ref{rem:pres}). It appears however as a natural variable in the generalized setting.
\end{itemize}


\subsection{The $a$-$b$-$c$ metric} 
The Lagrangian in \eqref{eq:ablag}  is a particular instance of a class of right-invariant Lagrangians on the diffeomorphism group of $M$ considered in \cite{khesin2013geometry}, which for $d=3$ can be written  as
\begin{equation}\label{eq:ablagg}
l(u) = a \int_M  | u |^2\, \ed \rho_0 + b \int_M  |\mr{div}\,u|^2\,\ed \rho_0 +  c \int_M  |\mr{curl} \,u|^2\,\ed \rho_0\,,
\end{equation}
where $a,b,c$ are positive constants. Such Lagrangians give rise to several important nonlinear evolution equations, including the EPDiff equation for the $H^1$ Sobolev norm of vector fields and the Euler-$\alpha$ model \cite{holm1998euler,holm1998eulerpoi}, both of which have also been regarded as possible multi-dimensional versions of the CH equation. 

\subsection{The $\dot{H}^1$ metric and the Hunter-Saxton equation}
The Hunter-Saxton equation \cite{hunter1991dynamics, khesin2013geometry} corresponds to choosing $a=c=0$ in \eqref{eq:ablagg}, in which case the metric is denoted by $\dot{H}^1$, and in one dimension. Lenells provided a simple description of the solutions to this equation as geodesic flows on the infinite-dimensional sphere of $L^2$ functions with constant norm \cite{lenells2007hunter}. This was established by constructing an explicit isometry between the group of orientation-preserving diffeomorphism of the circle $S^1$ (modulo rotations) and a subset of the sphere, given by the map 
\begin{equation}\label{eq:isom}
 \varphi \mapsto \sqrt{\partial_x \varphi}\,.
\end{equation}
This geometric point of view was particularly fruitful and led to a number of important results, namely a bound on the diameter of the diffeomeorphism group endowed with the $\dot{H}^1$ metric; that its curvature is positive and constant; that geodesics are gobally length-minimizing. Lenell's interpretation still holds when the domain is a higher dimensional manifold, as showed in \cite{khesin2013geometry}, which allowed the authors to prove complete integrability of the geodesic equations and that all solutions blow up in finite time. 
The simplifications that arise for the Hunter-Saxton equation are related to the fact that the $\dot{H}^1$ descends to a non-degenerate metric on the space of densities via the isometry \eqref{eq:isom}. This however does not apply to the full $H^1$ metric or the $H(\mr{div})$ metric because of the presence of the transport term, given by the $L^2$ norm of the velocity.
   
\subsection{The $L^2$ metric and the incompressible Euler equations}
The $L^2$ metric was used by Arnold \cite{arnold1966geometrie} to interpret the solutions to the incompressible Euler equations as geodesic curves on the group of volume-perserving diffeomorphisms $\mr{Diff}_{\rho_0}(M)$. As for the $H(\mr{div})$ problem, the existence of length-minimizing geodesics is guaranteed a priori only in a sufficiently strong topology \cite{ebin1970groups}. In fact, Shnirelman proved that the infimum is generally not attained when $d\geq 3$ and that even when $d=2$ there exist final configurations $h$ which cannot be connected to the identity map with finite action \cite{shnirelman1994generalized}. This motivated Brenier to adopt an extrinsic approach, viewing 
\begin{equation}
\mr{Diff}_{\rho_0}(M) \subset \{ \varphi \in L^2(M;M)\,;\,\varphi_\# \rho_0 = \rho_0\}
\end{equation}
and reinterpreting incompressible flows 
as probability measures $\bs{\mu}$  on $\Omega(M)$, the space of continuous curves on the domain $\mr{x}: t\in[0,T] \rightarrow x_t \in M$, satisfying 
\begin{equation}\label{eq:isoteulergen}
(e_t)_\# \bs{\mu} = \rho_0\,,
\end{equation}
where $e_t: \Omega(M) \rightarrow M$ is the evaluation map at time $t$ defined by $e_t(\mr{x})=x_t$, and $\rho_0$ is normalized so that $\rho_0(M) =1$. In this interpretation, the marginals $(e_0,e_t)_\# \bs{\mu}$ are probability measures on $M\times M$ and describe how particles move and spread their mass across the domain. Classical deterministic solutions, i.e.\ curves of volume preserving diffeomorphisms $t \mapsto \varphi_t$, correspond to the case where the marginals $(e_0,e_t)_\# \bs{\mu}$ are concentrated on the graph of $\varphi_t$. Then, equation \eqref{eq:isoteulergen} is equivalent to the incompressibility constraint $\varphi_\# \rho_0 = \rho_0$. Note also that formulating the incompressibility via the push-forward of $\varphi$, we allow changes in orientation. The minimization problem in terms of generalized flows consists in minimizing the action
\begin{equation}
\int_{\Omega(M)} \int_0^T \frac{1}{2} |\dot{x}_t|^2 \,\ed t \,\ed \bs{\mu}(\mr{x})\,
\end{equation}
among generalized incompressible flows, with the constraint $(e_0,e_T)_\# \bs{\mu} = (\mr{Id},h)_\# \rho_0$.  
Brenier proved that smooth solutions correspond to the unique minimizers of the generalized problem for sufficiently small times, and are therefore globally length-minimizing \cite{brenier1989least}. On the other hand, for any coupling there exists a unique pressure, defined as a distribution (but that can actually be defined as a function \cite{ambrosio2008regularity}), associated with generalized solutions.



\subsection{The $H(\mr{div})$ metric as an $L^2$ cone metric}
The link between the incompressible Euler equation and the $H(\mr{div})$ geodesic problem was established in \cite{gallouet2017camassa}, where it was proven that  problem \ref{prob:det} can be reformulated as a geodesic problem for the $L^2$ cone metric (see equation \eqref{eq:l2cone}; see also section \ref{sec:conemetric} for the cone metric structure) on a subgroup of the diffeomorphism group of $M\times \mathbb{R}_{>0}$. More precisely, Lagrangian flows are represented by time dependent automorphisms on $M\times \mathbb{R}_{>0}$, i.e.\
maps in the form
\begin{equation}
  (x,r) \in M\times \mathbb{R}_{>0} \mapsto (\varphi(x), \lambda(x) r) \in M\times \mathbb{R}_{>0}\,,
\end{equation}
where $\varphi: M \rightarrow M$ and $\lambda: M \rightarrow \mathbb{R}_{>0}$, satisfying
\begin{equation}\label{eq:isotch}
\varphi_\# (\lambda^2 \rho_0) = \rho_0\,.
\end{equation}
This condition relates $\varphi$ and $\lambda$ by requiring $\lambda = \sqrt{|\mr{Jac}(\varphi)|}$. 
Importantly, in this picture we cannot capture the blow up of solutions as induced by peakon collisions, as in this case the Jacobian would locally vanish. In addition, the metric space $M\times \mathbb{R}_{>0}$ equipped with the cone metric is not complete. We are then led to work with the cone $\cone =  (M \times  \mathbb{R}_{\geq 0})/ (M \times  \{0\})$, which allows us to represent solutions with vanishing Jacobian by paths on the cone reaching the apex. 

Interestingly, the decoupling between the Lagrangian flow map and its Jacobian has also been used in \cite{lee2017global} to construct global weak solutions of the CH equation. However, in their case, one continues solutions after the blowup by allowing the square root of the Jacobian to become negative, which does not occur in the formulation described above.  

\subsection{Generalized compressible flows and unbalanced optimal transport}
By analogy with the incompressible Euler case,  we reformulate the $H(\mr{div})$ geodesic problem using generalized flows interpreted as probability measures $\bs{\mu}$ on the space $\Omega(\cone)$ of continuous paths on the cone $\mr{z}:t\in[0,T] \rightarrow z_t = [x_t,r_t] \in \cone$. Our relaxed formulation consists in minimizing the action
\begin{equation}
\int_{\Omega(\cone)} \int_0^T |\dot{z}_t|^2_{g_{\cone}} \,\ed t \,\ed \bs{\mu}(\mr{z})\,
\end{equation}
among generalized flows satisfying appropriate constraints enforcing a generalized version of \eqref{eq:isotch} and the coupling between initial and final times. Choosing the correct form for such constraints is not trivial. It is the first contribution of the paper to define a formulation that allows to prove existence of minimizers while retaining uniqueness for short time in the smooth setting.


It should be noted that the cone construction has been developed and used extensively in \cite{liero2018optimal,chizat2018unbalanced} in order to characterize the metric side of the Wasserstein-Fisher-Rao (WFR) distance (which is also called Hellinger-Kantorovich distance) on the space of positive Radon measures. In fact, as noted in \cite{gallouet2017camassa} this has the same relation to the CH equation as the Wasserstein $L^2$ distance does to the incompressible Euler equations. In the geodesic problem associated the WFR distance a relation similar to \eqref{eq:isotch} is used to prescribe the initial and final density. The resulting problem coincides with the so-called optimal entropy-transport problem, a widespread form of unbalanced optimal transport based on the Kullback-Leibler divergence \cite{chizat2018unbalanced,chizat2018interpolating,liero2018optimal}.

Taking advantage of the optimal transport point of view, we propose a numerical scheme based on multi-marginal optimal transport and entropic regularization \cite{cuturi2013sinkhorn,benamou2015iterative,benamou2017generalized} to simulate the solutions of problem \ref{prob:det}.

%
%
\subsection{Structure of the paper}
In section \ref{sec:notation}, we introduce the notations and the needed background. In section \ref{sec:variational}, we recall the $L^2$ variational formulation of the $H(\mr{div})$ geodesic problem.
\par
In section \ref{sec:generalized} we introduce our relaxation, for which we prove existence of solutions as generalized compressible flows.  
We also show that our generalized solutions can be decomposed into two parts, one of which involves directly the cone singularity. When this latter is not trivial, it implies the appearance and disappearance of mass in the domain; we refer to such minimizers as singular solutions.
\par
In section \ref{sec:existence} we prove that for any boundary conditions, there always exists a unique pressure field defined as a distribution on $(0,T)\times M$ associated with any given generalized solution. 
\par
In section \ref{sec:correspondence} we prove that smooth solutions of the $H(\mr{div})$ geodesic equations are the unique minimizers of our generalized model for sufficiently short times. This proves that such solutions are also globally length-minimizing on $\mr{Diff}(M)$.
\par
In section \ref{sec:examples} we show that for $d \geq 2$, singular solutions emerge naturally from the continuous formulation for appropriate (smooth) boundary conditions. This proves that the infimum of the action in problem \ref{prob:det} may not be attained. 
We construct approximations for such minimizers using 
a particular form of peakon collision which arises from the Hunter-Saxton equation.
\par
Finally, in section \ref{sec:discrete} we construct a numerical scheme based on entropic regularization and Sinkhorn algorithm to compute generalized $H(\mr{div})$ geodesics. 

}

\section{Notation and preliminaries}\label{sec:notation}

In this section, we describe the notation and some basic results used throughout the paper. Because of the similarities between our setting and the one of \cite{liero2018optimal}, we will adopt a similar notation for the cone construction and the measure theory objects we will employ. 

\subsection{Function spaces} 

Given two metric spaces $X$ and $Y$, we denote by $C^0(X;Y)$ the space of continuous functions $f:X\rightarrow Y$, by $C^0(X)$ the space of real-valued continuous functions $f:X\rightarrow \mathbb{R}$\corr{, and by $C^0_b(X)$ the subset of bounded functions $f\in C^0(X)$}. If $X$ is compact $C^0(X)$  is a Banach space with respect to the sup norm $\| \cdot\|_{C^0}$. The set of Lipschitz continuous function on $X$ is denoted by $C^{0,1}(X)$ and the associated seminorm and norm are given respectively by\corr{
\begin{equation}
|f|_{C^{0,1}} \coloneqq \sup_{x,y\in{X}, x\neq y} \frac{|f(x)-f(y)|}{ d_X(x,y)}\,, \quad \|f\|_{C^{0,1}} \coloneqq \|f\|_{C^0} + |f|_{C^{0,1}} \,, 
\end{equation}}
where $d_X$ denotes the distance function on $X$.

If $X$ {is a subset of $\mathbb{R}^d$}, we use standard notation for Sobolev spaces on $X$. In particular, $H(\mr{div}; X)$ or simply $H(\mr{div})$ denotes the space of $L^2$ vector fields $f:X \rightarrow \mathbb{R}^d$ whose divergence $\mr{div}(f)$ is in $L^2$, with squared norm given by $\|f\|^2_{L^2} +\|\mr{div}\, f\|^2_{L^2}$ (which is equivalent to \eqref{eq:ablag}). Moreover, we denote by $\mr{Diff}(X)$ the group of smooth diffeomorphisms of $X$.

\subsection{The cone and metric structures}\label{sec:conemetric} 
\corr{Throughout the paper, $M$ will denote the closure of an open bounded set in $\mathbb{R}^d$ with  Lipschitz  boundary}. Occasionally, we will also consider the case $M=S^1_R \coloneqq \mathbb{R}/{2\pi R}\mathbb{Z}$ the circle of radius $R$, or $M=T^2_{R_1,R_2} \coloneqq S^1_{R_1}\times S^1_{R_2}$ the torus with radii $R_1,R_2>0$.  We will denote by $g$ the Euclidean metric tensor on $M$ 
and with $| \cdot|$ the Euclidean norm. We denote by $\cone \coloneqq (M \times  \mathbb{R}_{\geq 0})/ (M \times  \{0\})$ the cone over $M$. A point on the cone is an equivalence class ${p}=[x,r]$, where the equivalence relation is given by 
\begin{equation}
(x_1,r_1) \sim (x_2,r_2) \Leftrightarrow (x_1,r_1) = (x_2, r_2)~\text{or}~r_1 = r_2 = 0\,. 
\end{equation} 
The distinguished point of the cone $[x,0]$ is the apex of $\cone$ and it is denoted by ${o}$.
Every point on the cone different from the apex can be identified with a couple $(x,r)$ where $x\in M$ and $r \in \mathbb{R}_{>0}$. Moreover, we fix a point $\bar{x}\in M$ and we introduce the projections $\pi_x : \cone \rightarrow M$ and $\pi_r : \cone \rightarrow \mathbb{R}_{\geq 0}$ defined by
\begin{equation}
\pi_x ([x,r]) = \left\{ 
\begin{array}{ll}
x & \text{if}~r>0\,, \\
\bar{x} & \text{if}~r=0\,, 
\end{array}
\right.
 \qquad
\pi_r ([x,r]) = r\,.
\end{equation} 
We endow the cone with the metric tensor $g_{\cone} = r^2 g+ \mr{d} r^2$, defined on $M\times \mathbb{R}_{>0}$. We denote the associated norm by $|\cdot|_{g_\cone}$. All differential operators, e.g., $\nabla$, $\mr{div}$ and so on, are computed with respect to the Euclidean metric on $M$; we will use the superscript $g_\cone$ to indicate when they are computed with respect to the cone metric. 
The distance on the cone $d_{{\cone}}:\cone \times \cone \rightarrow \mathbb{R}_{\geq 0}$ is given by
\begin{equation}
d_{\cone}([x_1,r_1],[x_2,r_2])^2 = r_1^2 + r_2^2 - 2 r_1 r_2 \cos(\min(|x_1-x_2|, \pi))\,
\end{equation}
(see, for example, definition 3.6.16 in \cite{burago2001course}). The closed subset of the cone composed of points below a given radius $R>0$ is denoted by $\cone_R$, or more precisely
\begin{equation}
\cone_R \coloneqq \{ [x,r] \in \cone\,;\,r\leq R\}\,.
\end{equation} 

Given an interval $I\subset \mathbb{R}$, we denote by $C^0(I;\cone)$ and $AC(I;\cone)$ the spaces of, respectively, continuous and absolutely continuous curves $\mr{z}: t\in I \rightarrow {z}_t\in \cone$. We will generally use the notation
\begin{equation}
\mr{x}: t\in I \rightarrow x_t= \pi_x(z_t)\in M\,,\qquad \mr{r}: t\in I \rightarrow r_t= \pi_r({z}_t)\in [0,+\infty)\,,
\end{equation}
so that $\mr{z}=[\mr{x},\mr{r}]$ and ${z}_t=[{x}_t,{r}_t]$. Note that if $\mr{z}$ is  continuous (resp.\ absolutely continuous), then so is the path $\mr{r}$ but not $\mr{x}$. However,  $\mr{x}$ is continuous (resp.\ locally absolutely continuous) when restricted to the open set $\{ t\in I; r_t>0\}$. Then, if we define $\dot{\mr{z}}:  t\in  I \rightarrow \dot{z}_t\in\mathbb{R}^{d+1}$  by
\begin{equation}
\dot{z}_t = \left \{ 
\begin{array}{ll}
(\dot{x}_t, \dot{r}_t) & \text{if } r_t>0 \text{ and the derivatives exist},\\
(0,0)        & \text{otherwise},
\end{array}
\right.
\end{equation}
we have that $|\dot{z}_t |_{g_\cone}$ coincides for a.e.\ $t\in I$ with the metric derivative of $\mr{z}$ with respect to the distance $d_{\cone}$ \cite{liero2018optimal}. 
We denote by $AC^p(I;\cone)$  the space of absolutely continuous curves such that $ |\dot{\mr{z}}|_{g_\cone}\in L^p(I)$. Then, the following variational formula for the distance function holds
\begin{equation}
d_\cone({p}, {q})^2 = \inf \left \{ \int_0^1 | \dot{{z}_t} |^2_{g_\cone}\, \ed t \,; \, \mr{z} \in AC^2([0,1];\cone)\, ,\, {z}_0 = {p}\,, \, {z}_1 = {q}  \right\}.
\end{equation}

We will extensively use the class of homogeneous functions on the cone defined as follows. A function $f:\cone^n \rightarrow \mathbb{R}$ is $p$-homogeneous (in the radial direction) if for any constant $\lambda>0$ and for all $n$-tuples $([x_1,r_1], \ldots, [x_n, r_n] ) \in \cone^n$,
\begin{equation}
f([x_1,\lambda r_1], \ldots, [x_n,\lambda r_n] ) = \lambda^p f([x_1,r_1], \ldots, [x_n, r_n] )\,.
\end{equation}
In particular, a $p$-homogeneous function $f:\cone \rightarrow \mathbb{R}$ satisfies $f([x,\lambda r]) = \lambda^p f([x,r])$. Similarly, a functional $\sigma: C^0(I;\cone) \rightarrow \mathbb{R}$ is $p$-homogeneous if for any constant $\lambda >0$ and for any path $\mr{z}\in C^0(I;\cone)$,
\begin{equation}
\sigma( t\mapsto [x_t,\lambda r_t] ) = \lambda^p \sigma (\mr{z})\,,
\end{equation}
where $\mr{z} : t \in I \rightarrow [x_t,r_t] \in \cone$. 

\subsection{Measure theoretic background}
Let $X$ be a Polish space, i.e.\ a complete and separable metric space. We denote by $\mc{M}(X)$ the set of non-negative and finite Borel measures on $X$. The set of probability measures on $X$ is denoted by $\mc{P}(X)$. Let $Y$ be another Polish space and $F:X \rightarrow Y$ a Borel map. Given a measure $\mu\in \mc M(X)$ we denote by $F_\#\mu \in \mc{M}(Y)$ the push-forward measure defined by $(F_\#\mu)(A) \coloneqq \mu (F^{-1}(A))$ for any Borel set $A \subset Y$. Given a Borel set $B \subset X$ we let $\mu\mres B$ the restriction of $\mu$ to $B$ defined by $\mu\mres B (C) \coloneqq \mu(B \cap C)$ for any Borel set $C\subseteq X$. Note that we will generally use bold symbols to denote measures on product spaces, e.g., $\bs{\mu} \in \mc{M}(X \times \ldots \times X)$. 

We endow $\mc{P}(X)$ with the topology induced by narrow convergence, which is the convergence in duality with the space of real-valued continuous bounded functions $C^0_b(X)$. In other words, a sequence $\mu_n\in\mc{P}(X)$, $n\in \mathbb{N}$, is said to converge narrowly to $\mu \in \mc{P}(X)$ if for any $f\in C^0_b(X)$
\begin{equation}\label{eq:narconv}
\lim_{n\to+\infty} \int_X f \, \ed \mu_n = \int_X f \, \ed \mu\,.
\end{equation}
In practice, however, to check for narrow convergence it is sufficient to verify equation \eqref{eq:narconv} for all bounded Lipschitz continuous functions. With such a topology, $\mc{P}(X)$ can be identified with a subset of $[C^0_b(X)]^*$ with the weak-* topology (see Remark 5.1.2 in \cite{ambrosio2008gradient}). In addition, given a lower semi-continuous function $f:X \rightarrow \mathbb{R} \cup \{+\infty\}$, \corr{bounded from below,} the functional $\mc{F}:\mc{P}(X) \rightarrow  \mathbb{R} \cup \{+\infty\}$ defined by
\begin{equation}
\mc{F}(\mu) \coloneqq \int_X f \,\ed \mu\,
\end{equation}
is also lower-semicontinuous (see Lemma 1.6 in \cite{santambrogio2015optimal}) .

As usual in this setting, we will use Prokhorov's theorem for a characterization of compact subsets of $\mc{P}(X)$ endowed with the narrow topology.

\begin{theorem}[Prokhorov's theorem]\label{th:prok}
A set $\mc{K} \subset \mc{P}(X)$ is relatively sequentially compact in $ \mc{P}(X)$ if and only if it is tight, i.e.\ for any $\epsilon >0$ there exists a compact set $K_\epsilon\subset X$ such that $\mu(X\setminus K_\epsilon) < \epsilon$ for any $\mu \in \mc{K}$.
\end{theorem}

We also need a criterion to pass to the limit when computing integrals of unbounded functions: for this  will use the concept of uniform integrability. Given a set $\mc{K} \subset \mc{P}(X)$, we say that a Borel function $f:X \rightarrow \mathbb{R}_{\geq 0} \cup \{+\infty\}$ is uniformly integrable with respect to $\mc{K}$ if for any $\epsilon >0$ there exists a $k>0$ such that, for any $\mu \in \mc{K}$,
\begin{equation}
\int_{f(x)>k} f(x)\, \ed \mu(x) < \epsilon \,.
\end{equation}
\begin{lemma}[Lemma 5.1.7 in \cite{ambrosio2008gradient}] \label{lem:unifint}
Let $\{\mu_n\}_{n\in \mathbb{N}}$ be a sequence in $\mc{P}(X)$ narrowly convergent to $\mu \in \mc{P}(X)$ and let $f\in C^0(X)$. If \corr{$|f|$} is uniformly integrable with respect to the set $ \{\mu_n\}_{n\in \mathbb{N}}$ then
\begin{equation}
\lim_{n\to +\infty} \int_X f \, \ed \mu_n = \int_X f \, \ed \mu\,.
\end{equation}
\end{lemma}

For a fixed $T>0$, we will denote by $\Omega(X) \coloneqq C^0([0,T];X)$ the space of continuous paths on $X$. This is a Polish space so that we can use the tools introduced in this section also for probability measures $\bs{\mu} \in \mc{P}(\Omega(X))$. We call such probability measures \emph{generalized flows} or also \emph{dynamic plans}. When $X=\mc{C}$, where $\cone$ is the cone over $M\subset\mathbb{R}^d$, we will often use $\Omega$ to denote $\Omega(\mc{C})$.

Since we will work with homogeneous functions on the cone, we also introduce the space of probability measures $\mc{P}_p(X)$, for $p>0$, defined by
\begin{equation}
\mc{P}_p(X) \coloneqq \left \{ \mu \in \mc{P}(X) \,;\,  \int_X d_X(x,\bar{x})^p \,\ed \mu(x) < + \infty  \text{ for some }\bar{x}\in X \right \}\,.
\end{equation} 
Then, if $\bs{\mu} \in \mc{P}_p(\cone^n)$ it is easy to verify that any locally-bounded $p$-homogeneous function on $\cone^n$ is $\bs{\mu}$-integrable.

Finally, we will denote by $\rho_0$ the Lebesgue measure on $M$  normalized so that $\rho_0(M) = 1$.

\section{The variational formulation on the cone} \label{sec:variational}
In this section we describe the geometric structure underlying problem \ref{prob:det} using the group of automorphisms of the cone. Such a formulation was introduced in \cite{gallouet2017camassa} and it was used to interpret the CH equation as an incompressible Euler equations on the cone. 
In this section we will only focus on smooth solutions, but we will later use the variational interpretation presented here to guide the construction of generalized $H(\mr{div})$ geodesics. We will keep the discussion formal at this stage and we will use some standard geometric tools and notation commonly adopted in similar contexts.

For any $\varphi \in \mr{Diff}(M)$ and $\lambda \in C^\infty(M;\mb{R}_{>0})$, we let $(\varphi,\lambda):\cone \rightarrow \cone$ be the map defined by 
$(\varphi,\lambda)([x,r]) = [\varphi(x),\lambda(x)r]$. The automorphism group $\mr{Aut}(\cone)$ is the collection of such maps, i.e.\
\begin{equation}
\mr{Aut}(\cone) = \{ (\varphi,\lambda):\cone \rightarrow \cone ; \, \varphi \in \mr{Diff}(M),\, \lambda \in C^\infty(M;\mb{R}_{>0})  \}\,.
\end{equation}
The group composition law is given by
\begin{equation}
(\varphi,\lambda)\cdot (\psi,\mu) = (\varphi\circ\psi, (\lambda \circ \psi) \mu)\,,
\end{equation}
the identity element is $(\mr{Id},1)$, where $\mr{Id}$ is the identity map on $M$, and the inverse is given by $(\varphi,\lambda)^{-1} = (\varphi^{-1}, \lambda^{-1} \circ \varphi^{-1})$.
The tangent space of $\mr{Aut}(\cone)$ at $(\varphi,\lambda)$ is denoted by $T_{(\varphi,\lambda)}\mr{Aut}(\cone)$ and it
 can be identified with the space of vector fields $C^\infty(M;\mathbb{R}^{d+1})$. The collection all the tangent spaces is the tangent bundle $T\mr{Aut}(\cone)$.
We endow $T\mr{Aut}(\cone)$ with the $L^2(M;T\cone)$ metric inherited from $g_{\cone}$. This is defined as follows: given $(\dot{\varphi},\dot{\lambda})\in  T_{(\varphi,\lambda)} \mr{Aut}(\cone)$, 
\begin{equation}\label{eq:l2cone}
\|(\dot{\varphi},\dot{\lambda})\|^2_{L^2(M;T\cone)} \coloneqq \int_M (\lambda^2  |\dot{\varphi}|^2 + \dot{\lambda}^2 )\,\ed\rho_0 \, ,
\end{equation}
where $|\cdot |$ is the Euclidean norm and $\rho_0$ is the Lebesgue measure on $M$ normalized so that $\rho_0(M)=1$.

In \cite{gallouet2017camassa} the authors found that the $H(\mr{div})$ geodesic equations on $M$ coincide with the geodesic equation on the subgroup $\mr{Aut}_{\rho_0}(\cone) \subset \mr{Aut}(\cone)$ defined as follows:
\begin{equation}\label{eq:autrho}
\mr{Aut}_{\rho_0}(\cone) \coloneqq \{ (\varphi,\lambda) \in \mr{Aut}(\cone) \,; \varphi_\#(\lambda^2 \rho_0) = \rho_0\}\,.
\end{equation}
In other words, the group $\mr{Aut}_{\rho_0}(\cone)$ can be regarded as the configuration space for the $H(\mr{div})$ geodesic problem in the same way as the $\mr{Diff}_{\rho_0}(M)$ is the configuration space for the incompressible Euler equations, with
\begin{equation}
\mr{Diff}_{\rho_0}(M) \coloneqq \{ \varphi \in \mr{Diff}(M) \,; \varphi_\# \rho_0 = \rho_0\}\,.
\end{equation}

In order to see this, we first observe that the $L^2(M;T\cone)$ metric is right invariant when restricted to $\mr{Aut}_{\rho_0}(\cone)$.
In particular, for any $(\psi,\vartheta)\in  \mr{Aut}_{\rho_0}(\cone)$, consider the right translation map $R_{(\psi,\vartheta)}: \mr{Aut}_{\rho_0}(\cone) \rightarrow \mr{Aut}_{\rho_0}(\cone)$ defined by $R_{(\psi,\vartheta)}(\varphi,\lambda) = (\varphi,\lambda)\cdot (\psi,\vartheta)$. Its tangent map at $(\varphi,\lambda)$ is given by
\begin{equation}
TR_{(\psi,\mu)} (\dot{\varphi},\dot{\lambda}) = (\dot{\varphi}\circ \psi, (\dot{\lambda}\circ \psi)\vartheta).
\end{equation} 
Then, it is easy to check that $\|TR_{(\psi,\vartheta)}(\dot{\varphi},\dot{\lambda})\|^2_{L^2(M;T\cone)} = \|(\dot{\varphi},\dot{\lambda})\|^2_{L^2(M;T\cone)}$.
Geodesics on $\mr{Aut}_{\rho_0}(\cone)$ correspond to stationary paths on $T\mr{Aut}_{\rho_0}(\cone)$ for the action functional 
\begin{equation}\label{eq:actiondet}
\int_0^T L((\varphi,\lambda),(\dot{\varphi},\dot{\lambda}))\, \mr{d} t\,
\end{equation}
for a given $T>0$, where the Lagrangian $L((\varphi,\lambda),(\dot{\varphi},\dot{\lambda})) = \|(\dot{\varphi},\dot{\lambda})\|^2_{L^2(M;T\cone)}$.
Define the Eulerian velocities $(u,\alpha) \in T_{(\mr{Id},1)} \mr{Aut}(\cone)$ by
\begin{equation}
(u,\alpha) = TR_{(\varphi,\lambda)^{-1}} (\dot{\varphi},\dot{\lambda}) = (\dot{\varphi} \circ \varphi^{-1}, (\dot{\lambda}{\lambda}^{-1}) \circ \varphi^{-1})\,.
\end{equation}
In terms of these variables the constraint $\varphi_\#(\lambda^2 \rho_0) = \rho_0$ becomes $2 \alpha = \mr{div}\, u$, since for any $f\in C^\infty(M)$, 
\begin{equation}
0= \frac{\mr{d}}{\mr{d} t} \int_M f  \,\ed\varphi_\#(\lambda^2 \rho_0)  = \int_M (-\mr{div}\, u + 2 \alpha)f \,\ed\rho_0\,. \\
\end{equation}
Moreover, by right invariance,
\begin{equation}
L((\varphi,\lambda),(\dot{\varphi},\dot{\lambda})) = L((\mr{Id},1),(u,\alpha)) = \int_M |u|^2 + \frac{1}{4} |\mr{div}\, u|^2 \,\ed\rho_0\,,
\end{equation}
which is the $H(\mr{div})$ norm. Note that the coefficient $1/4$ is directly related to the choice of $g_\cone$ as cone metric. Using different coefficients in $g_\cone$ we can obtain the general form of the Lagrangian in equation \eqref{eq:ablagg} with $c=0$.
Introducing $P$ as the Lagrange multiplier for the constraint $\varphi_\#(\lambda^2 \rho_0) = \rho_0$, the Euler-Lagrange equations associated with $L$ read as follows
\begin{equation}\label{eq:geodesic}
\left \{
\begin{array}{l}
\lambda \ddot{\varphi} + 2  \dot{\lambda} \dot{\varphi} + \frac{1}{2}\lambda \nabla P \circ \varphi = 0\,,\\
\ddot{\lambda} - \lambda  |\dot{\varphi}|^2 + \lambda P \circ  \varphi = 0\,,
\end{array}
\right.
\end{equation}
which can be expressed in terms of $(u,\alpha)$ by composing both equation with $\varphi^{-1}$, yielding
\begin{equation}\label{eq:geodesiceulerian}
\left \{
\begin{array}{l}
\dot{u} + \nabla_u u +  2 u \alpha = -\frac{1}{2}   \nabla P \,,\\
\dot{\alpha} + u \cdot \nabla \alpha + \alpha^2 - |u|^2 = - P \,.
\end{array}
\right.
\end{equation}
In one dimension, using the relation $\alpha = \mr{div} \, u/2$, this finally gives the CH equation for $u$, i.e.\ equation \eqref{eq:ch1d}.

\begin{remark}\label{rem:pres} 
Note that in the literature for the CH equation the ``pressure field" is sometimes defined in a different way so that, when $M$ is one-dimensional, the first equation in \eqref{eq:geodesiceulerian} can be written as
\begin{equation}
\partial_t{u} +u \partial_x u = - \partial_x p \,,\\
\end{equation}  
\corr{for an appropriate function $p = (\mr{Id} - \frac{1}{4}\partial_{xx})^{-1}(u^2 + \frac{1}{8} (\partial_x u)^2)$ (see, e.g., \cite{holm2003wave}). Throughout the paper we will instead intend by pressure the Lagrange multiplier $P$ considered above, which is related to $p$ by
\begin{equation}
P = 2p-u^2\,.
\end{equation}  
In section \ref{sec:existence} we will prove that the pressure $P$ is uniquely defined for minimizers of the $H(\mr{div})$ geodesic problem. On the other hand, we cannot prove the same result for $p$, since the flow $\varphi$ and as a consequence the velocity field $u$ may not be well-defined for generalized solutions (see the explicit examples of generalized solutions in section \ref{sec:examples}).  } 
\end{remark}

\section{The generalized $H(\mr{div})$ geodesic formulation}\label{sec:generalized}
\corr{In section \ref{sec:variational} we recalled the interpretation of the $H(\mr{div})$ metric as an $L^2$ metric on $\mr{Aut}_{\rho_0}(\cone)$, which is defined in \eqref{eq:autrho}}. Then, we can reformulate problem \ref{prob:det} as follows:
\begin{problem}[$H(\mr{div})$ geodesic problem on the cone]\label{prob:detunb}
Given a diffemorphism $\fmap\in \mr{Diff}(M)$, find a smooth curve $t\in[0,T] \mapsto (\varphi_t,\lambda_t) \in \mr{Aut}_{\rho_0}(\cone)$ satisfying
\begin{equation}\label{eq:couplingconstrdet}
(\varphi_0,\lambda_0) = (\mr{Id},1) \,, \qquad (\varphi_T,\lambda_T) = (\fmap,\sqrt{|\mr{Jac}(\fmap)|})\,,
\end{equation}
and minimizing the action in equation \eqref{eq:actiondet}.
\end{problem} 

\corr{As already pointed out in \cite{gallouet2017camassa}, there is a remarkable analogy between this problem and Arnold's geometric interpretation of the incompressible Euler equations \cite{arnold1966geometrie}. This suggests that adapting to this problem Brenier's concept of generalized flow could be a successful strategy to characterize its minimizers. In this section we follow this path and in particular we formulate the generalized $H(\mr{div})$ geodesic problem and prove existence of solutions.}

By generalized flow or dynamic plan we mean a probability measure on the space of continuous paths of the cone $\bs{\mu} \in \mc{P}(\Omega)$. 
This is a generalization for curves on the automorphism group since for any smooth curve $(\varphi,\lambda):t\in[0,T] \rightarrow (\varphi_t,\lambda_t) \in \mr{Aut}_{\rho_0}( \cone)$, we can associate the generalized flow $\bs{\mu}$ defined by
\begin{equation}\label{eq:smoothmu}
\bs{\mu} = (\varphi,\lambda)_\# \rho_0\,,
\end{equation}
\corr{where we recall that the Lebesgue measure $\rho_0$ is normalized in such a way that $\rho_0(M)=1$.}
More explicitly, for any Borel functional ${\mc F}:\Omega \rightarrow \mathbb{R}$,
\begin{equation}\label{eq:Fphi}
\int_{\Omega} {\mc F}(\mr{z}) \,\ed\bs{\mu}(\mr{z}) = \int_{M} {\mc F}([\varphi(x),\lambda(x)]) \ed \rho_0(x) \,, 
\end{equation}
where $[\varphi(x),\lambda(x)]:t\in[0,T] \rightarrow [\varphi_t(x),\lambda_t(x)] \in \cone$.

The condition $(\varphi_t)_\# \lambda_t^2 \rho_0 = \rho_0$ is equivalent to requiring \corr{$\lambda_t = \sqrt{|\mr{Jac}(\varphi_t)|}$}. We want to generalize this condition for arbitrary $\bs{\mu} \in \mc{P}(\Omega)$.
Let $e_{t}:\Omega \rightarrow \cone$ be the evaluation map at time $t\in[0,T]$. Then, if $\bs{\mu}$ is defined as in \eqref{eq:smoothmu}, we have
\begin{equation}\label{eq:marginals}
{\mf h}^2_t (\bs{\mu}) \coloneqq  (\pi_x)_\# [r^2 (e_{t})_\# \bs{\mu}]= \rho_0\,.
\end{equation}  
In fact, for any $f\in C^0(M)$,
\begin{equation}
\begin{aligned}
\int_M f \,\ed {\mf h}^2_{t} (\bs{\mu})  & = \int_{\Omega} f(x_{t})  r_{t}^2 \,\ed \bs{ \mu} (\mr{z}) \\
& =  \int_{\Omega} f(x_{t})  r^2_{t} \,\ed (\varphi,\lambda)_\# \rho_0\\
& =  \int_{M} f \circ \varphi_{t} \lambda^2_{t}  \,\ed\rho_0\\
& =  \int_{M} f \,\ed(\varphi_{t})_\# \lambda^2_{t} \rho_0\\
& =  \int_{M} f \,\ed\rho_0\,,
\end{aligned}
\end{equation}
where for any path $\mr{z}$ and any time $t$, $x_t \coloneqq \pi_x(z_t)$ and  $r_t \coloneqq \pi_r(z_t)$.
By similar calculations, we also obtain
\begin{equation}\label{eq:coupling}
(e_0,e_T)_\#  \bs{\mu} = \bs{\gamma} \coloneqq [(\varphi_0,\lambda_0),(\varphi_T,\lambda_T)]_\# \rho_0\,.
\end{equation} 
In other words, enforcing the boundary conditions in the generalized setting boils down to constraining a certain marginal of $\bs{\mu}$ to coincide with a given \emph{coupling plan} $\bs{\gamma}$ on the cone, i.e.\ a probability measure in $\mc{P}(\mc{C} \times \mc{C})$.   

Consider now the energy functional $\mc{A}: \Omega \rightarrow \mathbb{R}_{\geq 0} \cup \{+\infty\}$ defined by
\begin{equation}\label{eq:energyfunctional}
{\mc A}(\mr{z}) \coloneqq  \left \{ 
\begin{array}{ll}
\int_0^T |\dot{z}_t|^2_{g_{\cone}}\, {\mr{d} t} & \text{if }\mr{z}\in AC^2([0,T]; \mc{C}) \,,\\
+\infty & \text{otherwise}\,.
\end{array}\right.
\end{equation}
Setting ${\mc F}(\mr{z}) = {\mc A}(\mr{z})$ in \eqref{eq:Fphi} we obtain the $H(\mr{div})$ action expressed in Lagrangian coordinates. This motivates the following definition for the generalized $H(\mr{div})$ geodesic problem.

\begin{problem}[Generalized $H(\mr{div})$ geodesic problem]\label{prob:generalizedunb}
Given a coupling plan on the cone ${\bs \gamma} \in \mc{P}_2(\mc{C}^2)$, find the dynamic plan $\bs \mu \in \mc{P}(\Omega)$ satisfying: the \emph{homogeneous coupling constraint}
\begin{equation}\label{eq:couplingconstr}
\int_\Omega f(z_0,z_T) \, \ed \bs{\mu}(\mr{z}) = \int_{\cone^2} f\, \ed \bs{\gamma}\,,
\end{equation}
for all 2-homogeneous continuous functions $f:{\cone^2} \rightarrow \mathbb{R}$;
the \emph{homogeneous marginal constraint}
\begin{equation}\label{eq:constraintsunb}
\int_\Omega \int_0^T f(t,x_t) r_t^2 \, \ed t \, \ed{\bs \mu}(\mr{z}) = \int_M \int_0^T  f(t,x) \,\ed t \, \ed\rho_0(x)  \quad \forall \, f \in C^0([0,T] \times M)\, ; 
\end{equation}
and minimizing the action 
\begin{equation}
\mc{A}(\bs{\mu}) \coloneqq \int_{\Omega} {\mc A}(\mr{z})  \,\ed\bs{\mu}(\mr{z}) \,.
\end{equation}
\end{problem} 

We remark three basic facts on this formulation:
\begin{itemize}
\item  we substituted the constraint in \eqref{eq:marginals} by its time-integrated version in equation \eqref{eq:constraintsunb} as this form will be easier to manipulate in the following. However, the two formulations are equivalent when restricting to generalized flows with finite action (see lemma \ref{lem:equivconstr});
\item  we replaced the strong coupling constraint \eqref{eq:coupling} by a weaker version, which is always implied by the former as long as $\bs{\gamma}\in\mc{P}_2(\mc{C}^2)$ and in particular when $\bs{\gamma}$ is deterministic, i.e.\ when it is induced by a diffeomorphism as in  equation \eqref{eq:coupling};
\item  we allow for general coupling plans in $\mc{P}_2(\mc{C}^2)$ so that the integral on the right-hand side of equation \eqref{eq:couplingconstr} is finite.
However, we will mostly be interested in the case where the coupling is deterministic. 
\end{itemize}

The first of the points above is made explicit in the following lemma, whose proof is postponed to the appendix.

\begin{lemma}\label{lem:equivconstr}
For any generalized flow $\bs{\mu}$ with $\mc{A}(\bs{\mu}) < + \infty$ and satisfying the homogeneous coupling constraint in equation \eqref{eq:couplingconstr}, the homogeneous marginal constraint in equation \eqref{eq:constraintsunb} is equivalent to the constraint 
\begin{equation}\label{eq:constraintsunbs}
{\mf h}^2_t (\bs{\mu}) = \rho_0\,
\end{equation}
for all $t\in[0,T]$. 
\end{lemma}

The main result of this section is contained in the following proposition, which states that generalized $H(\mr{div})$ geodesics are well-defined as solutions of problem \eqref{prob:generalizedunb}.

\begin{proposition}[Existence of minimizers] \label{prop:existenceunbounded}
Provided that there exists a dynamic plan $\bs{\mu}^*$ such that $\mc{A}(\bs{\mu}^*)<+\infty$, the minimum of the action in problem~\ref{prob:generalizedunb} is attained.
\end{proposition}

Before providing the proof of proposition \ref{prop:existenceunbounded}, we introduce a useful rescaling operation which will allow us to preserve the homogenous constraint when passing to the limit using sequences of narrowly convergent dynamic plans. Such an operation was introduced in \cite{liero2018optimal} in order to deal with the analogous problem arising from the formulation of optimal entropy-transport (i.e.\ unbalanced transport) on the cone. Adapting the notation in \cite{liero2018optimal} to our setting, we define for a functional $\theta:\Omega \rightarrow \mathbb{R}$,
\begin{equation}
\mr{prod}_\theta(\mr{z}) \coloneqq (t\in [0,T] \mapsto [x_t,r_t/\theta(\mr{z})] )\,.
\end{equation}
Then, given a dynamic plan $\bs{\mu}$, if $\theta(\mr{z})>0$ for $\bs{\mu}$-almost any path $\mr{z}$, we can define the dilation map
\begin{equation}
\mr{dil}_{\theta,2} (\bs{\mu}) \coloneqq \mr{prod}_{\theta\#} (\theta^2 \bs{\mu})\,.
\end{equation}
Since the constraints in equations \eqref{eq:couplingconstr} and \eqref{eq:constraintsunb} are 2-homogeneous in the radial coordinate $r$, they are invariant under the dilation map, meaning that if $\bs{\mu}$ satisfies \eqref{eq:couplingconstr} and \eqref{eq:constraintsunb}, also $\mr{dil}_{\theta,2} (\bs{\mu})$ does. For the same reason, we also have 
\begin{equation}
\mc{A}(\mr{dil}_{\theta,2}(\bs{\mu})) = \mc{A}(\bs{\mu}) \,.
\end{equation}
The map $\mr{dil}_{\theta,2}$ performs a \emph{rescaling} on the measure $\bs{\mu}$ in the sense specified by the following lemma.

\begin{lemma}\label{lem:rescaling}
Given a measure $\bs{\mu}\in \mc{M}(\Omega)$ and a 1-homogeneous functional $\sigma: \Omega \rightarrow \mathbb{R}$ such that $\sigma(\mr{z})>0$ for $\bs{\mu}$-almost every path $\mr{z}$, suppose that 
\begin{equation}
C\coloneqq \left(\int_\Omega (\sigma(\mr{z}) )^2 \, \ed \bs{\mu}(\mr{z}) \right)^{1/2} < +\infty \,;
\end{equation}
if $\tilde{\bs{\mu}} = \mr{dil}_{\sigma/C,2}(\bs{\mu})$ then $\tilde{\bs{\mu}}(\Omega) = 1$ and
\begin{equation}\label{eq:rescaledsupportlemma}
\tilde{\bs{\mu}}(\{\mr{z}\in\Omega\,;\, \sigma(\mr{z}) = C \} ) = 1\,.
\end{equation}
\end{lemma}
\begin{proof}
We prove this by direct calculation. Let $\theta \coloneqq \sigma/C$. By 1-homogeneity of $\sigma$, for $\bs{\mu}$-almost every path $\mr{z}$
\begin{equation}
\sigma(\mr{prod}_{\theta}(\mr{z})) = \frac{\sigma(\mr{z})}{|\theta(\mr{z})|} =C\,.
\end{equation}
Then,
\begin{equation}
\begin{aligned}
\int_{\{\mr{z}\in\Omega\,;\, \sigma(\mr{z}) = C \}}  \ed\tilde{{\bs{\mu}}}(\mr{z}) & = \int_{\{\mr{z}\in\Omega\,;\, \sigma(\mr{z}) = C \}}  \ed \mr{prod}_{\theta\#} (\theta^2 \bs{\mu})(\mr{z}) \\ 
& =  \int_{\{\mr{z}\in\Omega\,;\, \sigma( \mr{prod}_{\theta}(\mr{z})) = C \}}  \theta^2 \ed \bs{\mu}(\mr{z}) \\
& =  \frac{1}{C^2} \int_\Omega (\sigma(\mr{z}) )^2 \ed \bs{\mu}(\mr{z}) = 1 \,.
\end{aligned}
\end{equation}
By similar calculations we also have $\tilde{\bs{\mu}}(\Omega) = 1$.
\end{proof}

Besides the rescaling operator and lemma \ref{lem:rescaling}, we will also need the following result which will allow us to construct suitable minimizers of the action in problem \ref{prob:generalizedunb}. 

\begin{lemma}\label{lem:compactness}
The set of measures with uniformly bounded action $\mc{A}(\bs{\mu})\leq C$ and satisfying the homogeneous constraint in equation \eqref{eq:constraintsunb} is relatively sequentially compact for the narrow topology.
\end{lemma}
\begin{proof}
Due to Therorem~\ref{th:prok}, it is sufficient to prove that sequences of admissible measures are tight. 
For a given path $\mr{z}$ with $\mathcal{A}(\mr{z}) \leq Q$, for all $0\leq s \leq t \leq T$,
\begin{equation}\label{eq:holder}
d_\cone(z_s,z_t) \leq \int_s^t | \dot{z}_{t^*} |_{g_{\cone}} \, \ed t^* \leq Q^{1/2} |t-s|^{1/2}\,,
\end{equation}
which implies that level sets of $\mc A(\mr{z})$ are equicontinuous. Consider now the set
\begin{equation}
\Omega_R \coloneqq \Omega(\mc{C}_R)=\{\mr{z} \in \Omega\,;\,  \forall\,t \in[0,T]\, ,\, r_t \leq R \} \,;
\end{equation}
For any $Q>0$, the set $\{\mr{z}\in \Omega_R\,;\, \mc{A}(\mr{z}) \leq Q\}$ is also equicontinuous; moreover, since paths in this set are bounded at any time, it is contained in a compact subset of $\Omega$, by the Ascoli-Arzel\`a theorem.

In order to use such sets to prove tightness we need to be able to control the measure of $\Omega\setminus \Omega_R$. In particular, we now show that there exists a constant $C'>0$ such that
\begin{equation}\label{eq:measlargeR}
\bs{\mu}(\Omega\setminus \Omega_R) \leq \frac{C'}{R^2}\,.
\end{equation}
\corr{
In order to show this, consider first the following set of paths
\begin{equation}
\{\mr{z} \in \Omega\,;\,  \forall\,t \in [0,T] \, ,\, r_t > R \} \,.
\end{equation}
Integrating the constraint in equation \eqref{eq:constraintsunb} over such a set with $f=1$, we obtain
\begin{equation}
\bs{\mu}(\{\mr{z} \in \Omega\,;\,  \forall\,t \in [0,T] \, ,\, r_t > R \}) \leq \frac{1}{R^2}\,. 
\end{equation}
Now, consider the set 
\begin{equation}
 \{\mr{z} \in \Omega\setminus \Omega_R \,; \mc{A}(\mr{z})<Q \}\,. 
\end{equation}
For any $\mr{z}$ in this set, there exists $t^*\in[0,T]$ such that $r_{t^*}>R$. Moreover, since $\mc{A}(\mr{z})<Q$, by equation \eqref{eq:holder},
\begin{equation} 
|r_{t} - r_{t^*}| \leq d_\cone(z_{t},z_{t^*}) \leq Q^{1/2} |{t}-t^*|^{1/2}\,,
\end{equation}
for all $t\in[0,T]$, which implies
\begin{equation}
r_t \geq r_{t^*}-Q^{1/2} T^{1/2} > R -Q^{1/2} T^{1/2} .
\end{equation}
In particular, if $Q\leq R^2/(4T)$, then $r_t>R/2$, or also
\begin{equation}
 \{\mr{z} \in \Omega\setminus \Omega_R \,; \mc{A}(\mr{z})<Q \} \subseteq   \{\mr{z}\in \Omega \,; \forall \,t\in [0,T]\, ,  \,r_{t} > R/2 \} \,. 
\end{equation}
Therefore, if $Q \leq R^2/(4T)$,
\begin{equation}
\begin{aligned}
\bs{\mu}(\Omega \setminus \Omega_R) & \leq  \bs{\mu}((\Omega \setminus \Omega_R)\cap \{\mr{z}\,;\, \mc{A}(\mr{z})< Q \}) + \bs{\mu}(\{\mr{z}\,;\, \mc{A}(\mr{z})\geq Q \}) \\
 & \leq  \bs{\mu}(\{\mr{z} \in \Omega \,; \forall \,t\in [0,T]\, ,  \,r_{t} > R/2\}) + \frac{C}{Q} \\
  & \leq  \frac{4}{R^2} + \frac{C}{Q} \,.\\
\end{aligned}
\end{equation}
Taking $Q =R^2/(4T)$, we deduce that
\begin{equation}
\bs{\mu}(\Omega\setminus \Omega_R) \leq \frac{4(CT+1)}{R^2}\,,
\end{equation}
which proves equation \eqref{eq:measlargeR}.} 

Recall that $\{\mr{z}\in \Omega_R\, ;\, \mc{A}(\mr{z}) \leq Q\}$ is contained in a compact set for any $Q>0$ and $R>0$. For any $\epsilon>0$, set  
$R= (8(CT+1)/\epsilon)^{1/2}$. For any admissible $\bs{\mu}$, we have
\begin{equation}
\begin{aligned}
{\bs \mu}(\Omega\backslash  \{\mr{z}\in \Omega_R\, ;\, \mc{A}(\mr{z}) \leq 2C\epsilon^{-1} \} ) 
& \leq {\bs \mu}(\Omega\backslash \{\mr{z}\,; {\mc A}(\mr{z}) \leq 2C \epsilon^{-1}\}) +  {\bs \mu}(\Omega\backslash  \Omega_{R})\\
&\leq \frac{\epsilon}{2C} \int_{\Omega} {\mc A}(\mr{z}) \,\ed {\bs \mu}(\mr{z}) + \frac{\epsilon}{2} \leq \epsilon \,,
\end{aligned}
\end{equation}
which proves tightness.
\end{proof}

We are now ready to prove existence of optimal solutions for the generalized $H(\mr{div})$ geodesic problem. \corr{
Note that due to lemma \ref{lem:compactness}, we can always extract a converging subsequence from any minimizing sequence of problem \ref{prob:generalizedunb}. However, this approach fails to produce a minimizer, since convergence in the narrow topology is not sufficient to pass the constraints to the limit. Note in particular that this is also true if we enforce the strong coupling constraint \eqref{eq:coupling} instead of its homogeneous version in \eqref{eq:couplingconstr}. On the other hand, by choosing this latter as coupling constraint, we can use lemma \ref{lem:rescaling} to construct an appropriate minimizing sequence for which all constraints pass to the limit. We follow this strategy in the proof of proposition \ref{prop:existenceunbounded} below.
}
\begin{proof}[Proof of proposition \ref{prop:existenceunbounded}]
The functional $\mc A(\mr{z})$ is lower semi-continuous; hence so is $\mc A({\bs \mu})$. 
Consider a minimizing sequence $\bs{\mu}_n$ with $n \in \mathbb{N}$. By assumption we can take $\mc{A}(\bs{\mu}_n) \leq C$ for all $n\in \mathbb{N}$. Let $\mr{o}:t \in [0,T] \rightarrow {o} \in \cone$ the path on the cone assigning to every time the apex of the cone ${o}$. Let $\bs{\mu}_n^\mr{o}\coloneqq \bs{\mu}_n\mres\Omega^\mr{o} \in \mc{M}(\Omega)$ the restriction of $\bs{\mu}_n$ to $\Omega^\mr{o}\coloneqq \Omega\setminus\{\mr{o}\}$. Such an operation preserves both the action and the constraints.

Let $\sigma: \Omega \rightarrow \mathbb{R}$ be the 1-homogeneous functional defined by
\begin{equation}\label{eq:sigmaz}
\sigma(\mr{z}) \coloneqq \left( r_0^2 + r_T^2 + \int_0^T r_{t}^2\,\ed t \right)^{1/2}\,.
\end{equation}
For any $\bs{\mu}_n^\mr{o}$ in the sequence, we obviously have that $\sigma(\mr{z})>0$ for $\bs{\mu}_n^\mr{o}$-almost every path. Moreover, since $\bs{\mu}_n^\mr{o}$ satisfies both the homogeneous marginal and coupling constraint, for all $n\in \mathbb{N}$, 
\begin{equation}
 \int_\Omega \sigma(\mr{z})^2\, \ed \bs{\mu}_n(\mr{z}) = T +2 \,.
\end{equation}
Hence we can apply lemma \ref{lem:rescaling} and define a sequence $\tilde{\bs{\mu}}_n \in \mc{P}(\Omega)$ by $\tilde{\bs{\mu}}_n  \coloneqq \mr{dil}_{\sigma/{\sqrt{T+2}},2} \bs{\mu}_n^\mr{o} $. In particular, for all $n\in \mathbb{N}$,  $\tilde{\bs{\mu}}_n$ is concentrated on the set of paths such that $\sigma(\mr{z})= \sqrt{T+2}$, i.e.\
\begin{equation}\label{eq:rescaledsupport}
 \tilde{\bs{\mu}}_n\left (\left \{\mr{z}\in\Omega\,;\, r_0^2 + r_T^2 + \int_0^T r_{t}^2\,\ed t =T +2 \right\} \right) = 1\,.
\end{equation}
Moreover, $\tilde{\bs{\mu}}_n$ satisfies the homogeneous constraint and the coupling constraint, since these are both  2-homogeneous in the radial direction, and for the same reason $\mc{A}(\tilde{\bs{\mu}}_n) =  \mc{A}(\bs{\mu}_n) \leq C$.  This is enough to apply lemma \ref{lem:compactness}; thus, we can extract a subsequence $(\tilde{\bs{\mu}}_n)_n \rightharpoonup \tilde{\bs{\mu}}_\infty \in \mc{P}(\Omega)$. 

We now show that for any $f\in C^0([0,T]\times M)$ the functional 
\corr{\begin{equation}
{\mc{F}(\mr{z}) \coloneqq \int_0^T |f(t,x_t)| r_t^2 \,\ed t }
\end{equation}}
is uniformly integrable with respect to the sequence $(\tilde{\bs{\mu}}_n)_n$, that is, for any $\epsilon>0$ there exists a constant $K>0$ such that for all $n \in \mathbb{N}$ 
\begin{equation}
\int_{\Omega, \mc{F}(\mr{z})>K} \mc{F}(\mr{z}) \,\ed (\tilde{\bs{\mu}}_n)_n(\mr{z}) < \epsilon\,.
\end{equation}
It is sufficient to consider the case $\| f \|_{C^0} = 1$, because the case $\| f \|_{C^0} = 0$ is trivial and otherwise we can always rescale the functional by dividing it by $\| f \|_{C^0}$. 
Recall the definition of the functional $\sigma$ in equation \eqref{eq:sigmaz}; we have 
\begin{equation}
\int_{\Omega, \mc{F}(\mr{z})>K} \mc{F}(\mr{z}) \,\ed (\tilde{\bs{\mu}}_n)_n(\mr{z}) \leq \int_{\Omega, \sigma(\mr{z})^2>K} \sigma(\mr{z})^2 \ed (\tilde{\bs{\mu}}_n)_n(\mr{z})\,.
\end{equation}
 However, by equation \eqref{eq:rescaledsupport} the right-hand side is zero if $K>T+2$, which proves uniform integrability. Hence, using lemma \ref{lem:unifint}, we deduce that $\tilde{\bs{\mu}}_\infty$ satisfies the homogeneous marginal constraint.  Similarly, we can deduce that $\tilde{\bs{\mu}}_\infty$ also satisfies the homogeneous coupling constraint since $(e_0,e_T)_\# (\tilde{\bs{\mu}}_n)_n$ is concentrated on $\mc{C}_R^2$ with $R = \sqrt{T+2}$; hence it is an optimal solution of problem \ref{prob:generalizedunb}.
\end{proof}

\corr{
\begin{remark}
Given $\fmap\in \mr{Diff}(M)$, set $\bs{\gamma}= [(\mr{Id},1),(\fmap,\sqrt{|\mr{Jac}(\fmap)|})]_\#\rho_0$. For such a coupling, there always exists a dynamic plan $\bs{\mu}^*$ such that $\mc{A}(\bs{\mu}^*)<+\infty$. This is constructed explicitly in lemma \ref{lem:diameter}. Therefore, the minimum of the action in problem \ref{prob:generalizedunb} is attained.
\end{remark}
}

In general, we cannot ensure that there exists a minimizer $\bs{\mu}$ of problem \ref{prob:generalizedunb} satisfying the \emph{strong coupling constraint}:
\begin{equation}
(e_0,e_T)_\# \bs{\mu} = \bs{\gamma}\,.
\end{equation}
However, we can easily obtain a characterization for the existence of such minimizers when $\bs{\gamma}$ is deterministic. This relies on the following crucial result which allows us to isolate the part of the solution involving the cone singularity.

\begin{proposition} \label{prop:rescaling}
Suppose that $\bs{\gamma}= [(\mr{Id},1),(\fmap,\sqrt{|\mr{Jac}(\fmap)|})]_\#\rho_0$. 
Any measure $\bs{\mu}\in\mc{M}(\Omega)$ satisfying the homogeneous coupling constraint admits the decomposition
\begin{equation}
\bs{\mu} = \tilde{\bs{\mu}} + \tilde{\bs{\mu}}^0\,,
\end{equation}
where  $\tilde{\bs{\mu}} = \bs{\mu}\mres \{ \mr{z} \in \Omega\,;\, r_0\neq 0 \,,\, r_T \neq 0\}$ and $\tilde{\bs{\mu}}^0 = \bs{\mu}\mres \{ \mr{z} \in \Omega\,;\, r_0 = r_T = 0\}$. Moreover $\tilde{\bs{\mu}}^1 \coloneqq \mr{dil}_{r_0,2} \tilde{\bs{\mu}}$ satisfies the strong coupling constraint, i.e. $(e_0,e_T)_\#  \tilde{\bs{\mu}}^1 = \bs{\gamma}$.
\end{proposition}
\begin{proof}
Let $\bs{\mu}\in\mc{M}(\Omega)$ be any dynamic plan satisfying the homogeneous coupling constraint. We decompose $\bs{\mu} = \tilde{\bs{\mu}} + \tilde{\bs{\mu}}^0$ where 
\begin{equation}
\tilde{\bs{\mu}} \coloneqq \bs{\mu} \mres \{ \mr{z}\in \Omega\,;\, r_0 \neq 0\}\,,\qquad \tilde{\bs{\mu}}^0 \coloneqq \bs{\mu} \mres \{ \mr{z}\in \Omega\,;\, r_0 = 0\}\,.
\end{equation}
Consider the 1-homogeneous functional $\tilde{\sigma}(\mr{z}) : \Omega \rightarrow \mathbb{R}$ defined by $\tilde{\sigma}(\mr{z}) = r_0$. Clearly $\tilde{\sigma}(\mr{z})>0$ for $\tilde{\bs{\mu}}$-almost every path $z$. Moreover, we have 
\begin{equation}
\int_\Omega (\tilde{\sigma}(\mr{z}))^2 \,\ed \tilde{\bs{\mu}}(\mr{z}) = \int_\Omega r_0^2 \,\ed \tilde{\bs{\mu}}(\mr{z}) =1\,.
\end{equation}
Hence, by lemma \ref{lem:rescaling}, the measure $\tilde{\bs{\mu}}^1 \coloneqq \mr{dil}_{r_0,2} \tilde{\bs{\mu}}\in \mc{P}(\Omega)$ is concentrated on paths such that $r_0=1$. Moreover, $\tilde{\bs{\mu}}^0+ \tilde{\bs{\mu}}^1$ still satisfies the homogeneous coupling constraint and in particular, for any $\alpha \in [0,2)$,
\begin{equation}
\begin{aligned}
\int_\Omega r_T^\alpha  \, \ed \tilde{\bs{\mu}}^1(\mr{z}) & = \int_\Omega r_0^{2-\alpha } r_T^\alpha  \, \ed \tilde{\bs{\mu}}^1(\mr{z}) \\
& = \int_\Omega r_0^{2-\alpha } r_T^\alpha  \, \ed (\tilde{\bs{\mu}}^0+\tilde{\bs{\mu}}^1)(\mr{z}) \\
& = \int_M \zeta^\alpha  \, \ed \rho_0\,. 
\end{aligned}
\end{equation}
Taking the limit for $\alpha \rightarrow 2$, by the dominated convergence theorem, 
\begin{equation}
\int_\Omega r_T^2  \, \ed \tilde{\bs{\mu}}^1(\mr{z}) = \int_M \zeta^2  \, \ed \rho_0 = 1\, .
\end{equation}
In turn, this implies that
\begin{equation}
\int_\Omega r_T^2  \, \ed \tilde{\bs{\mu}}^0(\mr{z}) = 0\, ,
\end{equation}
which means that $\tilde{\bs{\mu}}^0$-almost every path $\mr{z}$ has $r_T=0$. This proves that $\tilde{\bs{\mu}}^0 = \bs{\mu}\mres \{ \mr{z} \in \Omega\,;\, r_0 = r_T = 0\}$ and that $\tilde{\bs{\mu}}$ satisfies the homogeneous coupling constraint.

Next, we prove that $(e_0,e_T)_\#  \tilde{\bs{\mu}}^1 = \bs{\gamma}$. For any $g \in C^0(M^2)$ we can take $f =g r_0^2$ in equation \eqref{eq:couplingconstr} yielding 
\begin{equation}
\int_\Omega g(x_0,x_T) \, \ed \tilde{\bs{\mu}}^1(\mr{z}) = \int_{M} g(x,\fmap(x)) \, \ed \rho_0(x)\,.
\end{equation}
Similarly, letting $\zeta \coloneqq \sqrt{|\mr{Jac}(\fmap)|}$,
\begin{equation}
\begin{aligned}
\int_\Omega (r_T-\zeta(x_0))^2 \, \ed \tilde{\bs{\mu}}^1(\mr{z}) &= \int_\Omega  (r_T^2 + \zeta(x_0)^2 -2\zeta(x_0) r_T) \, \ed \tilde{\bs{\mu}}^1(\mr{z}) \\
&= \int_\Omega (r_T^2 + r_0^2\zeta(x_0)^2 -2\zeta(x_0) r_0r_T) \, \ed \tilde{\bs{\mu}}^1(\mr{z}) \\ 
&=  2\int_M \zeta(x)^2 \ed \rho_0(x)   -2 \int_M \zeta(x)^2 \ed \rho_0(x)=0 \,,\\
\end{aligned}
\end{equation}
which means that for $\tilde{\bs{\mu}}^1$-almost every path $r_T = \zeta(x_0)$. Then, for any continuous bounded function $f: \cone^2 \rightarrow \mathbb{R}$, we have
\begin{equation}
\begin{aligned}
\int_\Omega f(z_0, z_T) \, \ed \tilde{\bs{\mu}}^1(\mr{z}) & =  \int_\Omega f([x_0,1], [x_T,\zeta(x_0)]) \, \ed \tilde{\bs{\mu}}^1(\mr{z}) \\ 
& = \int_{M} f([x,1], [\varphi(x),\zeta(x)]) \, \ed \rho_0(x) \,,
\end{aligned} 
\end{equation}
which proves the second part of the proposition. Finally, we must also have $\tilde{\bs{\mu}} = \bs{\mu}\mres \{ \mr{z} \in \Omega\,;\, r_0\neq 0 \,,\, r_T \neq 0\}$, since by definition of the dilation map
\begin{equation}
\int_{\{\mr{z}\in\Omega\,;\, r_T =0\}} r_0^2 \,\ed \tilde{\bs{\mu}} = \int_{\{\mr{z}\in\Omega\,;\, r_T =0\}} r_0^2 \,\ed \tilde{\bs{\mu}}^1 =  \tilde{\bs{\mu}}^1(\{\mr{z}\in\Omega\,;\, r_T =0\}) = 0\,. 
\end{equation}
\end{proof}

\begin{remark} \label{rem:minimalclass}
It should be noted that proposition \ref{prop:rescaling} can be proved also if the coupling constraint in equation \eqref{eq:couplingconstr} is enforced only for homogeneous functions $f\in C^0(\cone^2)$ in the form  $f(z_0,z_1)=g(x_0,x_1)  r_0^{2-\alpha}r_1^\alpha$ and $\alpha\in[0,2]$, for example. Nonetheless, if we defined the constraint in this way, given the fact that $\tilde{\bs{\mu}}^1$ satisfies the strong coupling constraint, we would still retrieve that (when the coupling is deterministic) $\bs{\mu}$ satisfies the coupling constraint with respect to any homogeneous function.
\end{remark}

\begin{corollary}[Existence of minimizers satisfying the strong coupling constraint] \label{cor:existencestrong}
\corr{Suppose that $\bs{\gamma}= [(\mr{Id},1),(\fmap,\sqrt{|\mr{Jac}(\fmap)|})]_\#\rho_0$ and let $\bs{\mu}\in \mc{M}(\Omega)$ (not necessarily a probability measure) be a minimizer of problem \ref{prob:generalizedunb}. Then, if
\begin{equation}\label{eq:r00}
\bs{\mu}(\{\mr{z}\in\Omega\,;\, r_0 =r_T =0\}) = 0\,,
\end{equation}   
the measure $\bs{\mu}$ can be rescaled (in the sense of lemma \ref{lem:rescaling}) to a minimizer satsifying the strong coupling constraint.}
\end{corollary}

The proofs of proposition \ref{prop:existenceunbounded} and \ref{prop:rescaling} give us several insights on the nature of the generalized solutions of the $H(\mr{div})$ geodesic problem. First of all, it is evident that such solutions can only be unique up to rescaling. In fact, since all constraints are homogeneous and preserved by rescaling, given one minimizer one can generate others using the dilation map as in lemma \ref{lem:rescaling}. In addition, if the coupling is deterministic, even using rescaling, in principle one might not be able to find a minimizer satisfying the coupling constraint in the classical sense. By proposition \ref{prop:rescaling}, this happens if all minimizers charge paths which start and end at the apex of the cone and are not trivial. In this case the optimal solutions use the 
the apex to enforce the homogeneous marginal constraint on some time interval contained in $(0,T)$. We will refer to such minimizers as \emph{singular solutions} since they involve the cone singularity. More precisely:
\begin{definition}[Singular generalized $H(\mr{div})$ geodesics] \label{def:singular}
A singular solution of the generalized $H(\mr{div})$ geodesic problem is a minimizer $\bs{\mu}\in\mc{P}(\Omega)$ such that 
\begin{equation}
\bs{\mu}(\{\mr{z}\in \Omega\setminus \{\mr{o}\}\,;\, r_0=r_T=0\})>0\,,
\end{equation}
where $\mr{o}:t\in[0,T] \rightarrow o \in \cone$.
\end{definition}
Proposition \ref{prop:rescaling} can also help us visualize such solutions. 
In fact, for deterministic boundary conditions, to any singular minimizer $\bs{\mu}$ we can still associate a measure $\tilde{\bs{\mu}}^1=\mr{dil}_{r_0,2} \tilde{\bs{\mu}}$ which satisfies the strong coupling constraint but not necessarily the homogeneous marginal constraint. 
In section \ref{sec:examples} we will construct some specific examples of singular minimizers, which will provide some intuition on their meaning.

\section{Existence and uniqueness of the pressure}\label{sec:existence}
In the previous section, we proved existence of minimizers of the generalized $H(\mr{div})$ geodesic problem. In general, given that all constraints are homogeneous, such minimizers are only defined up to rescaling. However, even using rescaling, it might not always be possible to find a minimizer that satisfies the strong coupling constraint. Here, we show that independently of this, the pressure field $P$ in  \eqref{eq:geodesiceulerian} is uniquely defined as a distribution for any given deterministic coupling constraint. This reproduces a similar result proved by Brenier for the incompressible Euler case \cite{brenier1993dual}.

The idea is to extend the set of admissible generalized flows in order to define appropriate variations of the action. By analogy to the Euler case, we consider dynamic plans whose homogeneous marginals are not the Lebesgue measure $\rho_0$, but are sufficiently close to it. Given a dynamic plan $\bs{\nu} \in \mc{P}(\Omega)$ we denote by $\rho^{\bs{\nu}}: [0,T] \times M \rightarrow \mathbb{R}$ the function defined by
\begin{equation}
\rho^{\bs{\nu}}(t,\cdot) \coloneqq \frac{\ed {\mf{h}_t^2 \bs{\nu}}}{\ed \rho_0}  \,,
\end{equation}
for any $t\in[0,T]$.
For an admissible generalized flow $\bs{\nu}$, $\rho^{\bs \nu} =1$. Dynamic plans $\bs{\nu}$ with $\rho^{\bs \nu} \neq 1$ correspond to generalized automorphisms of the cone with a mismatch between the radial variable and the  Jacobian of the flow map on the base space.  

 \begin{definition}[Almost diffeomorphisms]\label{def:almost} 
 A generalized almost diffeomorphism is a probability measure $\bs{\nu}\in \mc{P}(\Omega)$ such that $\rho^{\bs{\nu}} \in C^{0,1}([0,T]\times M)$ and
 \begin{equation}\label{eq:boundrho}
 \| \rho^{\bs{\nu}} -1 \|_{C^{0,1}( [0,T]\times M)} \leq \frac{1}{2}\,.
 \end{equation}
 \end{definition}

For any $\rho \in C^{0,1}([0,T]\times M )$ with $\rho>0$, let $\Phi^\rho: \Omega \rightarrow \Omega$ be the map defined by
\begin{equation}
\Phi^\rho(\mr{z}) \coloneqq (t\in [0,T] \mapsto [x_t, r_t \sqrt{\rho(t,x_t)}] \in \cone)\,.
\end{equation} 
We use this map in the following proposition, which is the equivalent of proposition 2.1 in \cite{brenier1993dual} and justifies our choice for the space of densities in definition \ref{def:almost}. 

\begin{proposition}\label{prop:actionboundgendiff}
Fix a $\rho\in C^{0,1}( [0,T] \times M)$ such that
\begin{equation}
 \| \rho -1 \|_{C^{0,1}} \leq \frac{1}{2}\,, \qquad \rho(0,\cdot) = \rho(1,\cdot) = 1\,.
\end{equation}
Then, given any dynamic plan $\bs{\mu} \in \mc{P}(\Omega)$ with finite action $\mc{A}(\bs{\mu})<+\infty$, satisfying the homogeneous constraint in equation \eqref{eq:constraintsunb}, i.e.\ $\rho^{\bs{\mu}} = \rho_0$, and the coupling constraint \eqref{eq:couplingconstr}, the dynamic plan $\bs{\nu} \coloneqq \Phi^\rho_\# \bs{\mu} \in \mc{P}(\Omega)$ still satisfies the coupling constraint and we have $\rho^{\bs{\nu}} = \rho$; moreover,
\corr{
\begin{equation}\label{eq:boundalmost}
\mc{A}(\bs{\nu}) \leq \mc{A}(\bs{\mu}) +  \| \rho-1 \|_{C^{0,1}} \left( \frac{T}{2}+\mc{A}(\bs{\mu}) \right) + | \rho-1 |^2_{C^{0,1}}(T+\mc{A}(\bs{\mu}))\,.
\end{equation}}
\end{proposition}
\begin{proof} The fact that $\rho^{\bs{\nu}} = \rho$ and that $\bs{\nu}$ satisfies the coupling constraint follows from direct computation. As for equation \eqref{eq:boundalmost}, observe \corr{that $\bs{\mu}$-almost every path is absolutely continuous and that the map $([x,r],t)\in \cone \times [0,T] \mapsto r \sqrt{\rho(t,x)} \in \mathbb{R}_{\geq 0}$ is Lipschitz. Then, for $\bs{\mu}$-almost every path $\mr{z}$ the curve $t\in[0,T] \mapsto r_t \sqrt{\rho(t,x_t)}\in \mathbb{R}_{\geq 0}$ is also absolutely continuous and we have
\begin{equation}\label{eq:estA}
\begin{aligned}
\mc{A}(\bs{\nu}) &= \int_\Omega \int_0^T \mc{A}(\Phi^\rho(\mr{z})) \,\ed t \,\ed \bs{\mu}(\mr{z})  \\
 &= \int_\Omega \int_0^T \rho(t,x_t) |\dot{z}_t |^2_{g_{\cone}} + r_t \dot{r_t} \partial_t(\rho(t,x_t))  + r_t^2 (\partial_t \sqrt{\rho(t,x_t)})^2  \,\ed t \,\ed \bs{\mu}(\mr{z}) \\ 
 &\leq  \| \rho \|_{C^0} \mc{A}(\bs{\mu}) +  \int_\Omega \int_0^T r_t \dot{r_t} \partial_t(\rho(t,x_t))  + r_t^2 (\partial_t \sqrt{\rho(t,x_t)})^2   \,\ed t \,\ed \bs{\mu}(\mr{z}) \,.
\end{aligned}
\end{equation}}
Moreover,
\corr{
\begin{equation}
\begin{aligned}
\int_\Omega \int_0^T r_t \dot{r_t} \partial_t(\rho(t,x_t))\,\ed t \,\ed \bs{\mu}(\mr{z})  & \leq | \rho-1 |_{C^{0,1}} \int_\Omega \int_0^T  r_t |\dot{r}_t| (1+ | \dot{x}_t |) \,\ed t \,\ed \bs{\mu}(\mr{z}) \\ 
& \leq   | \rho-1 |_{C^{0,1}} \left(\frac{T}{2}+ \mc{A}(\bs{\mu}) \right)\,,
\end{aligned}
\end{equation}}
and similarly, since $\rho\geq 1/2$,
\begin{equation}
\begin{aligned}
\int_\Omega \int_0^T  r_t^2 (\partial_t \sqrt{\rho(t,x_t)})^2\,\ed t \,\ed \bs{\mu}(\mr{z})  & \leq \frac{1}{2} \int_\Omega \int_0^T r_t^2   (\partial_t ({\rho(t,x_t)}))^2\,\ed t \,\ed \bs{\mu}(\mr{z}) \\ 
& \leq \frac{1}{2} | \rho-1 |^2_{C^{0,1}} \int_\Omega \int_0^T r_t^2 (1+ | \dot{x}_t |)^2 \,\ed t \,\ed \bs{\mu}(\mr{z}) \\ 
& \leq | \rho-1 |^2_{C^{0,1}}(T+\mc{A}(\bs{\mu}))\,.
\end{aligned}
\end{equation}
Reinserting these estimates into equation \eqref{eq:estA} we obtain \eqref{eq:boundalmost}.
\end{proof}

Consider now the following space
\begin{equation}
B_0 \coloneqq  \{ \rho \in  C^{0,1}([0,T]\times M)\,;\,   \rho(0,\cdot) = \rho(1,\cdot) = 0\}\,,
\end{equation}
which we regard as a Banach space with the $C^{0,1}$ norm. The following theorem shows that we can define the pressure as an element $P\in B_0^*$ and it is the analogue of Theorem 6.2 in \cite{ambrosio2009geodesics}.

\begin{theorem}\label{th:existpres}
Let $\bs{\mu}^*$ be a minimizer for the generalized $H(\mr{div})$ geodesic problem such that $\mc{A}(\bs{\mu}^*) <+\infty$. Then, there exists $P\in B_0^*$ such that
\begin{equation}\label{eq:existpres}
\langle P, \rho^{\bs{\nu}} -1 \rangle \leq \mc{A}(\bs{\nu}) - \mc{A}(\bs{\mu}^*)\,,
\end{equation}
for all generalized almost diffeomorphisms $\bs{\nu}$ satisfying the coupling constraint \eqref{eq:couplingconstr}.
\end{theorem}
\begin{proof}
First of all, observe that for any generalized almost diffeomorphism $\bs{\nu}$ satisfying the coupling constraint,
\begin{equation}
\rho^{\bs{\nu}}(0,\cdot) = \rho^{\bs{\nu}}(1,\cdot) = 1\,;
\end{equation}
hence $\rho^{\bs{\nu}} -1 \in B_0$ and the pairing in equation \eqref{eq:existpres} is well defined. Now, consider the convex set $C\coloneqq \{ \tilde{\rho} \in B_0; \| \tilde{\rho}\|_{C^{0,1}} \leq \frac{1}{2} \}$ and the functional
$\phi: B_0 \rightarrow \mathbb{R}^+ \cup \{+\infty\}$ defined by 
\begin{equation}
\phi(\tilde{\rho})\coloneqq \left \{
\begin{array}{ll}
\inf\{ \mc{A}(\bs{\nu})\,;\, \rho^{\bs{\nu}} = \tilde{\rho}+1 \, \text{and \eqref{eq:couplingconstr} holds} \} & \text{if}\, \tilde{\rho} \in C\,, \\
+\infty & \text{otherwise}\,.
\end{array}
\right.
\end{equation}  
We observe that $\phi(0) = \mc{A}(\bs{\mu}^*) <+\infty$ and so $\phi$ is a proper convex function. We prove that it is bounded in a neighborhood of $\tilde{\rho} =0$. By proposition \ref{prop:actionboundgendiff}, for any $\tilde{\rho} \in C$ there exists a $\bs{\nu}\in \mc{P}(\Omega)$ satisfying $\rho^{\bs{\nu}} = \tilde{\rho}+1$ and the coupling constraint, such that
\corr{
\begin{equation}
\mc{A}(\bs{\nu}) \leq \mc{A}(\bs{\mu}^*) +  \| \tilde{\rho}\|_{C^{0,1}}\left(\frac{T}{2}+ \mc{A}(\bs{\mu}^*) \right)  + |\tilde{\rho} |^2_{C^{0,1}}(T+\mc{A}(\bs{\mu}^*))\,,
\end{equation}
which implies
\begin{equation}
\phi(\rho) \leq \phi(0) +  \| \tilde{\rho} \|_{C^{0,1}}\left(\frac{T}{2}+ \mc{A}(\bs{\mu}^*)  \right) + |\tilde{\rho} |^2_{C^{0,1}}(T+\mc{A}(\bs{\mu}^*))\,.
\end{equation}}
Therefore, $\phi$ is bounded in a neighborhood of $\tilde{\rho}=0$. As a consequence, by standard convex analysis arguments, $\phi$ is also locally Lipschitz on the same neighborhood and the subdifferential of $\phi$ at 0 is not empty, i.e.\ there exists $P\in B_0^*$ such that 
\begin{equation}
\langle P, \tilde{\rho} \rangle \leq \phi(\tilde{\rho}) - \phi(0)\,.
\end{equation}
By the definition of $\phi$, this implies
\begin{equation}
\langle P, \tilde{\rho} \rangle  \leq \mc{A}(\bs{\nu}) - \mc{A}(\bs{\mu}^*)\,,
\end{equation}
for all generalized almost diffeomorphisms $\bs{\nu}$ satisfying $\rho^{\bs{\nu}} = \tilde{\rho} +1$ and the coupling constraint in \eqref{eq:couplingconstr}.
\end{proof}

Theorem \ref{th:existpres} tells us that $\bs{\mu}^*$ is also a minimizer for the augmented action
\begin{equation}
\mc{A}^p(\bs{\nu}) \coloneqq \mc{A}(\bs{\nu}) - \langle P, \rho^{\bs{\nu}}-1 \rangle\,,
\end{equation}
defined on generalized almost diffeomorphisms. Then, for any $\tilde{\rho} \in B_0$, $\bs{\mu}^*_\epsilon\coloneqq \Phi^{1+\epsilon \tilde{\rho}}_\# \bs{\mu}^*$ is a generalized almost diffeomorphism if $\epsilon$ is sufficiently small. Moreover, we must have
\begin{equation}
\left.\frac{\ed}{\ed \epsilon} \mc{A}(\bs{\mu}^*_\epsilon) \right\vert_{\epsilon = 0} =0\,.
\end{equation} 
By the same calculation as in the proof of proposition \ref{prop:actionboundgendiff}, this implies
\begin{equation}
\langle P, \tilde{\rho} \rangle = \int_\Omega \int_0^T \tilde{\rho}(t,x_t) |\dot{z}_t|^2_{g_\cone} + \partial_t( \tilde{\rho}(t,x_t) ) r_t \dot{r}_t \, \ed t \,\ed \bs{\mu}^*(\mr{z})\,,
\end{equation}
for any $\tilde{\rho} \in B_0$, which defines $P$  uniquely as a distribution. This also implies that the functional $\phi$ is actually differentiable at 0 since its subdifferential reduces to a single element. 



\section{Correspondence with deterministic solutions}\label{sec:correspondence}
In this section we study the correspondence between generalized and classical solutions of the $H(\mr{div})$ geodesic equations. In particular, we show that for sufficiently short times classical solutions generate dynamic plans which are the unique minimizers of problem \ref{prob:generalizedunb}. 

We start by proving a modified version of a result presented in \cite{gallouet2017camassa} stating that smooth solutions of the $H(\mr{div})$ geodesic equations are length-minimizing for short times in an $L^\infty$ neighborhood on $\mr{Aut}_{\rho_0}(\cone)$. Let $(\varphi,\lambda)$ be a smooth solution of the system \eqref{eq:geodesic} on the interval $[0,T]$ . Let $P$ be the associated pressure and $\Psi_p(t,x,r) \coloneqq P(t,x) r^2$. Following \cite{brenier1989least} we introduce the following functional on $\Omega$,
\begin{equation}\label{eq:Bfunc}
{\mc B}(\mr{z}) \coloneqq  \left \{ 
\begin{array}{ll}
\int_0^T | \dot{z}_t |^2_{g_{\cone}} - \Psi_p(t,x_t,r_t) \,\ed t & \text{if } \mr{z}\in AC^2([0,T];\cone)\,,\\
+\infty & \text{otherwise}\,.
\end{array}\right.
\end{equation}
Moreover, we consider the function
$b:\cone^2\rightarrow \mathbb{R}$ defined by
\begin{equation}\label{eq:infh1}
 b (p,q) \coloneqq \inf \{ {\mc B}(\mr{z})\, ; \, z_0 = p\,,\, z_T = q   \} \,.
\end{equation}

\begin{lemma}\label{lem:gv} 
Suppose that $M\subset\mathbb{R}^d$ is convex and let $(\varphi,\lambda)$ be a smooth solution of \eqref{eq:geodesic} on $[0,T]\times M$, with $P$ being the associated pressure and $\Psi_p(t,x,r) \coloneqq P(t,x) r^2$.
For any fixed $x\in M$, let $\mr{z}^*=[\mr{x}^*,\mr{r}^*] \in \Omega$ be the curve defined by
$\mr{x}^*: t\rightarrow x^*_t \coloneqq \varphi_t(x)$ and $\mr{r}^*: t \rightarrow r^*_t \coloneqq \lambda_t(x)$. 
Let $r_{min} \coloneqq \min_{(t,x)\in [0,T]\times M } \lambda_t(x)$, $r_{max} \coloneqq \max_{(t,x)\in [0,T]\times M } \lambda_t(x)$ and $\varrho \coloneqq 2 r_{max}/r_{min}$. There exists a constant $C_0>0$ such that, if  
\begin{itemize}
\item for all $t\in[0,T]$ and for all $w \in T_{z_t^*} \cone$, 
\begin{equation}\label{eq:hessassumption}
| \mr{Hess}^{g_\cone}\, \Psi_p( w, w) | \leq \frac{C_0\pi^2}{T^2} |w|^2_{g_\cone} \,;
\end{equation}
\item for all $t_0,t_1\in[0,T]$, 
\begin{equation}\label{eq:distassumption}
d_\cone(z_{t_0}, z_{t_1}) \leq \frac{r_{min}}{4}\,;
\end{equation}
\item the following inequality holds:
\begin{equation}\label{eq:rhoassumption}
\left[\varrho^2 + \left(  \varrho + 1 \right)^2\right] \|P\|_{C^0} \leq \frac{3}{2T^2}\,;
\end{equation} 
\end{itemize}
then, $ {\mc B}(\mr{z}^*) = b (z^*_0,z^*_T)$; moreover,  for any other ${\mr{z}\in AC^2([0,T];\cone)}$ such that $z_0=z_0^*$ and $z_T=z_T^*$,  ${\mc B}(\mr{z}) = {\mc B}(\mr{z}^*)$ if and only if $\mr{z}=\mr{z}^*$. When $M$ is the circle of unit radius $S^1_1 \coloneqq \mathbb{R}/{2\pi}\mathbb{Z}$ the same holds with $C_0=2$ but without the conditions in equations \eqref{eq:distassumption} and \eqref{eq:rhoassumption}.
\end{lemma}

\begin{remark}\label{rem:presshess} The assumption in \eqref{eq:hessassumption} amounts to requiring that the spectral norm of the matrix 
\begin{equation}
g_\cone^{-1/2} (\mr{Hess}^{g_\cone}\, \Psi_p) {g_\cone^{-1/2}}= 
\left( 
\begin{array}{cc}
2P +  (\nabla)^2 P &  \nabla P \\
 (\nabla P)^T  & 2 P
\end{array}
\right)
\end{equation}
be bounded by $C_0 \pi^2/T^2$. This is verified for sufficiently small $T$ if, e.g., $P \in L^\infty([0,T];C^2(M))$. Similarly, the assumption in \eqref{eq:rhoassumption} is verified for sufficiently small $T$ if $P\in C^0([0,T]\times M)$, since for a given smooth solution $\varphi$ with $\varphi_0 = \mr{Id}$, $\varrho= 2 r_{max}/r_{min}\rightarrow 2$ as $T\rightarrow 0$. In addition, note that when $M=S^1_1$ the cone $\cone$ can be identified with $\mathbb{R}^2$ and we do not have to deal with the singularity introduced by the apex. This is the reason why the assumptions in \eqref{eq:distassumption} and \eqref{eq:rhoassumption} are not necessary in this case.
\end{remark}

The proof of lemma \ref{lem:gv} is postponed to the appendix. Lemma \ref{lem:gv} is the equivalent of lemma 5.2 in \cite{brenier1989least} on the cone. As in \cite{brenier1989least}, we can use it to prove the optimality of the plan concentrated on the continuous solution.
For this, however, we also need the following additional result on the function $b$, which characterizes the mininimizing paths starting and ending at the apex.

\begin{lemma}\label{lem:bis0}
Suppose $P\in C^0([0,T]\times M)$ and $P_{max}\coloneqq \max_{(t,x)\in [0,T]\times M} P(t,x) \leq (\pi/T)^2$. Then $b(o,o)=\mc{B}(\mr{o})=0$ where $\mr{o}:t\in[0,T] \rightarrow o \in \cone$. If the inequality is strict then  for any other ${\mr{z}\in AC^2([0,T];\cone)}$ such that $z_0=o$ and $z_T=o$,  ${\mc B}(\mr{z}) = {\mc B}(\mr{o})$ if and only if $\mr{z}=\mr{o}$.
\end{lemma} 
\begin{proof}
For the first part, observe that for any $\mr{z}\in AC^2([0,T];\cone)$ such that $r_0 = r_T =0$, using Poincar\'e inequality on $\mr{r}:t\in[0,T]\rightarrow r_t \in \mathbb{R}_{\geq 0}$
\corr{
\begin{equation}
\begin{aligned}
\mc{B}(\mr{z}) & \geq \int_0^T |\dot{z}_t|_{g_\cone}^2 -r_t^2 P_{max} \, \ed t\\
& \geq \int_0^T r_t^2 |\dot{x}_t|^2 + \frac{\pi^2}{T^2} r_t^2  -r_t^2  P_{max}   \, \ed t\\
& \geq  \left(\frac{\pi^2}{T^2} -  P_{max} \right) \int_0^T  r_t^2 \, \ed t\,.
\end{aligned}
\end{equation}}
This implies that $b(o,o) \geq 0$. Clearly, $b(o,o) \leq \mc{B}(\mr{o}) = 0$ and therefore $b(o,o) =0$. 
For the second part, if the inequality is strict, $C\coloneqq \frac{\pi^2}{T^2} - P_{max}  >0$. Then, for any other ${\mr{z}\in AC^2([0,T];\cone)}$ such that $z_0=o$ and $z_T=o$, and satisfying ${\mc B}(\mr{z}) = {\mc B}(\mr{o})$, we have
\corr{\begin{equation}
0 = {\mc B}(\mr{z}) \geq C \int_0^T  r_t^2 \, \ed t\,,
\end{equation}}
which implies $\mr{z} = \mr{o}$.
\end{proof}

\begin{theorem}\label{th:correspondence}
Under the assumptions of lemma \ref{lem:gv}, the dynamic plan $\bs{\mu}^* = (\varphi, \lambda)_\# \rho_0$ is an optimal solution of problem \ref{prob:generalizedunb}  with $\bs \gamma = [(\varphi_0,\lambda_0),(\varphi_T,\lambda_T)]_\# \rho_0$.  \corr{If the inequalities \eqref{eq:hessassumption} and \eqref{eq:rhoassumption} are strict, the solution $\bs{\mu}^*$ is  unique in the following sense: for any minimizer $\bs{\mu}$, the measure $\bs{\mu}^\mr{o} \coloneqq \bs{\mu} \mres \Omega^\mr{o}$, with $\Omega^\mr{o} \coloneqq \Omega \setminus \{\mr{o}\}$, is equal to $\bs{\mu}^*$ up to rescaling (defined in lemma \ref{lem:rescaling}). }
\end{theorem}
 
\begin{proof}
Let $\bs \mu$ be any dynamic plan with finite action, i.e.\ $\mc A(\bs \mu) <+\infty$, and satisfying the constraints in \eqref{eq:couplingconstr} and \eqref{eq:constraintsunb}. Consider the functional
\begin{equation}
\mc P(\mr{z}) = \int_0^T \Psi_p(t,x_t,r_t) \ed t\,.
\end{equation}
Then,
\begin{equation}
\begin{aligned}
\int_{\Omega} \mc P(\mr{z})\, \ed \bs \mu(\mr{z}) & = \int_{\Omega} \int_0^T \Psi_p(t, x_t ,r_t)  \,\ed t \, \ed \bs \mu(\mr{z})\\
& = \int_{\Omega} \int_0^T   P(t,x_t) r_t^2 \, \ed t \, \ed \bs \mu (\mr{z}) \\
& = \int_0^T \int_{M}  P \, \ed \rho_0  \ed t\,.\\
\end{aligned}
\end{equation}
Hence,
\begin{equation}\label{eq:bstar}
\mc{B}(\bs \mu) \coloneqq \int_\Omega \mc{B}(\mr{z}) \, \ed \bs{\mu}(\mr{z}) = \mc A(\bs \mu) -  \int_0^T \int_{M}  P \, \ed \rho_0  \ed t\,,
\end{equation}
and since equation \eqref{eq:bstar} also holds replacing $\bs{\mu}$ with  $\bs{\mu}^*$,
\begin{equation}\label{eq:AB}
\mc{B}(\bs \mu) - \mc{B}(\bs \mu^*) = \mc A(\bs \mu) - \mc A(\bs \mu^*) \,.
\end{equation}

Now, by proposition \ref{prop:rescaling} we have the decomposition $\bs{\mu} = \tilde{\bs{\mu}} + \tilde{\bs{\mu}}^0$ where $\tilde{\bs{\mu}} = \bs{\mu}\mres \{ \mr{z} \in \Omega\,;\, r_0\neq 0 \}$ and $\tilde{\bs{\mu}}^0 = \bs{\mu}\mres \{ \mr{z} \in \Omega\,;\, r_0 =r_T= 0 \}$. Therefore, integrating the function $b$ defined in \eqref{eq:infh1} with respect to $\bs \mu$ we obtain 
\begin{equation}\label{eq:b00}
\begin{aligned}
\int_{\Omega} b(z_0,z_T) \,  \ed \bs \mu(\mr{z})  
& = \int_{\Omega}  b(z_0,z_T) \,\ed \tilde{\bs{\mu}}(\mr{z})+  \int_{\Omega}  b(o,o) \,\ed \tilde{\bs{\mu}}^0(\mr{z}) \\
& = \int_{\Omega}  b(z_0,z_T) \,\ed \tilde{\bs{\mu}}(\mr{z})\,,
\end{aligned}
\end{equation}
where we used the fact that $b(o,o)=0$  by lemma \ref{lem:bis0}. By proposition \ref{prop:rescaling},  $\tilde{\bs{\mu}}^1 \coloneqq \mr{dil}_{r_0,2} \tilde{\bs{\mu}}$
satisfies the strong coupling constraint $(e_0,e_T)_\# \tilde{\bs{\mu}}^1 = \bs{\gamma}$. Moreover, $b$ is 2-homogeneous (because $\mc{B}$ is) and therefore 
\begin{equation}\label{eq:b}
\begin{aligned}
\int_{\Omega} b(z_0,z_T) \,  \ed \tilde{\bs \mu}(\mr{z})  = \int_{\Omega}  b(z_0,z_T) \,\ed \tilde{\bs{\mu}}^1 (\mr{z})  
 = \int_{\cone^2}  b(p,q) \,\ed \bs \gamma(p,q)\,. 
\end{aligned}
\end{equation}
We get the same result integrating $b$ with respect to $\bs \mu^*$. In particular, by lemma \ref{lem:gv},
\begin{equation}
\int_{\Omega} b(z_0,z_T) \,  \ed \bs \mu(\mr{z}) = \mc{B}(\bs \mu^*) \,.
\end{equation}
By definition of $b$ in \eqref{eq:infh1}, for any path $\mr{z}\in{\Omega}$,  $\mc{B}(\mr{z}) \geq b(z_0,z_T)$ and therefore
\begin{equation}
\mc{B}(\bs \mu) \geq \int_{\Omega} b(z_0,z_T) \,  \ed \bs \mu(\mr{z}) = \mc{B}(\bs \mu^*) \,,
\end{equation}
which implies the same inequality for $\mc A$ due to equation \eqref{eq:AB}. This proves that  $\bs{\mu}^*$ is an optimal solution.

In order to prove uniqueness, let $\bs \mu$ be a solution of problem \ref{prob:generalizedunb}. \corr{Without loss of generality we can assume that $\bs{\mu} = \bs{\mu}^{\mr{o}}$. Then, equations \eqref{eq:AB} and \eqref{eq:b} imply
\begin{equation} 
\int_{\Omega} \mc{B}(\mr{z}) - b(z_0,z_T) \,  \ed \bs \mu(\mr{z}) = \mc{B}(\bs \mu) - \mc{B}(\bs \mu^*) = \mc A(\bs \mu) - \mc A(\bs \mu^*) = 0 \,.
\end{equation}
Since for any $\mr{z}\in \Omega$ we have $\mc{B}(\mr{z}) \geq b(z_0,z_T)$, then for $\bs \mu$-almost every path $\mr{z}$, $\mc{B}(\mr{z}) = b(z_0,z_T)$. Clearly, also for $\bs \mu^*$-almost every path $\mr{z}$, $\mc{B}(\mr{z}) = b(z_0,z_T)$. 
Now, if $\bs{\mu}$ satisfies the strong coupling constraint, for $\bs \mu$-almost every path $\mr{z}$ and for $\bs \mu^*$-almost every path $\mr{z}^*$ such that $z_0 = z_0^*$ and $z_T = z_T^*$, we have $\mc{B}(\mr{z}) = \mc{B}(\mr{z}^*)=b(z^*_0,z^*_T)$. This implies $\mr{z}=\mr{z}^*$ by lemma \ref{lem:gv}. In other words, $\bs{\mu}$ and $\bs{\mu}^*$ are concentrated on the same paths and due to the strong coupling constraint they must coincide. } 

%


%
\corr{
On the other hand, suppose that $\bs{\mu}$ does not satisfy the strong coupling constraint. Recall that for $\bs{\mu}$-almost every path $\mr{z}$ we have $\mc{B}(\mr{z}) = b(z_0,z_T)$. Then, defining $\tilde{\Omega} \coloneqq  \{\mr{z}\in\Omega\,;\, r_0 =r_T = 0 \}$, we have
\begin{equation}
\int_{\tilde{\Omega}}  \mc{B}(\mr{z}) \, \ed\bs{\mu}(\mr{z}) = \int_{\tilde{\Omega}}  b(z_0,z_T) \, \ed\bs{\mu}(\mr{z}) =  b(o,o)\,\bs{\mu}({\tilde{\Omega}})   = 0\,.
\end{equation}
For any $\mr{z} \in \tilde{\Omega}$ we also have $\mc{B}(\mr{z})\geq b(o,o)=0$. Hence, we find that for $\bs{\mu}$-almost every path $\mr{z}$ such that $z_0 =z_T =o$, we have $\mc{B}(\mr{z})= 0$ which by lemma \ref{lem:bis0} is equivalent to $\mr{z} = \mr{o}$. This implies  
\begin{equation}\label{eq:conditionrescaling} 
\bs{\mu}(\{ \mr{z}\in\Omega \,;\, r_0=r_T=0\}) =0\,.  
\end{equation}
Then, corollary \ref{cor:existencestrong} implies that $\bs{\mu}$ can be rescaled to a measure satisfying the strong coupling.}
\end{proof}

The assumptions on the pressure in lemma \ref{lem:gv} are less strict for the case of the circle.
This leads to the following result. 

\begin{corollary}\label{cor:correspondenceS1}
Let $M=S^1_1$ and let $(\varphi,\lambda)$ be a smooth solution of \eqref{eq:geodesic} on $[0,T]\times M$, with $P$ being the associated pressure and $\Psi_p(t,x,r) \coloneqq P(t,x) r^2$.  If for all $t\in[0,T]$ and for all $w \in T_{z_t^*} \cone$, 
\begin{equation}\label{eq:hessassumptionS1}
| \mr{Hess}^{g_\cone}\, \Psi_p( w, w) | \leq \frac{2\pi^2}{T^2} |w|^2_{g_\cone} \,,
\end{equation}
then the dynamic plan $\bs{\mu}^* = (\varphi, \lambda)_\# \rho_0$ is an optimal solution of problem \ref{prob:generalizedunb} for the coupling $\bs \gamma = [(\varphi_0,\lambda_0),(\varphi_T,\lambda_T)]_\# \rho_0$. If the inequality in equation \eqref{eq:hessassumptionS1} is strict,
it is unique  up to rescaling (in the sense of lemma \ref{lem:rescaling}).
\end{corollary}


\section{Some examples of generalized $H(\mr{div})$ geodesics}\label{sec:examples}

In this section we construct some instructive examples of generalized $H(\mr{div})$ geodesics which shed some light on the need of the relaxation and its tightness. In particular, we will focus on deterministic boundary conditions and construct singular solutions, i.e. minimizers that charge (non-trivial) paths starting and ending at the apex of the cone. 
\corr{This will allow us to prove two main results. First, that our relaxation is not tight on $S^1_R$, the circle of radius $R$, when $R$ is sufficiently large; and second, that on the torus, for specific boundary conditions, problem \ref{prob:det} may admit no smooth minimizer, whereas we can construct a singular solution as the limit of a minimizing sequence of smooth flows. This suggests that problem \ref{prob:generalizedunb} is a tight relaxation of problem \ref{prob:det} in dimension $d\geq 2$. } 

We start by considering an important generalized flow which provides an upper bound on the action on any domain and for any deterministic coupling. 

\begin{lemma}\label{lem:diameter}
Consider the generalized $H(\mr{div})$ geodesic problem with coupling given by $\bs{\gamma}=(\fmap,\sqrt{|\mr{Jac}(h)|})_\# \rho_0$ where $\fmap\in\mr{Diff}(M)$. Denote by $\rho_0$ the Lebesgue measure on $M$, \corr{normalized so that $\rho_0(M)=1$}. Then the measure
\begin{equation}\label{eq:nondetrot}
\bs{\mu}^* = \frac{1}{2}(\mr{Id},\zeta^0)_\# \rho_0 + \frac{1}{2}(\psi^1,\zeta^1)_\# \rho_0\,,
\end{equation}
with
\begin{gather}
\zeta^0_t(x) = \sqrt{2}\sin(\sqrt{P^*}t)\,, \quad 
\zeta^1_t (x) = \left\{
\begin{array}{ll}
\sqrt{2} \cos(\sqrt{P^*} t)  & t\leq T/2\,,\\
- \sqrt{2\, |\mr{Jac}(h(x))|} \cos(\sqrt{P^*} t)  & t> T/2\,,
\end{array} 
\right.\\
\psi^1_t(x) = \left\{
\begin{array}{ll}
x & t\leq T/2\,,\\
\fmap(x) & t> T/2\,,
\end{array}
\right.
\end{gather}
where $P^*= \pi^2/T^2$, is an admissible generalized flow and the action of the minimizer is bounded from above by $\mc{A}(\bs{\mu}^*)=\pi^2/T$.
\end{lemma}
\begin{proof}
We need to check that $\bs{\mu}^*$ is a probability measure and that it satisfies the homogeneous marginal and coupling constraints. The fact that $\bs{\mu}^*(\Omega)=1$ is immediate from the definition. As for the marginal constraint, observe that for any $ f \in C^0([0,T] \times M)$,
\begin{equation}
\begin{aligned}
\int_\Omega \int_0^T f(t,x_t) r_t^2 \, \ed t \, \ed{\bs \mu}(\mr{z}) & = \frac{1}{2} \int_M \int_0^T  f(t,x) 2 \sin^2(\sqrt{P^*} t) \,\ed t \, \ed\rho_0(x)\\&\quad + \frac{1}{2} \int_M \int_0^{T/2}  f(t,x) 2 \cos^2(\sqrt{P^*} t) \,\ed t \, \ed\rho_0(x)\\&\quad + \frac{1}{2} \int_M \int_{T/2}^{T}  f(t,\fmap(x)) 2 |\mr{Jac}(\fmap(x))| \cos^2(\sqrt{P^*} t) \,\ed t \, \ed\rho_0(x)  \\
& = \int_M \int_0^T  f(t,x) \,\ed t \, \ed\rho_0(x) \,. 
\end{aligned}
\end{equation}
By similar calculations also the homogeneous coupling constraint holds and therefore $\bs{\mu}^*$ is admissible. Moreover, the action associated with $\bs{\mu}^*$ is given by
\begin{equation}
\begin{aligned}
\mc{A}(\bs{\mu}^*) &=  \frac{1}{2}  \int_{M} \int_0^T |\dot{\zeta}^0_t(x)|^2+ |\dot{\zeta}^1_t(x)|^2 \ed  t \,\ed \rho_0(x) \\ 
 &=  \int_{M} \int_0^T P^* \ed  t \,\ed \rho_0(x) =\frac{\pi^2}{T}\,. 
\end{aligned}
\end{equation}
\end{proof}

The dynamic plan in lemma \ref{lem:diameter} shows that in our generalized formulation we can reach any final configuration only by changes in the Jacobian, although in a non-deterministic sense. In the following we will consider several instances of this flow for different domains and couplings and we will prove that in some cases it also minimizes the generalized $H(\mr{div})$ action.
In fact, the idea behind the construction of the generalized flow $\bs{\mu}^*$ is that as for geodesics on the cone, we expect that for a sufficiently large displacement optimal solutions would concentrate on straight lines in the radial direction passing by the apex of the cone. If there is no motion on the base space $M$, the geodesic equation \eqref{eq:geodesic} in the radial direction reduces to
\begin{equation}
\ddot{\lambda} + \lambda P = 0
\end{equation}
The dynamic plan $\bs{\mu}^*$ concentrates precisely on solutions to this equation with constant pressure $P=P^*$.

It should also be noted that $\bs{\mu}^*$ is exactly in the form discussed in proposition \ref{prop:rescaling}, i.e.\ it is decomposed in the sum of two measures, $\bs{\mu}^* = \tilde{\bs{\mu}} + \tilde{\bs{\mu}}^0$, where
\begin{equation}
\tilde{\bs{\mu}}^0 = \frac{1}{2}(\mr{Id},\zeta^0)_\# \rho_0\,, \quad   \tilde{\bs{\mu}} =  \frac{1}{2}(\psi^1,\zeta^1)_\# \rho_0\,.
\end{equation}
In particular, $\tilde{\bs{\mu}}$ does not charge paths starting and ending at the apex, so it can be rescaled to a probability measure satisfying the strong coupling constraint but not the homogeneous marginal constraint. This is given by
\begin{equation}
\tilde{\bs{\mu}}^1 = \mr{dil}_{r_0,2} \tilde{\bs{\mu}} = (\psi^1,\zeta^1/\sqrt{2})_\# \rho_0\,.
\end{equation}
The dynamic plan $\tilde{\bs{\mu}}^1$ describes a peculiar solution in which particles gradually disappear up to time $T/2$, when the whole domain vanishes, and then gradually reappear in the final configuration. 
\corr{This phenomenon is related to the collision of a peakon and an anti-peakon in one dimension, which is a well-known solution of the CH equation \cite{camassa1993integrable}}. Such a solution implies that arbitrarily small portions of the domain can be stretched to occupy finite area at finite bounded cost. The generalized solution in \eqref{eq:nondetrot} replicates this behavior in an averaged sense across the domain.  \corr{This will be made precise by the approximation results in proposition \ref{prop:convergences1} and theorem \ref{th:approxt2}.}

\subsection{Construction of a generalized solution on the circle}\label{sec:gencircle}
We now consider the generalized $H(\mr{div})$ geodesic problem on $S^1_R$, the circle of radius $R$. For specific boundary conditions given by uniform rotation and when $R=1$, we show that the generalized flow in lemma \ref{lem:diameter} is a minimizer although not unique, having the same cost as the deterministic solution.
When $R>1$, the constant speed rotation is not a minimizer since its action is strictly larger than $\pi^2/T$. \corr{ Moreover, if $R$ is sufficiently large, there is no minimizing sequence of deterministic smooth flows whose action tend to $\pi^2/T$. This implies that in this case, the relaxed problem \ref{prob:generalizedunb} is not tight. 
In order to make this precise, we start by proving the following lower bound on the action of problem \ref{prob:det} on $S^1_R$ and for boundary conditions given by uniform rotation.
\begin{lemma}\label{lem:lowbound1d}
Let $\varphi^*:[0,T]\times S^1_R \rightarrow S^1_R$ be a smooth flow satisfying $\varphi^*_0 =\mr{Id}$ and $\varphi^*_T=\fmap$, where $\fmap$ prescribes uniform rotation by half of the circle length, i.e.\ $\fmap: x \in \mathbb{R}/{2\pi R}\mathbb{Z} \rightarrow x+\pi R \in \mathbb{R}/{2\pi}\mathbb{Z}$. Then,
\begin{equation}
\frac{1}{2\pi R} \int_{0}^{2\pi R} \mc{A}([\varphi^*(x),\sqrt{\mr{Jac}(\varphi^*(x))}])\, \ed x \geq \frac{\tanh (2 \pi R)}{2 T} \pi R
\,.
\end{equation} 
\end{lemma} 
\begin{proof}
Consider the following problem
\begin{equation}\label{eq:peakonac}
\inf \left\{  \frac{1}{2\pi R} \int_{0}^{2\pi R} \mc{A}([\varphi(x),\sqrt{\mr{Jac}(\varphi(x))}])\, \ed x \, ; \, \varphi_0(0) = 0\,, \varphi_T(0) = \pi R \right\}\, .
\end{equation}
where the infimum is taken over smooth curves $t \in [0,T] \mapsto \varphi_t \in  \mr{Diff}(S^1_R)$. The quantity in equation \eqref{eq:peakonac} provides a lower bound for the action associated with $\varphi^*$. Fix a smooth flow $\varphi$. For any $t\in(0,1)$, let $u_t \in H^1(S^1_R)$ be the velocity field minimizing
\begin{equation}
\frac{1}{2\pi R} \int_{0}^{2\pi R} |u|^2 + \frac{1}{4} |\partial_x u|^2 \ed x \, ,
\end{equation}
over all $u\in H^1(S^1_R)$ such that $u(\varphi_t(0)) = \partial_t\varphi_t(0)$. In particular, we have $u_t = G*m_t$ where $m$ is in the form
\begin{equation}
m_t(x) = p_t \, \delta(x -{\varphi_t(0)})\,,
\end{equation}
with $p_t \in \mathbb{R}$ depends on the boundary conditions, and $G$ is the Green's function for the operator $\mr{Id} - \frac{1}{4} \partial_{xx}$, which is given by
\begin{equation}\label{eq:green}
G(x,y) = \frac{\cosh(2|x-y|-2\pi R)}{\sinh(2\pi R)}
\end{equation}
(note that $u_t$ has the same form of a peakon on $S^1_R$, see section \ref{sec:collpeak}). 
Then, by direct calculation,
\begin{equation}
\begin{aligned}
\frac{1}{2\pi R} \int_{0}^{2\pi R} \mc{A}([\varphi(x),\sqrt{\mr{Jac}(\varphi(x))}])\, \ed x 
&\geq \int_0^T \int_{0}^{2\pi R} |u_t|^2 + \frac{1}{4} |\partial_x u_t|^2 \ed x \,\ed t \\
& =    \frac{\tanh (2 \pi R)}{2 \pi R} \int_0^T |\partial_t\varphi_t(0)|^2 \ed t\,.
\end{aligned}
\end{equation}
Using the boundary conditions on $\varphi$ from equation \eqref{eq:peakonac} gives the result.
\end{proof}
In view of lemma \ref{lem:diameter}, lemma \ref{lem:lowbound1d} implies that our relaxation (problem \ref{prob:generalizedunb}) is not tight on $S^1_R$ for sufficiently large $R$. This is made precise in the following theorem.  
}

\begin{theorem}\label{th:rots1}
Consider the generalized $H(\mr{div})$ geodesic problem on $S^1_R$ with coupling constraint given by uniform rotation by half of the circle length, i.e.\ in polar coordinates $\fmap:\theta \in \mathbb{R}/{2\pi}\mathbb{Z} \rightarrow \theta+\pi \in \mathbb{R}/{2\pi}\mathbb{Z} $. Denote by $\rho_0 = (2\pi)^{-1}\ed \theta$ the normalized Lebesgue measure on the circle. The following holds:
\begin{enumerate}
\item when $R=1$ the dynamic plan $\bs{\mu}^*$ in lemma \ref{lem:diameter}, i.e.\ equation \eqref{eq:nondetrot} with
\begin{equation}
\zeta^0_t(\theta) = \sqrt{2}\sin(\sqrt{P^*}t)\,, \quad \zeta^1_t (\theta) = \sqrt{2}| \cos(\sqrt{P^*} t) | \,, \quad 
\psi^1_t(\theta) = \left\{
\begin{array}{ll}
\theta & t\leq T/2\,,\\
\theta+\pi & t> T/2\,,
\end{array}
\right.
\end{equation}
as well as the dynamic plan induced by constant speed rotation are minimizers corresponding to the constant pressure $P^* = (\pi/T)^2$; 
\item when $R>1$ the constant speed rotation is not a minimizer; \corr{moreover, if $R$ is sufficiently large, the infimum of the deterministic $H(\mr{div})$ geodesic problem \ref{prob:det} is strictly larger than that of the generalized geodesic problem \ref{prob:generalizedunb}. }
\end{enumerate} 
\end{theorem}   
\begin{proof}
For the first point, observe that from the Euler-Lagrange equations \eqref{eq:geodesiceulerian} the pressure relative to constant speed rotation on $S^1_1$ is given by
\begin{equation}
P^{rot} = |u|^2 = \frac{\pi^2}{T^2}.
\end{equation}
This satisfies the hypotheses of corollary \ref{cor:correspondenceS1} (see remark \ref{rem:presshess}) and therefore the constant rotation is a minimizer. Since the Jacobian stays constant during the rotation, the associated action is given by 
\begin{equation}
\mc{A}^{rot} = \frac{1}{2\pi} \int_0^{2\pi} \int_0^T |u|^2\, \ed t \,\ed \theta = \frac{\pi^2}{T}\,.
\end{equation}
On the other hand, by lemma \ref{lem:diameter}, $\bs{\mu}^*\in \mc{P}(\Omega)$ is admissible and its action is equal to $\mc{A}(\bs{\mu}^*) = \pi^2/T$, independently of $R$.
Hence $\bs{\mu}^*$ is also a minimizer and it must share the same pressure with the constant speed rotation, $P^{rot} = P^*$.
For the second point, observe that the action for constant speed rotation on $S^1_R$ is $\mc{A}^{rot}_R = R^2 \mc{A}^{rot}>\mc{A}(\bs{\mu}^*)$ whenever $R>1$. \corr{Similarly, for $R$ sufficiently large we have
\begin{equation}
\frac{\tanh (2 \pi R)}{2 T} \pi R > \mc{A}(\bs{\mu}^*)\,.
\end{equation}
We conclude applying lemma \ref{lem:lowbound1d}.}
\end{proof}

%
\corr{
\begin{remark}
For $d=1$ one can produce a tight relaxation of problem \ref{prob:det} using different techniques than those used in the present paper. This approach is developed in \cite{dimarino:hal-02161686} and is specific to dimension one. However note that one still needs to rely on theorem \ref{th:correspondence} in order to conclude that smooth geodesics are the unique global length-minimizers for this tight one-dimensional relaxation.
\end{remark}
}

\subsection{Collision of peakons and an approximation result}\label{sec:collpeak}

Before going further with the construction of a generalized solutions on a two-dimensional domain, we need to clarify the connection between the solution presented in theorem \ref{th:rots1} and diffeomorphisms of the circle. In particular, here we show that if no rotation occurs, the generalized flow in theorem \ref{th:rots1} can be approximated using linear peakon/anti-peakon collisions. This will serve as a basis to construct a sequence of deterministic flows converging to a non-deterministic minimizer in two dimensions.

Consider the CH equation on the circle $S^1_1$ with Lagrangian $\int_{0}^{2\pi} |u|^2 + \frac{1}{4}|\partial_\theta u|^2\, \ed \theta$, where $u:[0,T] \times S^1_1 \rightarrow \mathbb{R}$ is the Eulerian velocity field. Peakon solutions can be described in terms of momentum $m = u - \frac{1}{4} \partial_{\theta}^2 u$ as a linear combination of Dirac delta functions, i.e.\
\begin{equation}\label{eq:peakonmomentum}
m(t,\theta) = \sum_{i=1}^N p_i(t) \delta(\theta - {\theta_i(t)})\,,
\end{equation}
where $p_i:[0,T] \rightarrow \mathbb{R}$ and $\theta_i: [0,T] \rightarrow S^1_1$ are appropriate functions specifying the momentum carried by the $i$th peakon and its location, respectively.  The associated velocity field is given by $u = G * m$ where $G$ is the Green's function for the operator $\mr{Id} - \frac{1}{4}\partial_{\theta}^2$ (see equation \eqref{eq:green}). 

The collision of a peakon and an anti-peakon corresponds to the case $N=2$, $p_2 =-p_1$, $\theta_2 = 2\pi-\theta_1$, in which case there exists a finite time $T^*$ such that as $t\rightarrow T^*$ collision occurs, i.e.\ $\theta_1 =\theta_2$. A similar behavior occurs for the Lagrangian $\int_{0}^{2\pi} \frac{1}{4} |\partial_\theta u|^2\, \ed \theta$, which corresponds to the Hunter-Saxton equation. In this case, the velocity field is simply given by the linear interpolation of the velocity at $\theta_1$ and $\theta_2$ (see figure \ref{fig:shock}) and the Jacobian of the flow map is piecewise constant. Hence specifying the trajectory $\theta_1(t)$ uniquely defines the flow. We refer to such a solution as \emph{linear peakon/anti-peakon} collision. The associated flow on a circle of arbitrary radius $R$ is described in the following lemma.

\begin{lemma}\label{lem:collision}
For a given $\epsilon > 0$, let $\varphi^\epsilon:[0,T] \times S^1_R \rightarrow S^1_R$ be the flow map defined in polar coordinates by
\begin{equation}\label{eq:shock}
\varphi^\epsilon_t(0) = 0\,, \quad \
\partial_\theta \varphi^\epsilon_t(\theta) = 
\left\{
\begin{array}{ll}
1- \sin \left( \frac{\pi t}{2 (T+\epsilon)} \right)   &\text{if}~ \frac{\pi}{2}<\theta<\frac{3\pi}{2}\,, \\
1+ \sin \left( \frac{\pi t}{2 (T+\epsilon)} \right)   &\text{otherwise}\,,
\end{array}
\right. 
\end{equation}
Then the associated action is uniformly bounded and 
\begin{equation}\label{eq:actionlimit}
\lim_{R\to 0} \lim_{\epsilon \to 0}  \frac{1}{2\pi} \int_{0}^{2\pi} \mc{A}([R\varphi^\epsilon(\theta),\lambda^\epsilon(\theta)]) \, \ed \theta = \frac{\pi^2}{16T}\,,
\end{equation}
where $\lambda^\epsilon = \sqrt{\mr{Jac}(\varphi^\epsilon)}$ and  
\begin{equation}\label{eq:polaraction}
 \frac{1}{2\pi} \int_{0}^{2\pi} \mc{A}([R\varphi^\epsilon(\theta),\lambda^\epsilon(\theta)])\, \ed \theta= \frac{1}{2\pi} \int_{0}^{2\pi} \int_0^T R^2 (\lambda^\epsilon_t(\theta))^2|\dot{\varphi}^\epsilon_t(\theta)|^2+ |\dot{\lambda}_t^\epsilon(\theta)|^2\,\ed t  \, \ed \theta
\end{equation}
is the action expressed in polar coordinates.
\end{lemma}

\begin{proof} The result follows by direct computation and by definition of the functional $\mc{A}$ in equation \eqref{eq:energyfunctional}. Note that the expression for the action in equation \eqref{eq:polaraction} can be justified by an appropriate change of variables. Specifically, denoting by $\varphi^\epsilon_{R}$ the flow map in arc length coordinates $x\in \mathbb{R}/{2\pi R}\mathbb{Z}$, we have $\theta = x/R$ and
\begin{equation}
\varphi^\epsilon_R(x) = R \varphi^\epsilon\left( \frac{x}{R}\right)\,, \quad \partial_x \varphi^\epsilon_R(x) = \partial_\theta \varphi^\epsilon\left( \frac{x}{R}\right)\,.
\end{equation}
Denoting $\lambda^\epsilon_R = \sqrt{\partial_x \varphi^\epsilon_{R}}$, since $\rho_0 = (2\pi R)^{-1} \ed x$  we obtain that the action is given by
\begin{equation}
\begin{aligned}
\int_{S^1_R} \mc{A}([\varphi^\epsilon_R(x),\lambda^\epsilon_R(x)])  \, \ed \rho_0(x) & = \frac{1}{2\pi R} \int_{0}^{2\pi R} \int_0^T  ((\lambda^\epsilon_R)_t(x))^2|(\dot{\varphi}^\epsilon_R)_t(x)|^2+ |(\dot{\lambda}_R^\epsilon)_t(x)|^2\,\,\ed t  \, \ed {x} \\
&=\frac{1}{2\pi} \int_{0}^{2\pi} \int_0^T R^2 (\lambda^\epsilon_t(\theta))^2|\dot{\varphi}^\epsilon_t(\theta)|^2+ |\dot{\lambda}_t^\epsilon(\theta)|^2\,\ed t  \, \ed \theta.
\end{aligned}
\end{equation}
\end{proof}
\begin{remark} The flow described in lemma \ref{lem:collision} coincides with a linear peakon/anti-peakon solution of the Hunter-Saxton equation where the momentum is in the form of equation \eqref{eq:peakonmomentum} and the two peak trajectories are given by
\begin{equation}
\theta_1^\epsilon(t) = \frac{\pi}{2} \left(1+ \sin \left( \frac{\pi t}{2 (T+\epsilon)} \right) \right)\,,\quad \theta_2^\epsilon(t) = \frac{\pi}{2} \left(3- \sin \left( \frac{\pi t}{2 (T+\epsilon)} \right) \right)\,.
\end{equation}
\corr{The reason why we consider solutions  Hunter-Saxton rather than CH peakons is due to the fact that as $R\rightarrow 0$
 the action in \eqref{eq:polaraction} tends to the $\dot{H}^1$ action. }
\end{remark}

In figure \ref{fig:shockflow}, we give an illustration of the flow defined in equation \eqref{eq:shock} for fixed $\epsilon$ both in terms of particle trajectories and as a measure on the cone for $R=1$ (in which case the cone can be identified with $\mathbb{R}^2$). Note that at collision time the trajectories of particles between the peaks reach the apex of the cone.

\corr{In the next lemma we construct a flow using $n$ linear peakon/anti-peakon collisions that converges as $n\rightarrow+\infty$ to a measure in the same form as the one in lemma \ref{lem:diameter}. 
}

\begin{lemma}\label{lem:mcollision}
Let $\varphi^\epsilon:[0,T] \times S^1_R \rightarrow S^1_R$ the flow in lemma \ref{lem:collision} and  for each $n\in \mathbb{N}$ let $\hat{\varphi}^n:[0,T] \times S^1_R \rightarrow S^1_R$ be defined by
\begin{equation}\label{eq:hatflowdef}
\hat{\varphi}^n(\theta) \coloneqq \frac{2\pi}{n} \left\lfloor \frac{\theta n}{2\pi }\right\rfloor+ \frac{1}{n}\,\varphi^{\epsilon_n}\left( n\theta -2\pi \left\lfloor \frac{\theta n}{2\pi }\right\rfloor \right) 
\end{equation}
with $\epsilon_n$ being any positive sequence such that $\epsilon_n\rightarrow 0$. Then $\hat{\bs{\mu}}_n \coloneqq (\hat{\varphi}^n, \sqrt{\mr{Jac}(\hat{\varphi}^n)})_\# \rho_0 \rightharpoonup \hat{\bs{\mu}}^*$, 
where $\hat{\bs{\mu}}^*$ is defined by
\begin{equation}\label{eq:nondetnorot}
\hat{\bs{\mu}}^* = \frac{1}{2}(\mr{Id},\zeta^0)_\# \rho_0 + \frac{1}{2}(\mr{Id},\zeta^1)_\# \rho_0\,,
\end{equation}
with
\begin{equation}
\zeta^0_t(\theta) = \sqrt{2}\sin\left(\frac{\pi t}{4 T}+\frac{\pi}{4} \right)\,, \quad \zeta^1_t (\theta) = \sqrt{2} \cos\left(\frac{\pi t}{4 T}+\frac{\pi}{4} \right)  \,.
\end{equation}
Moreover, $\mc{A}(\hat{\bs{\mu}}_n) \rightarrow \mc{A}(\hat{\bs{\mu}}^*) = {\pi^2}/(16T)$.
\end{lemma}

\begin{proof}
For simplicity, we prove the result for $R=1$ but the argument presented here applies for any $R>0$.
Let $\mc{F}$ be any bounded Lipschitz functional on $\Omega$ with Lipschitz constant $L$. We need to check that 
\begin{equation}
\lim_{n\to+\infty} \int_\Omega \mc{F}(\mr{z}) \, \ed \hat{\bs{\mu}}_n(\mr{z}) = \int_\Omega \mc{F}(\mr{z}) \, \ed \hat{\bs{\mu}}^*(\mr{z}) \,.
\end{equation}
Denoting by $\hat{\lambda}^n= \sqrt{\mr{Jac}(\hat{\varphi}^n)}$ and by $\hat{\lambda}^{\epsilon_n}= \sqrt{\mr{Jac}(\hat{\varphi}^{\epsilon_n})}$, we observe that 
\begin{equation}\label{eq:decomposition1}
\begin{aligned}
\int_\Omega \mc{F}(\mr{z}) \, \ed \hat{\bs{\mu}}_n(\mr{z}) &= \frac{1}{2\pi} \int_{0}^{2\pi}  \mc{F}([\hat{\varphi}^n(\theta), \hat{\lambda}^n(\theta)] ) \,\ed \theta\\
&= \frac{1}{2\pi} \sum_{i=0}^{n-1} \int_{0}^{2\pi/n}  \mc{F}\left(\left[\frac{2\pi i}{n} + \frac{1}{n}\,\varphi^{\epsilon_n}( n\theta), {\lambda}^{\epsilon_n}(n\theta)\right] \right) \,\ed \theta\,,
\end{aligned}
\end{equation}
and similarly,
\begin{equation}\label{eq:decomposition2}
\int_\Omega \mc{F}(\mr{z}) \, \ed \hat{\bs{\mu}}^*(\mr{z}) = \frac{1}{4\pi} \sum_{i=0}^{n-1} \int_{0}^{2\pi/n}  \mc{F}\left(\left[\frac{2\pi i}{n} + \theta, {\zeta}^0(\theta) \right] \right) + \mc{F}\left(\left[\frac{2\pi i}{n} + \theta, {\zeta}^1(\theta) \right] \right) \,\ed \theta\,.
\end{equation}
We consider separately each integral in the sums in equation \eqref{eq:decomposition1} and \eqref{eq:decomposition2}. Rescaling the integrals in $\theta$ and using Lipschitz continuity of $\mc{F}$, we observe that the result is proven if   
\begin{equation}\label{eq:sumreform}
\lim_{n\to+\infty} \frac{1}{n} \sum_{i=0}^{n-1} \left| \frac{1}{2\pi} \int_{0}^{2\pi} \mc{F}\left(\left[\frac{2\pi i}{n} + \frac{1}{n}\,\varphi^{\epsilon_n}( \theta), {\lambda}^{\epsilon_n}(\theta)\right] \right) \,\ed \theta  - I_i^n\right| =0\,,
\end{equation}
where
\begin{equation}
I_i^n= \frac{1}{2} \mc{F}\left(\left[\frac{2\pi i}{n} , {\zeta}^0\left(\frac{2\pi i}{n}\right) \right] \right)+ \frac{1}{2} \mc{F}\left(\left[\frac{2\pi i}{n} , {\zeta}^1\left(\frac{2\pi i}{n}\right) \right] \right) \,.
\end{equation}
For any fixed sufficiently large $n$, we need to provide an appropriate bound for each term in the sum in equation \eqref{eq:sumreform}.
For any integer $i$ with $0 \leq i \leq n-1$, we have
\begin{equation}
\begin{aligned}
E^{i,n} & \coloneqq \left| \frac{1}{2\pi} \int_{0}^{2\pi} \mc{F}\left(\left[\frac{2\pi i}{n} +\frac{1}{n}\,\varphi^{\epsilon_n}( \theta), {\lambda}^{\epsilon_n}(\theta)
\right] \right) \,\ed \theta  - I_i^n\right|\\
 & \leq \frac{1}{2\pi}  \left| \int_{0}^{2\pi} \mc{F}\left(\left[\frac{2\pi i}{n} +\frac{1}{n}\,\varphi^{\epsilon_n}( \theta), {\lambda}^{\epsilon_n}(\theta) \right] \right) \,\ed \theta  -\int_{0}^{2\pi} \mc{F}\left(\left[\frac{2\pi i}{n}, {\lambda}^{\epsilon_n}(\theta) \right] \right) \,\ed \theta \right|\\
 &\quad +   \left|\frac{1}{2\pi}\int_{0}^{2\pi} \mc{F}\left(\left[\frac{2\pi i}{n}, {\lambda}^{\epsilon_n}(\theta) \right] \right) \,\ed \theta   - I_0^n\right| \coloneqq E^n_0+E^n_1\,.
\end{aligned}
\end{equation} 
Observe that for $\alpha\in[0,\pi/2]$, $\sqrt{1-\sin(\alpha)} = \sqrt{2} \cos(\alpha/2+\pi/4)$ and $\sqrt{1+\sin(\alpha)} = \sqrt{2} \sin(\alpha/2+\pi/4)$ therefore
\begin{equation}
{\lambda}^{\epsilon_n}(\theta) = \sqrt{\partial_\theta \varphi^{\epsilon_n}_t(\theta)} = 
\left\{
\begin{array}{ll}
\sqrt{2} \cos\left( \frac{\pi t}{4 (T+\epsilon_n)} + \frac{\pi}{4} \right) & \text{if}~ \frac{\pi}{2}<\theta<\frac{3\pi}{2}\,, \\
\sqrt{2} \sin\left( \frac{\pi t}{4 (T+\epsilon_n)} + \frac{\pi}{4} \right) & \text{otherwise}\,.
\end{array}
\right. 
\end{equation}
Since ${\lambda}^{\epsilon_n}$ is piecewise constant in $\theta$ we can write
\begin{equation}
\frac{1}{2\pi}\int_{0}^{2\pi} \mc{F}\left(\left[\frac{2\pi i}{n} , {\lambda}^{\epsilon_n}(\theta) \right] \right) \ed \theta = \frac{1}{2}  \mc{F}\left(\left[\frac{2\pi i}{n}, {\lambda}^{\epsilon_n}(0) \right] \right) +   \frac{1}{2} \mc{F}\left(\left[\frac{2\pi i}{n}, {\lambda}^{\epsilon_n}\left( {\pi}\right) \right] \right)\,.
\end{equation}
Comparing the expression for $\zeta^0$ and $\zeta^1$ with that of $\lambda^{\epsilon_n}$ and using the fact that $\mc{F}$ is Lipschitz we obtain
$ E^n_1 \leq C(\epsilon_n)$, where $C(\epsilon_n)>0$ is a constant depending on $\epsilon_n$ and $L$ such that $C(\epsilon_n)\rightarrow 0$ as $n\rightarrow +\infty$. A similar argument holds for $E^n_0$ and therefore we can find a constant $C_n$ independent of $i$ such that $E^{i,n}\leq C_n$ and $C_n\rightarrow 0$ as $n\rightarrow +\infty$. This implies equation \eqref{eq:sumreform}.

Finally, convergence of the action is a consequence of lemma \ref{lem:collision}. In particular, it is immediate to verify that $\mc{A}(\hat{\bs{\mu}}^*) = {\pi^2}/(16T)$. Moreover, by the same reasoning as in the proof of lemma \ref{lem:collision} and the change of variables in equation \eqref{eq:decomposition1} we obtain that the action $\mc{A}(\hat{\bs{\mu}}_n)$  is given by
\begin{equation}
\begin{aligned}
\frac{1}{2\pi} \int_{0}^{2\pi} \mc{A}([\hat{\varphi}^n(\theta),\hat{\lambda}^n(\theta)])\, \ed \theta 
&= \frac{1}{2\pi} \sum_{i=0}^{n-1} \int_{0}^{2\pi/n}  \mc{A}\left(\left[\frac{1}{n}\,{\varphi}^{\epsilon_n}( n\theta), {\lambda}^{\epsilon_n}(n\theta)\right] \right) \,\ed \theta\\
&= \frac{1}{2\pi} \int_{0}^{2\pi}  \mc{A}\left(\left[\frac{1}{n}\,{\varphi}^{\epsilon_n}( n\theta), {\lambda}^{\epsilon_n}(n\theta)\right] \right) \,\ed \theta\,.
\end{aligned}
\end{equation}
Therefore, the limit of $\mc{A}(\hat{\bs{\mu}}_n)$ for $n\rightarrow+\infty$ is the same to that in equation \eqref{eq:actionlimit}.
\end{proof}

In figure \ref{fig:shockflowm}, we give an illustration of the flow defined in equation \eqref{eq:hatflowdef} for fixed $n$ both in terms of particle trajectories and as a measure on the cone for $R=1$. It can be seen that convergence towards the measure $\bs{\mu}^*$ defined in lemma \ref{lem:mcollision} is due to the appearance of fast oscillations in the Jacobian together with the fact that particles tend to stay still as $n\rightarrow +\infty$. 

We can use the flows defined in lemma \ref{lem:mcollision} to construct a sequence that converges to the generalized flow $\bs{\mu}^*$ in theorem \ref{th:rots1} but where no rotation occurs. The construction consists in concatenating in time the flows in lemma \ref{lem:mcollision} so that a small portion of the domain stretches and then return to its original size. This is shown in figure \ref{fig:shockflowmm}. The convergence result is stated explicitly in the following proposition.

\begin{proposition}\label{prop:convergences1}
Let $\hat{\varphi}^n:[0,T] \times S^1_R \rightarrow S^1_R$ be the sequence defined in lemma \ref{lem:mcollision} and for each $n\in \mathbb{N}$ 
let ${\varphi}^n:[0,T] \times S^1_R \rightarrow S^1_R$ be defined by ${\varphi}^n_t = \varphi^n_{T-t}$ and
\begin{equation}\label{eq:flowseqs1}
{\varphi}^n_t(\theta) = 
\left\{
\begin{array}{ll}
\hat{\varphi}^n_{T-4t}((\hat{\varphi}^n_{T})^{-1}(\theta)) & \text{if}~t\leq \frac{T}{4}\,,\\
\hat{\varphi}^n_{4t-T}\left((\hat{\varphi}^n_{T})^{-1}(\theta)+ \frac{\pi}{n}\right)- \frac{\pi}{n} & \text{if}~\frac{T}{4}< t\leq \frac{T}{2}\,.
\end{array}
\right.
\end{equation} 
Then ${\bs{\mu}}_n \coloneqq ({\varphi}^n, \sqrt{\mr{Jac}({\varphi}^n)})_\# \rho_0$ can be rescaled to a sequence $\tilde{\bs{\mu}}_n\rightharpoonup {\bs{\mu}}^*$,
where ${\bs{\mu}}^*$ is defined as in equation \eqref{eq:nondetnorot} with
\begin{equation}
\zeta^0_t(\theta) = \sqrt{2}\sin\left(\frac{\pi t}{T}\right)\,, \quad \zeta^1_t (\theta) = \sqrt{2}\left| \cos\left(\frac{\pi t}{T}\right) \right| \,.
\end{equation}
Moreover, $\mc{A}({\bs{\mu}}_n) \rightarrow \mc{A}({\bs{\mu}}^*) = {\pi^2}/{T}$.
\end{proposition}
\begin{proof}
The rescaling to be performed in order to obtain the sequence $\tilde{\bs{\mu}}_n$ is given by 
\begin{equation}\label{eq:resct4}
\tilde{\bs{\mu}}_n =\mr{dil}_{r_{T/4},2} {\bs{\mu}}_n\,.
\end{equation}
In fact, by lemma \ref{lem:rescaling},  $\tilde{\bs{\mu}}_n$ is concentrated on paths such that $r_{T/4}=1$. Then, the result can be deduced from lemmas \ref{lem:mcollision} and \ref{lem:collision}. 
\end{proof}

\begin{remark}\label{rem:regularitys1}
The maps defined by equation \eqref{eq:flowseqs1} are piecewise smooth in space since their Jacobian is piecewise constant with a finite number of discontinuities. However, using a regularization argument, it is not difficult to construct a sequence of smooth diffeomorphisms satisfying proposition \ref{prop:convergences1}. For this it is sufficient to repeat the construction above using a regularized version of the linear peakon/anti-peakon collision, which can be defined by convolution of the flow map with a sequence of positive symmetric mollifiers.
\end{remark}

\begin{figure}
\centering
    \input{shock}
\caption{Velocity field evolution for the linear peakon/anti-peakon solution of the Hunter-Saxton equation.}\label{fig:shock}
\end{figure}
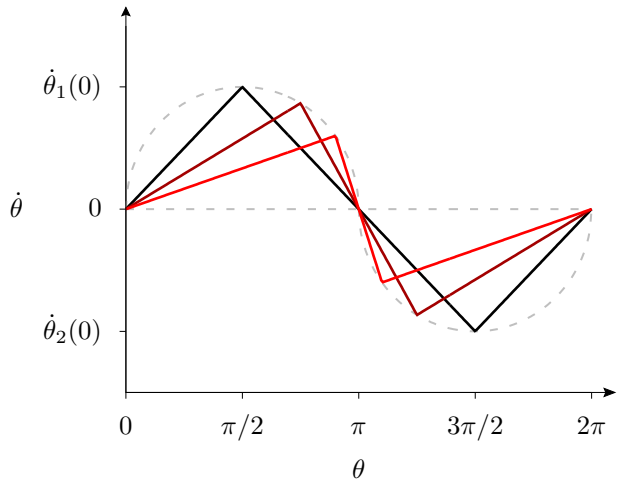

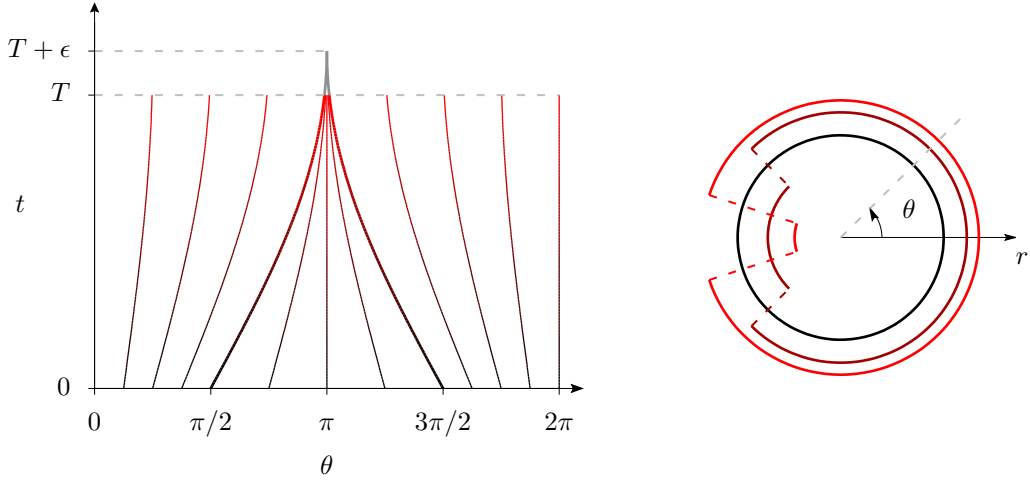
\begin{figure}
\centering
    \input{shockflow}
    \input{shockcone}
\caption{Particle trajectories $t\mapsto \varphi^\epsilon_t(\theta)$ for the linear peakon/anti-peakon solution (left) and support of fixed time marginals for the measure $(\varphi^\epsilon,\sqrt{\mr{Jac}(\varphi^\epsilon)})_\# \rho_0$ (right).}\label{fig:shockflow}
\end{figure}

\begin{figure}
\centering
    \input{shockflowm}    
    \input{shockconem}
\caption{Particle trajectories $t\mapsto\hat{\varphi}^n_t(\theta)$ relative to the map constructed in proposition \ref{lem:mcollision}  (left) and support of fixed time marginals for the measure $(\hat{\varphi}^n,\sqrt{\mr{Jac}(\hat{\varphi}^n)})_\# \rho_0$ (right), for $n=5$.}\label{fig:shockflowm}
\end{figure}
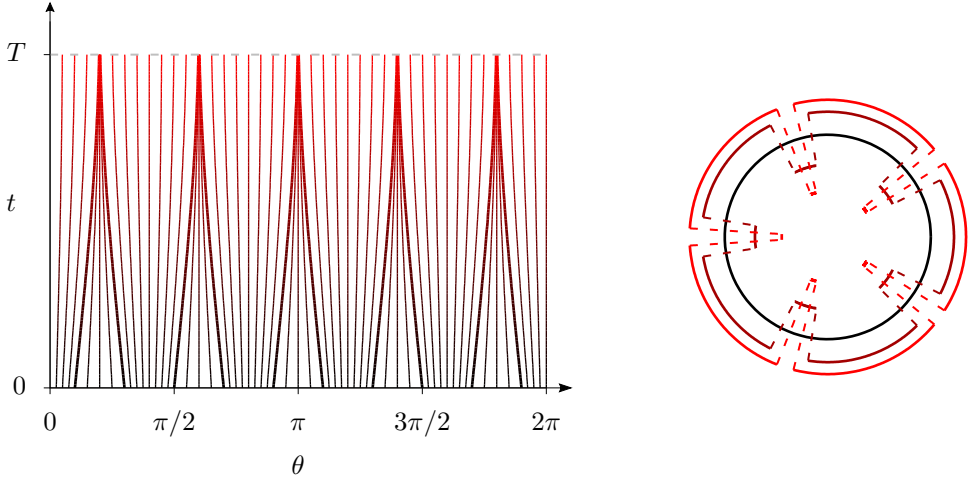

\begin{figure}
\centering
    \input{shockflowmm}  
\caption{Particle trajectories $t\mapsto{\varphi}^n_t(\theta)$ relative to the map constructed in proposition \ref{prop:convergences1} for $n=5$.}\label{fig:shockflowmm}
\end{figure}
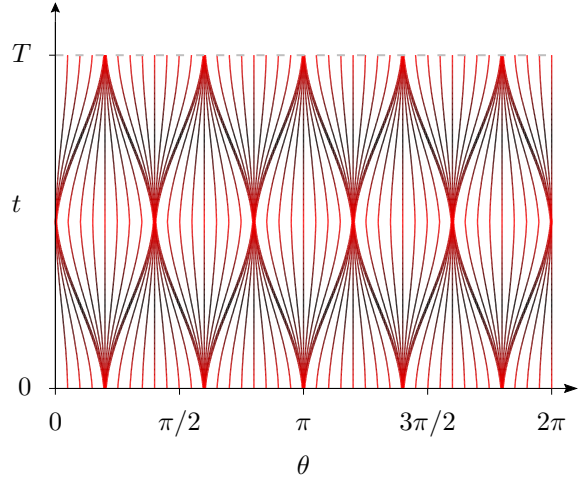

\subsection{Construction of a generalized solution on the torus}
We now consider the generalized $H(\mr{div})$ geodesic problem on the torus $T^2_{1,R} \coloneqq S^1_1 \times S^1_R$, with one of the two radii set to one. We consider as boundary condition a uniform rotation on the torus in which each particle rotates of half of the length on both circles. For this specific boundary condition we can construct a generalized minimizer using the construction of the previous section which realizes smaller action than \corr{any deterministic smooth flow}.  

\corr{
In the following lemma we provide a lower bound on the action associated with a deterministic minimizer.
\begin{lemma}\label{lem:lowbound2d}
Suppose that the smooth curve $t\in[0,T] \mapsto \varphi_t^* \in \mr{Diff}(T^2_{1,R})$ is a minimizer for the deterministic $H(\mr{div})$ geodesic problem \ref{prob:det}, with $\varphi_0 = \mr{Id}$, $\varphi_T = \fmap$ and where $\fmap$ is given in polar coordinates by $\fmap: (\theta,\phi) \in \mathbb{R}^2/({2\pi}\mathbb{Z})^2 \rightarrow (\theta+\pi,\phi+\pi) \in \mathbb{R}^2/({2\pi}\mathbb{Z})^2 $. Denote by $\rho_0 = (2\pi)^{-2}\ed \theta \ed \phi$ the normalized Lebesgue measure on the torus.  Then,
\begin{equation}
\int_{T^2_{1,R}} \mc{A}([\varphi^*,\sqrt{\mr{Jac}(\varphi^*)}])\, \ed \rho_0 \geq \frac{\tanh (2 \pi \sqrt{1+R^2})}{ 2T} \pi^2  \sqrt{1+R^2}
\,.
\end{equation} 
\end{lemma}
\begin{proof}
By an appropriate change of coordinates, it is sufficient to show the result on $T^2_{R_1,R_2}$ with 
\begin{equation}
R_1 = \frac{R}{\sqrt{1+R^2}} \,, \quad R_2 = \sqrt{1+R^2}\,,
\end{equation}
and $\fmap: (\theta,\phi) \in \mathbb{R}^2/({2\pi}\mathbb{Z})^2 \rightarrow (\theta,\phi+\pi) \in \mathbb{R}^2/({2\pi}\mathbb{Z})^2$. 
Let $t\in [0,T] \mapsto \varphi^*_t=(\varphi^\theta_t,\varphi^\phi_t) \in \mr{Diff}(T^2_{R_1,R_2})$ be a smooth minimizer for these boundary conditions. Define the flow $t\in [0,T] \mapsto \tilde{\varphi}_t \in \mr{Diff}(T^2_{R_1,R_2})$ by
\begin{equation}
\tilde{\varphi}_t(\theta,\phi) = \left \{
\begin{array}{ll}
(\varphi^\theta_t(2 \theta,\phi)/2,\varphi^\phi_t(2\theta,\phi)) & \text{if } 0<\theta\leq \pi\,,\\
(\varphi^\theta_t(2 \theta,\phi)/2 +\pi,\varphi^\phi_t(2\theta,\phi)) & \text{if } \pi<\theta\leq 2\pi\,.
\end{array}
\right.
\end{equation}
Then, by direct computation,
\begin{equation}
\int_{T^2_{1,R}} \mc{A}([\tilde\varphi,\sqrt{\mr{Jac}(\tilde\varphi)}])\, \ed \rho_0 \leq \int_{T^2_{1,R}} \mc{A}([\varphi^*,\sqrt{\mr{Jac}(\varphi^*)}])\, \ed \rho_0 \,;
\end{equation}
moreover, the inequality is strict unless $\dot{\varphi}^\theta_t = 0$ for all $t \in(0,T)$. Since $\varphi^*$ is a minimizer, we conclude that we must have $\varphi^\theta_t = \theta$ for all $t \in [0,T]$. Then, we obtain the result applying lemma \ref{lem:lowbound1d}.
\end{proof}
}

\begin{theorem}\label{th:rott2}
Consider the generalized $H(\mr{div})$ geodesic problem on $T^2_{1,R}$ with coupling constraint given by uniform rotation on both circles by half of the circles length, i.e.\ in polar coordinates $\fmap: (\theta,\phi) \in \mathbb{R}^2/({2\pi}\mathbb{Z})^2 \rightarrow (\theta+\pi,\phi+\pi) \in \mathbb{R}^2/({2\pi}\mathbb{Z})^2 $. Denote by $\rho_0 = (2\pi)^{-2}\ed \theta \ed \phi$ the normalized Lebesgue measure on the torus. Then,
the dynamic plan $\bs{\mu}^*$ in lemma \ref{lem:diameter}, i.e.\ equation \eqref{eq:nondetrot} with
\begin{equation}\label{eq:nondetrot2}
\zeta^0_t(\theta,\phi) = \sqrt{2}\sin(\sqrt{P^*}t)\,, \quad \zeta^1_t (\theta,\phi) = \sqrt{2}| \cos(\sqrt{P^*} t) | \,, \quad 
\psi^1_t(\theta,\phi) = \left\{
\begin{array}{ll}
(\theta,\phi) & t\leq T/2\,,\\
(\theta+\pi,\phi+\pi) & t> T/2\,,
\end{array}
\right.
\end{equation}
where  $P^* = (\pi/T)^2$, is a minimizer, whereas the constant speed rotation is not a minimizer. \corr{Moreover, if $R$ is sufficiently large no smooth flow can be a minimizer}.
\end{theorem}   
\begin{proof}
Consider the functional $\pi_\theta:\Omega(T^2_{1,R}) \rightarrow \Omega(S^1_{1})$ defined by
\begin{equation}
\pi_\theta(\mr{z}) \coloneqq ( t\in[0,T] \mapsto [\theta_t,r_t] \in \cone)\,,
\end{equation}
for any $\mr{z} = ( t\in[0,T] \mapsto [(\theta_t,\phi_t),r_t] \in \cone)$. In other words, $\pi_\theta$ applies at each time the canonical projection on the circle of unit radius. We observe that for any admissible dynamic plan $\bs{\mu}\in \mc{P}(\Omega(T^2_{1,R}))$ for the generalized $H(\mr{div})$ geodesic problem on the torus, 
\begin{equation}
\bs{\mu}_\theta \coloneqq \pi_{\theta\#} \bs{\mu} \in \mc{P}(\Omega(S^1_1))
\end{equation}
 is admissible for the generalized $H(\mr{div})$ geodesic problem on $S^1_1$ with boundary conditions associated with the map $\fmap_\theta: \theta \in \mathbb{R}/{2\pi}\mathbb{Z} \rightarrow \theta+\pi$. In fact, if for example $\bs{\mu}$ satisfies the homogeneous marginal constraint with respect to the normalized measure $(2\pi)^{-2} \ed \theta \ed \phi$, then also $\bs{\mu}_\theta$ satisfies the same constraint since for any $t\in[0,T]$ and $f\in C^0(S^1_1)$,
\begin{equation}
\int_{\Omega(S_1^1)} f(\theta_t) r_t^2\, \ed \bs{\mu}_\theta(\mr{z}) = \int_{\Omega(T_{1,R}^2)} f(\theta_t) r_t^2\, \ed \bs{\mu}(\mr{z}) =  \frac{1}{2\pi}\int_{S^1_1} f(\theta) \, \ed \theta\,, 
\end{equation}
and similarly for the coupling constraint. The problem on $S^1_1$ admits a non-deterministic minimizer, which was given in theorem \ref{th:rots1} and we denote it by $\bs{\mu}^*_\theta$. Then, we have for any admissible $\bs{\mu} \in \mc{P}(\Omega(T^2_{1,R}))$,
\begin{equation}
\mc{A}(\bs{\mu}) \geq \mc{A}(\bs{\mu}_\theta) \geq \mc{A}(\bs{\mu}^*_\theta) = \frac{\pi^2}{T} \,.
\end{equation}
However, by lemma \ref{lem:diameter}, the dynamic plan $\bs{\mu}^*$ defined by equation \eqref{eq:nondetrot2} satisfies $\mc{A}(\bs{\mu}^*)=  \mc{A}(\bs{\mu}^*_\theta)$ and so it must be a minimizer. On the other hand, the action for constant speed rotation is given by $\mc{A}^{rot} = \pi^2 (R^2+1)/{T}$ and therefore such a solution cannot be a minimizer since $R>0$. \corr{Similarly, if $R$ is sufficiently large, by lemma \ref{lem:lowbound2d}, no smooth minimizer can exist, since otherwise its action would be strictly larger than $\pi^2/T$.}
\end{proof}

\subsection{Approximation of a generalized minimizer on the torus}

The generalized minimizer in theorem \ref{th:rott2} is very similar to its one-dimensional counterpart of theorem \ref{th:rots1}. Importantly, however, the extra dimension gives us enough flexibility to produce deterministic approximations, which is the main result of this section. Such approximations will be similar in spirit to those presented in the one-dimensional case. In brief, using again peakon/anti-peakon collisions we will be able to reach the final configuration by moving two complementary subsets of the domain at different times, when they occupy a small volume.     

\begin{theorem}\label{th:approxt2}
Let ${\bs{\mu}}^*$ and $\fmap$ be the minimizer and the coupling, respectively, defined in theorem \ref{th:rott2} on the torus $M = T^2_{1,R}$. There exists a sequence of continuous flow maps ${\varphi}^n:[0,T] \times M \rightarrow M$, $n\in \mathbb{N}$, such that for every $t\in [0,T]$, ${\varphi}_t^n:M \rightarrow M$ is smooth almost everywhere, and
\begin{itemize}
\item for all $n\in \mathbb{N}$, ${\varphi}_0^n = \mr{Id}$ and ${\varphi}_T^n = h$;
\item the sequence ${\bs{\mu}}_n \coloneqq ({\varphi}^n, \sqrt{\mr{Jac}({\varphi}^n)})_\# \rho_0$ can be rescaled to a sequence $\tilde{\bs{\mu}}_n\rightharpoonup {\bs{\mu}}^*$;
\item $\mc{A}(\tilde{\bs{\mu}}_n) \rightarrow \mc{A}({\bs{\mu}}^*)$.
\end{itemize}
\end{theorem}
\begin{proof}
For simplicity, we prove the result for $R=1$ but the argument presented here applies for any $R>0$.
In addition, performing an appropriate change of variables, one can easily verify that it is sufficient to prove the theorem with $\fmap:(\theta,\phi) \in \mathbb{R}^2/({2\pi}\mathbb{Z})^2 \rightarrow (\theta,\phi+\pi)$  and $\bs{\mu}^*$ defined as in equation \eqref{eq:nondetrot2}, but with $\psi^1$ defined by
\begin{equation}
\psi^1_t(\theta,\phi) = \left\{
\begin{array}{ll}
(\theta,\phi) & t\leq T/2\,,\\
(\theta,\phi+\pi) & t> T/2\,.
\end{array}
\right.
\end{equation}
For each $n\in\mathbb{N}$, the map $\varphi^n$ will be constructed using two basic flows. The first is defined as follows. Fix a sequence $\epsilon_n = \epsilon_0/n^3$, $n\in \mathbb{N}$, where $\epsilon_0$ is a sufficiently small constant. 
Moreover, for any $\epsilon>0$ consider the set $B_\epsilon \subset S^1_1$ defined by
\begin{equation}
B_\epsilon \coloneqq \bigcup_{i=0}^{n-1} \left[\frac{\pi }{n}(2i+1)-\frac{\epsilon}{2},\frac{\pi }{n}(2i+1)+\frac{\epsilon}{2}\right]\,,
\end{equation}
and let $\phi_{rot}^n: S^1_1\rightarrow S^1_1$ such that $0\leq \phi_{rot}^n\leq \pi$, $\phi_{rot}^n(\theta) = \pi$ for all $\theta\in B_{\epsilon_n}$ and $\phi_{rot}^n(\theta) = 0$ for all $\theta\in S^1_1\setminus B_{2\epsilon_n}$. 
For $k=0,1$, we let $\varphi^{k,n}_{rot}:[0,\sqrt{\epsilon_n}] \times T^2_{1,1} \rightarrow T^2_{1,1}$ be the flow defined by 
\begin{equation}
(\varphi^{0,n}_{rot})_t(\theta,\phi) \coloneqq 
\left(\theta, \phi + \frac{t}{\sqrt{\epsilon_n}} \phi^n_{rot}(\theta)\right) \,,\, (\varphi^{1,n}_{rot})_t(\theta,\phi) \coloneqq 
\left(\theta, \phi + \frac{t}{\sqrt{\epsilon_n}} (\pi- \phi^n_{rot}(\theta))\right)\,.
\end{equation}

Consider now the flow $\hat{\varphi}^n$ defined in equation \eqref{eq:hatflowdef}, with $\epsilon_n$ defined as above. With a slight abuse of notation, we will also denote by $\hat{\varphi}^n$ its canonical extension to the torus which leaves the $\phi$ coordinate fixed. Moreover, for any $\alpha\in\mathbb{R}/{2\pi}\mathbb{Z}$ denote by $R^\theta_{\alpha}:T^2_{1,1} \rightarrow T^2_{1,1}$ the map $R_\alpha^\theta(\theta,\phi)\coloneqq (\theta+\alpha,\phi)$. Then, we define the flow ${\varphi}^{0,n}_{exp}:[\sqrt{\epsilon_n}, T/2]\times  T^2_{1,1} \rightarrow T^2_{1,1}$ by
\begin{equation}
({\varphi}^{0,n}_{exp})_t(\theta,\phi) \coloneqq
\left\{
\begin{array}{ll}
\hat{\varphi}^n_{a_n(t)}( (\hat{\varphi}^n_{T})^{-1}(\theta,\phi)) & \text{if}~t\leq \frac{T}{4}\,,\\
R^\theta_{-\pi/n}\circ\hat{\varphi}^n_{4t-T}\left(R^\theta_{\pi/n}\circ(\hat{\varphi}^n_{T})^{-1}(\theta,\phi)\right) & \text{if}~\frac{T}{4}< t\leq \frac{T}{2}\,,
\end{array}
\right.
\end{equation} 
where $a_n(t) \coloneqq  T(T-4t)(T-4\sqrt{\epsilon_n})^{-1}$. Note that setting $\epsilon_n=0$ this flow coincides with the canonical extension to the torus of the flow in equation \eqref{eq:flowseqs1}. Similarly, 
\begin{equation}
{\varphi}^{1,n}_{exp} \coloneqq R^\theta_{-\pi/n}\circ{\varphi}^{0,n}_{exp} \circ R^\theta_{\pi/n}.
\end{equation}

We construct the sequence $\varphi^n$ by glueing together the maps $\varphi^{k,n}_{rot}$ and ${\varphi}^{k,n}_{exp}$ so that for each $n\in \mathbb{N}$ the final flow consists of four stages: in the first, $n$ stripes of the domain rotate while the rest of the domain stays put as prescribed by $\varphi^{0,n}_{rot}$; in the second, the stripes expand up to a symmetric configuration in which the rest of the domain occupies stripes of the same size, as prescribed by ${\varphi}^{0,n}_{exp}$; in the third, the rest of the points rotate as prescribed by $\varphi^{1,n}_{rot}$; finally, we use ${\varphi}^{1,n}_{exp}$ to compress the stripes to their original size. More precisely,
\begin{equation}\label{eq:compflow}
{\varphi}^n_t \coloneqq
\left\{
\begin{array}{ll}
(\hat{\varphi}^{0,n}_{rot})_t & \text{if}~t\leq \sqrt{\epsilon_n}\,,\\
(\hat{\varphi}^{0,n}_{exp})_{t} \circ (\hat{\varphi}^{0,n}_{rot})_{\sqrt{\epsilon_n}} & \text{if}~\sqrt{\epsilon_n}<t\leq \frac{T}{2}\,,\\
(\hat{\varphi}^{1,n}_{rot})_{t-T/2}\circ (\hat{\varphi}^{0,n}_{exp})_{T/2} \circ (\hat{\varphi}^{0,n}_{rot})_{\sqrt{\epsilon_n}} & \text{if}~\frac{T}{2}<t\leq \frac{T}{2}+\sqrt{\epsilon_n}\,,\\
(\hat{\varphi}^{1,n}_{exp})_{t-T/2} \circ 
(\hat{\varphi}^{1,n}_{rot})_{\sqrt{\epsilon_n}}\circ 
(\hat{\varphi}^{0,n}_{exp})_{T/2} \circ 
(\hat{\varphi}^{0,n}_{rot})_{\sqrt{\epsilon_n}} & \text{if}~\frac{T}{2}+\sqrt{\epsilon_n}<t\leq T \,.\\
\end{array}
\right.
\end{equation} 
A graphical representation of this flow is given in figure \ref{fig:shockflow2d} for $n=1$ (so that we have only one stripe) and in the original coordinates (so that the boundary conditions are those associated with double rotation).

Note that the flow defined in equation \eqref{eq:compflow} is very similar to the one defined in proposition \ref{prop:convergences1}, whose canonical extension to the torus will be denoted by $\varphi^{0,n}$. As in proposition \ref{prop:convergences1}, we define again the rescaled measure $\tilde{\bs{\mu}}_n$ using equation \eqref{eq:resct4}.
This means that for any Lipschitz continuous bounded functional ${\mc F}:\Omega \rightarrow \mathbb{R}$,
\begin{equation}\label{eq:dilt4}
\int_\Omega \mc F(\mr{z})\,\ed \tilde{\bs{\mu}}_n(\mr{z}) = \frac{1}{4\pi^2} 
\int_{T^2_{1,1}} \mc F\left(\left[ {\varphi}^n\circ ({\varphi}^n_{T/4})^{-1}(\theta,\phi), \bar{\lambda}^n(\theta,\phi)   \right]\right) \ed  \theta\, \ed \phi\,,
\end{equation}
where
\begin{equation}
 \bar{\lambda}^n_t \coloneqq \left(\frac{\mr{Jac}(\varphi^n_t)}{\mr{Jac}(\varphi^n_{T/4})}\right)^{1/2}\circ({\varphi}^n_{T/4})^{-1} = \left(\mr{Jac}(\varphi^n_t\circ(\varphi^n_{T/4})^{-1})\right)^{1/2}\,.
\end{equation}
Note that equation \eqref{eq:dilt4} is a direct consequence of the definition of the dilation map and the change of variables formula.
Due to proposition \ref{prop:convergences1} and the way $\varphi^n$ is constructed, to prove the convergence $\tilde{\bs{\mu}}_n\rightharpoonup \bs{\mu}^*$, it is sufficient to focus on the interval $[0,T/2]$ and check that $I^n\rightarrow 0$, where
\begin{equation}
I^n \coloneqq \int_{T^2_{1,1}} \sup_{t\in[0,T/2]} d_{\cone}( [\varphi^n_t\circ({\varphi}^n_{T/4})^{-1},\bar{\lambda}^n_t], [\varphi^{0,n}_t\circ({\varphi}^{0,n}_{T/4})^{-1},\bar{\lambda}^{0,n}_t]) \,\ed \theta \, \ed\phi 
\end{equation}
and $\bar{\lambda}^{0,n}_t \coloneqq \left(\mr{Jac}(\varphi^n_t\circ(\varphi^{0,n}_{T/4})^{-1})\right)^{1/2}$.  Because of the similar structure of the flows $\varphi^n$ and $\varphi^{0,n}$, $I^n$ reduces to
\begin{equation}
I^n = \int_{T^2_{1,1}} \sup_{t\in[0,T/4]} d_{\cone}( [\varphi^n_t\circ({\varphi}^n_{T/4})^{-1},\bar{\lambda}^n_t], [\varphi^{0,n}_t\circ({\varphi}^{0,n}_{T/4})^{-1},\bar{\lambda}^{0,n}_t]) \,\ed \theta \, \ed\phi\,.
\end{equation}
Let $A_{\epsilon} \coloneqq \varphi^{0,n}_{T/4}(B_\epsilon \times S^1_1)$, for any $\epsilon>0$. We decompose $I^n = I^{0,n}+I^{1,n}$ where $I^{0,n}$ and $I^{1,n}$ are the integrals over $A_{2\epsilon_n}$ and $T^{2}_{1,1}\setminus A_{2\epsilon_n}$, respectively.
Define for $0\leq a<b\leq T/4$
\begin{equation}
I^{0,n}_{a,b} \coloneqq  \int_{A_{2\epsilon_n}} \sup_{t\in[a,b]} d_{\cone}( [\varphi^n_t\circ({\varphi}^n_{T/4})^{-1},\bar{\lambda}^n_t], [\varphi^{0,n}_t\circ({\varphi}^{0,n}_{T/4})^{-1},\bar{\lambda}^{0,n}_t]) \,\ed \theta \, \ed\phi\,.
\end{equation}
We have $I^n\leq  I^{0,n}_{0,\sqrt{\epsilon_n}} +I^{0,n}_{\sqrt{\epsilon_n},T/4}+I^{1,n}$. By continuity of the flow maps it is easy to verify that $I^{0,n}_{\sqrt{\epsilon_n},T/4}\rightarrow 0$ and $I^{1,n}\rightarrow 0$ as $n\rightarrow +\infty$. On the other hand, by construction $0< \bar{\lambda}^n,\bar{\lambda}^{0,n}<\sqrt{2}$. Therefore, by the triangular inequality
\begin{equation}
\begin{aligned}
I^{0,n}_{0,\sqrt{\epsilon_n}} &\leq \int_{A_{2\epsilon_n}} \sup_{t\in[0,\sqrt{\epsilon_n}]}( \bar{\lambda}^n_t +  \bar{\lambda}^{0,n}_{t}) \, \ed \theta \, \ed \phi\,,\\
&\leq 4\pi n\epsilon_n (2\sqrt{2}) + \int_{B_{\pi/n}\times S^1_1} \sup_{t\in[0,\sqrt{\epsilon_n}]}( \bar{\lambda}^n_t +  \bar{\lambda}^{0,n}_{t}) \, \ed \theta \, \ed \phi\,,\\
\end{aligned}
\end{equation} 
where on the second line, we decomposed the integral over the part of $A_{2\epsilon_n}$ that gets stretched and the part that gets compressed under $\varphi^n \circ (\varphi^n_{T/4})^{-1}$ for $t\in [0,\sqrt{\epsilon_n}]$.
In particular, the integrand in the second line tends to 0 as $n\rightarrow +\infty$, which yields $I^n \rightarrow 0$. A similar argument can be applied on the interval $[T/2, T]$, which proves that $\tilde{\bs{\mu}}_n \rightharpoonup{\bs{\mu}}^*$.

In order to prove convergence of the action, in view of lemma \ref{lem:collision}, it is sufficient to show
\begin{equation}
\int_{T^2_{1,1}} \mc{A}([\varphi^{k,n}_{rot}(\theta,\phi), 1])\,\ed\theta\,\ed \phi \rightarrow 0\,,
\end{equation}
for $k=0,1$, as $n\rightarrow+\infty$, where the action is computed over the time interval $[0,\sqrt{\epsilon_n}]$. For all $n\in \mathbb{N}$, under the flow $\varphi^{k,n}_{rot}$ only points with $\theta\in B_{2\epsilon_n}$ rotate with velocity bounded by $\pi/\sqrt{\epsilon_n}$; hence
\begin{equation}
\int_{T^2_{1,1}} \mc{A}([\varphi^{k,n}_{rot}(\theta,\phi), 1])\,\ed\theta\,\ed \phi \leq 2 \pi (n 2\epsilon_n) \frac{\pi^2}{\sqrt{\epsilon_n}} = \sqrt{\frac{\epsilon_0}{n}} 4\, \pi^3  \,,
\end{equation}
which concludes the proof.
\end{proof}

\begin{remark}\label{rem:regularityt2}
As for the one-dimensional case (see remark \ref{rem:regularitys1}), the maps defined by equation \eqref{eq:compflow} are piecewise smooth in space since their Jacobian is piecewise constant with a finite number of discontinuities. Also in this case, it is sufficient to repeat the construction above using a regularized version of the linear peakon/anti-peakon collision, to obtain a sequence of smooth diffeomorphisms satisfying theorem \ref{th:approxt2}.
\end{remark}

\corr{
\begin{remark}\label{re:shinerlman}
In theorem \ref{th:rots1} we proved that our relaxation is not tight on $S^1_R$ (for $R$ sufficiently large), whereas theorem \ref{th:rott2} suggests it is tight for $d\geq 2$. It should be noted that the situation is similar for the incompressible Euler equations. In fact, Shnirelman proved that Brenier's relaxation is not tight for $d = 2$ but it is tight when $d=3$, as in this case any generalized incompressible flows can be approximated using deterministic maps \cite{shnirelman1994generalized}. 
\end{remark}
}

\begin{figure}
\centering
    \input{shockflow2t0}\hspace{-0.5em}
    \input{shockflow2t1}\hspace{-0.5em} 
    \input{shockflow2t2}  
    \input{shockflow2t3}\hspace{-0.5em}
    \input{shockflow2t4}\hspace{-0.5em}  
    \input{shockflow2t5}  
    \input{shockflow2t6}\hspace{-0.5em}  
    \input{shockflow2t7}\hspace{-0.5em}  
    \input{shockflow2t8}    
\caption{Particle trajectories for the flow $\varphi^n$ in \eqref{eq:compflow} for $n=1$ (in appropriate coordinates to determine double rotation on the torus). We use different colors to label particles and follow their motion. The stripes indicate the particles in between the peakons' peaks. In the limit, the trajectories of such particles lifted to the cone will start and end at the apex.}\label{fig:shockflow2d}
\end{figure}
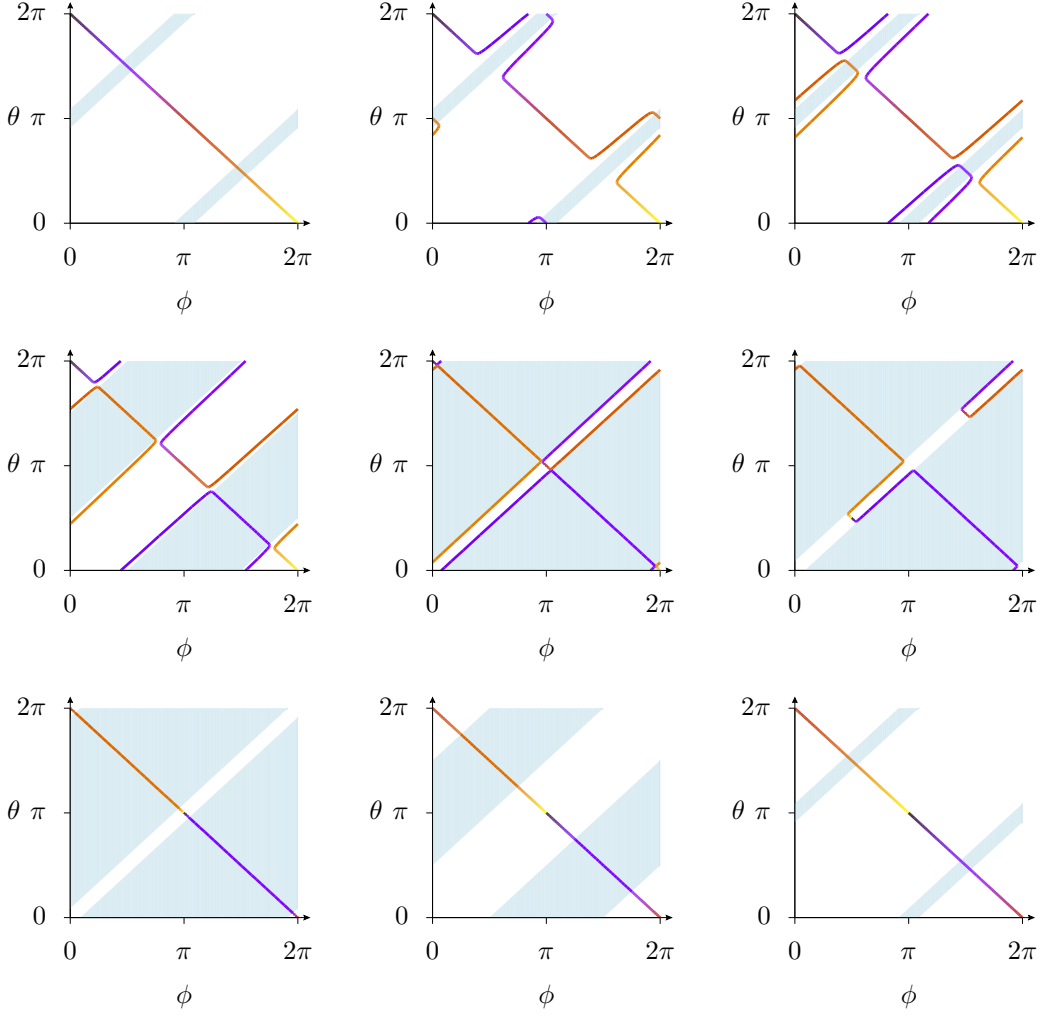
 
\section{Discrete generalized solutions} \label{sec:discrete}

There are two main obstacles in translating problem \ref{prob:generalizedunb} to the discrete setting. On one hand, we need to make computations on an unbounded domain; on the other, we need to be able to single out a representative for the equivalence class of minimizers with respect to rescaling. However, if one is interested in simulating solutions that are not singular (see definition \ref{def:singular}), it is appropriate to enforce the strong coupling constraint in \eqref{eq:coupling} instead of \eqref{eq:couplingconstr}. Hence, if we substitute $\cone$ by $\cone_R$ for a fixed $R>1$ and use the strong coupling constraint in the generalized $H(\mr{div})$ geodesic problem, we obtain a modified formulation that is able to reproduce a particular class of solutions, which includes all deterministic solutions with bounded Jacobian. In this section we describe a numerical algorithm based on entropy regularization and Sinkhorn algorithm that solves such a modified formulation. Our scheme is based on similar methods for the incompressible Euler equations developed in \cite{nenna2016numerical,benamou2017generalized,benamou2015iterative}.    
We also provide some numerical results illustrating the behavior of generalized $H(\mr{div})$ geodesics.

\subsection{Discrete formulation}
We set $M=[0,1]^d$ and consider a uniform discretization with points $\{x_i\}_{i=1}^{N_x}$, and a discretization of the interval $(0,R]$ with points $\{r_i\}_{i=1}^{N_r}$ such that $r_j = 1$ for a fixed $j\in\{1,\ldots, N_r\}$. These induce a discretization of the cone with points $\{z_i\}_{i=1}^N$ where $N= N_x N_r$. Similarly, we also consider a uniform discretization $\{t_i \}_{i=1}^K$ of $[0,T]$. Generalized flows are then replaced by a coupling arrays $\bs{\mu} \in (\mb{R}_{\geq 0}^N)^{K}$. Note that we can incorporate the boundary condition $\lambda_0 =1$ by reducing the dimension of $\bs{\mu}$. In particular, we now denote by $\pi_x$ and $\pi_r$ the canonical projections from $M\times(0,R]$ to $M$ and $(0,R]$ respectively. We use the same notation to indicate the maps $\pi_x : \{1,\ldots,N\}\rightarrow \{1,\ldots,N_x\}$ and $\pi_r: \{1,\ldots,N\}\rightarrow \{1,\ldots,N_r\}$ mapping directly the discretization indices.  
Then, we set for any $\{j_1,\ldots, j_K\} \in \{1, \ldots, N \}^K$,
\begin{equation}\label{eq:initial}
\bs{\mu}_{j_1,\ldots,j_K} =  \mathds{1}_{\{\pi_r({z_{j_1}}) = 1\}}  \tilde{\bs{\mu}}_{\pi_x(j_1),j_2,\ldots,j_K} \,,
\end{equation}
where $\mathds{1}$ is the indicator function and $\tilde{\bs{\mu}}\in \mb{R}_{\geq 0}^{N_x}\times({\mb{R}_{\geq 0}^N})^{K-1}$. We denote by $\Pi_0$ the set of couplings satisfying \eqref{eq:initial}.
The marginal at a given time $t_k$ is a discrete measure on $M \times(0,R]$. We denote this by $S_k(\bs{\mu}) \in \mb{R}_{\geq 0}^{N}$, and it is defined as follows: 
\begin{equation}
[S_k(\bs{\mu})]_{j} = \sum_{j_1, \ldots, j_{k-1},j_{k+1}, \ldots, j_{K}} \bs{\mu}_{j_1, \ldots, j_{k-1},j,j_{k+1}, \ldots, j_{K}} \,.
\end{equation}
We denote by $M_n: \mb{R}_{\geq 0}^{N} \rightarrow \mb{R}_{\geq 0}^{N_x}$ the $n$th moment taken in the radial direction, i.e.\
\begin{equation}
M_n[A]_i = \sum_{j, \pi_x(j) = i} \pi_r(z_j)^n A_{j}\,.
\end{equation}
Hence the constraint in \eqref{eq:marginals} becomes
\begin{equation}
M_2[S_k(\bs{\mu})]_i = 1/N_x\,.
\end{equation}
Moreover, we denote by $\Pi$ the set of admissible coupling arrays,
\begin{equation}
\begin{aligned}
\Pi = \{ \bs{\mu} \in \Pi_0 ;\, \forall i,\, M_2[S_k(\bs{\mu})]_i = 1/N_x \}\,.
\end{aligned}
\end{equation}
The constraint on the coupling between time 0 and $T$ can be enforced weakly by including it directly in the cost, which is given by the following array
\begin{equation}
C_{j_1,\ldots,j_K} = \frac{K-1}{T} \sum_{k=1}^{K-1}d_\cone(z_{j_k},z_{j_{k+1}})^2 + \alpha\, d_\cone(z_{j_K},(\fmap(\pi_x(z_{j_1})), \sqrt{|\mr{Jac}(\fmap)|}))^2\,,
\end{equation}
where $\alpha>0$ is a parameter. The regularized discrete problem is then,
\begin{equation}\label{eq:minimization}
\min_{\bs{\mu} \in \Pi}\, \langle C,\bs{\mu} \rangle  - \epsilon E(\bs{\mu})\,,
\end{equation}
where $\epsilon>0$ is another parameter and $E(\bs{\mu})$ is the entropy of the coupling defined by
\begin{equation}
E(\bs{\mu}) = -\langle\bs{\mu}, \log(\bs{\mu}) -1\rangle\,. 
\end{equation}

Problem \eqref{eq:minimization} can be solved by means of alternating projections which consist in enforcing recursively the marginal constraints at the different time levels. In particular, we consider the following augmented functional
\begin{equation}\label{eq:augmented}
\min_{\bs{\mu}}\, \langle C,\bs{\mu} \rangle  - \epsilon E(\bs{\mu}) - \sum_{i,k} p^k_i (M_2[S_k (\bs{\mu})]_i - 1/N_x)\,,
\end{equation} 
where $p^k\in \mb{R}^{N_x}$ for all $k\in\{1,\ldots,K\}$. From \eqref{eq:augmented} we obtain
\begin{equation}
\bs{\mu}_{j_1,\ldots,j_K} = e^{-\frac{C_{j_1,\ldots,j_K}}{\epsilon}} e^{ \sum_k p^k_{\pi_x(j_k)}   r_{\pi_r(j_k)}^2}\,.
\end{equation}
Enforcing the constraint at time level $n$ allows us to solve for $p^n$ given the set $\{p^k\}_{k\neq n}$. This amounts to solving the following nonlinear equation for all $i\in\{1,\dots, N_x\}$,
\begin{equation}
\sum_j B_{i,j} e^{ p^n_{i}   r_{j}^2} r_j^2 = 1/N_x\,	,
\end{equation}
where
\begin{equation}
B = S_n \left[ e^{-\frac{C_{j_1,\ldots,j_K}}{\epsilon}} e^{ \sum_{k,k\neq n} p^k_{\pi_x(j_k)}   r_{\pi_r(j_k)}^2}      \right ]\,.
\end{equation}
Due to the structure of the cost, we only need to store two arrays $D^0,D^1\in \mb{R}^N\times \mb{R}^N$, given by
\begin{equation}
D^0_{i,j} = d_\cone(z_{i},z_{j})^2\,, \qquad D^1_{i,j} = d_\cone(z_{i},(\fmap(\pi_x(z_{j})), \sqrt{|\mr{Jac}(\fmap)}|))^2\,.
\end{equation}

\subsection{Numerical results: from CH to Euler} \label{subsec:numerical}

We now present some numerical results illustrating the behavior of generalized solutions of the $H(\mr{div})$ geodesic problem and their relation to generalized incompressible Euler solutions. We consider two types of couplings to define the boundary conditions: a classical deterministic coupling, which we use to illustrate the emergence of discontinuities in the flow map, and a generalized coupling that obliges particles to cross each other so that the solution is not deterministic. For both cases, the domain will be the one-dimensional interval $M=[0,1]$ and $T=1$.
\\

\noindent \emph{A peakon-like solution}. Consider the continuous map $h:[0,1] \rightarrow [0,1]$, defined by 
\begin{equation}\label{eq:bcpeakon}
h(x) = \left \{
\begin{array}{ll}
1.4 \,x & \text{if } x\leq 0.5 \,,\\
0.6 \,x + 0.4 & \text{if } x> 0.5 \,.
\end{array}
 \right.
\end{equation} 
We use this map to define the coupling on the cone as in equation \eqref{eq:coupling}. We compute the solution using the algorithm presented in the previous section with $N_x=40$, $N_r =41$, $0.55\leq r \leq 1.45$, $K=35$, $\alpha=40$, $\epsilon = 5\cdot 10^{-4}$. In figure \ref{fig:peakon1} we show the evolution of the transport plan on the domain $M$ given by $(e^M_{0,t_k})_\#\bs{\mu} \in \mc{P}(M^2)$, where $e^M_{0,t_k}(\mr{z}) \coloneqq (x_0,x_{t_k})$, for selected times. In figure \ref{fig:peakon2} we show the evolution of the marginals on the cone given by $(e_{t_k})_\#\bs{\mu} \in \mc{P}(\cone)$ for the same times. 
We remark that the dynamic plan is approximately deterministic since there is very little diffusion of the mass in the domain, which is at least partially due to the entropic regularization. In addition the discontinuity in the Jacobian of the coupling map propagates to the whole solution, which resembles a peakon with the discontinuity point corresponding to the peak of the peakon. 
\\

\begin{figure}[htbp]
\centering
\subfigure[$k=1$]{\includegraphics[width=0.23\textwidth]{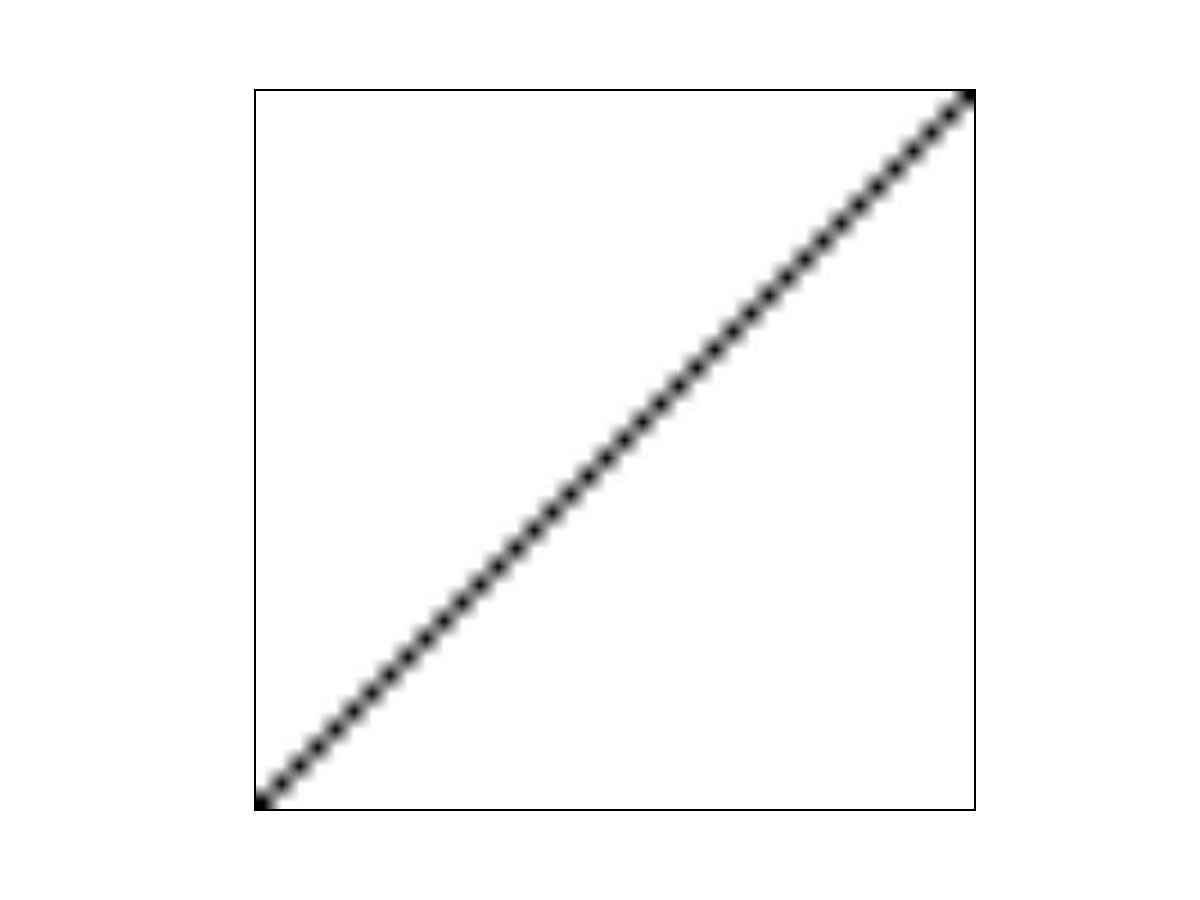}}
\subfigure[$k=6$]{\includegraphics[width=0.23\textwidth]{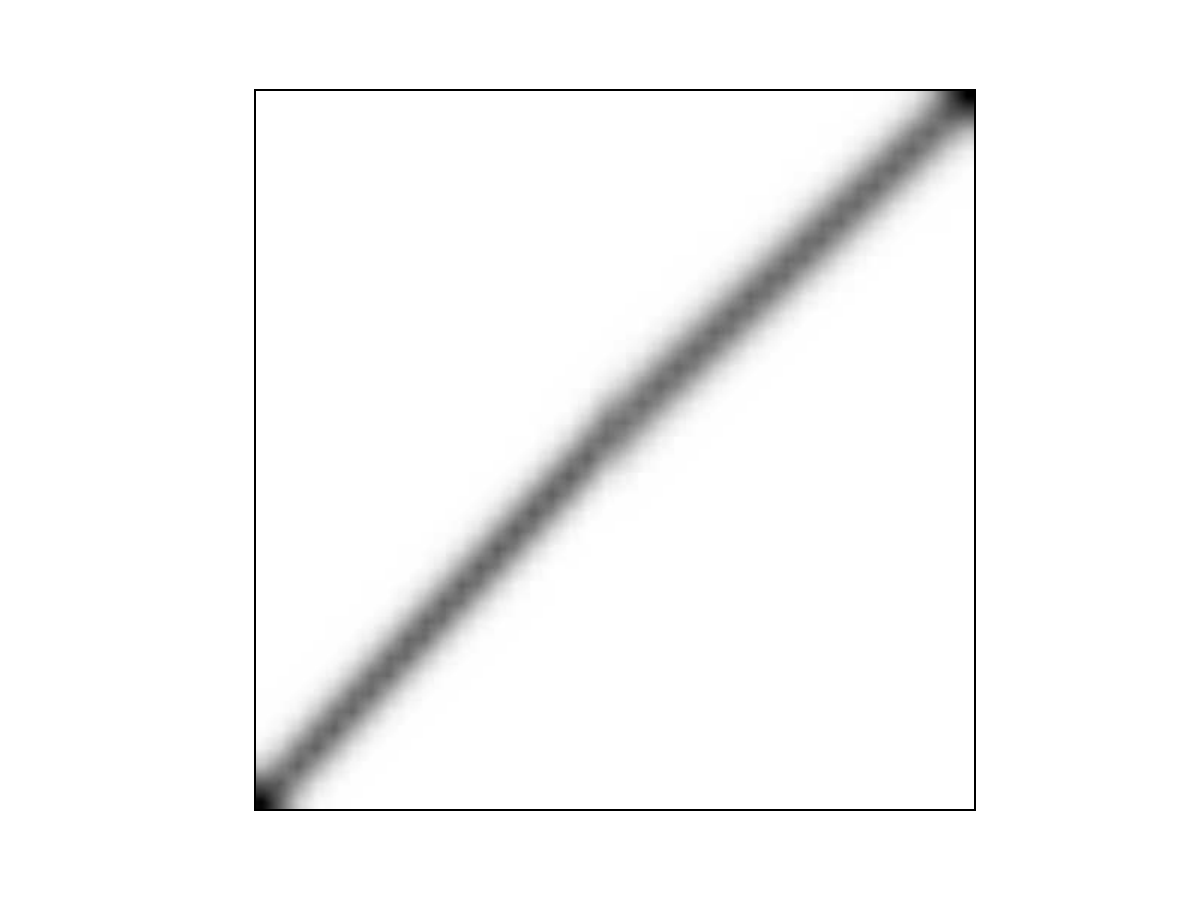}}
\subfigure[$k=11$]{\includegraphics[width=0.23\textwidth]{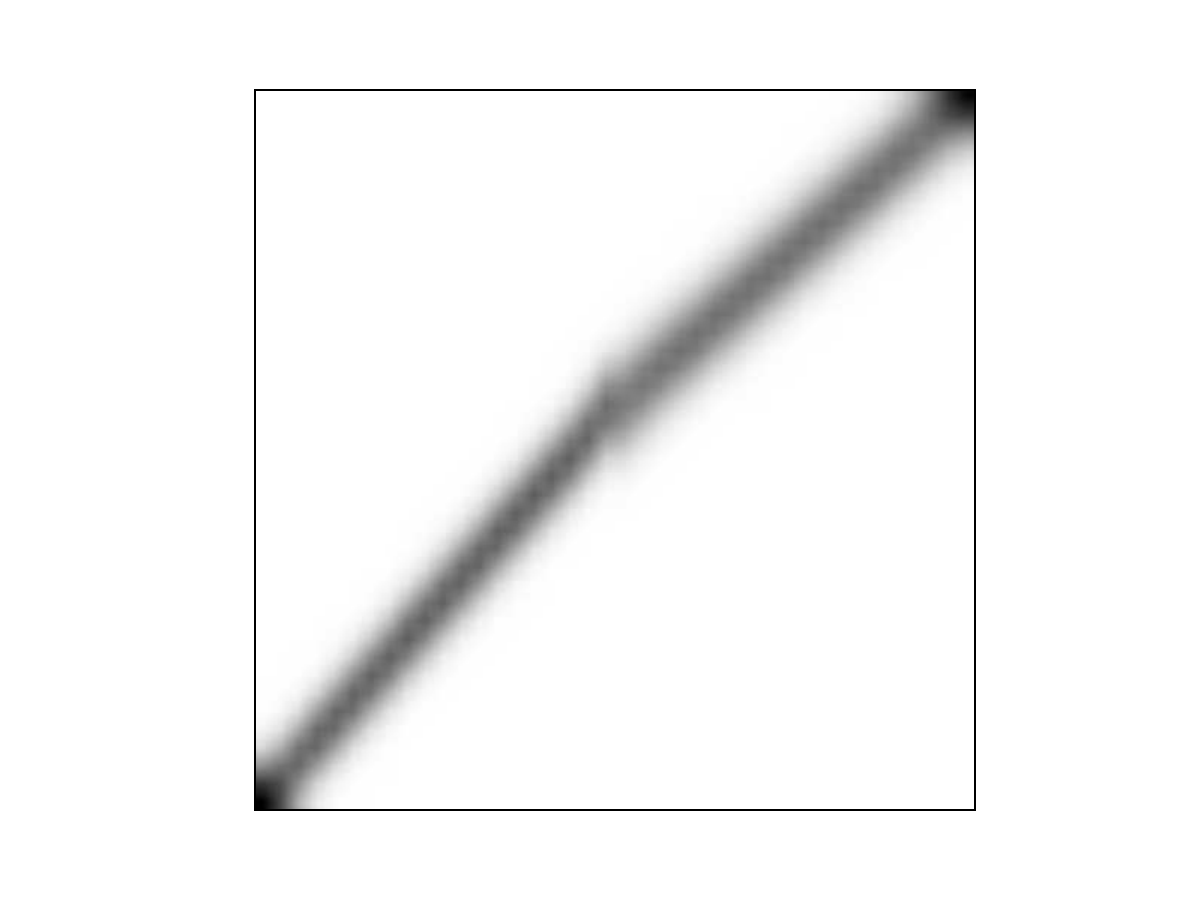}}
\subfigure[$k=16$]{\includegraphics[width=0.23\textwidth]{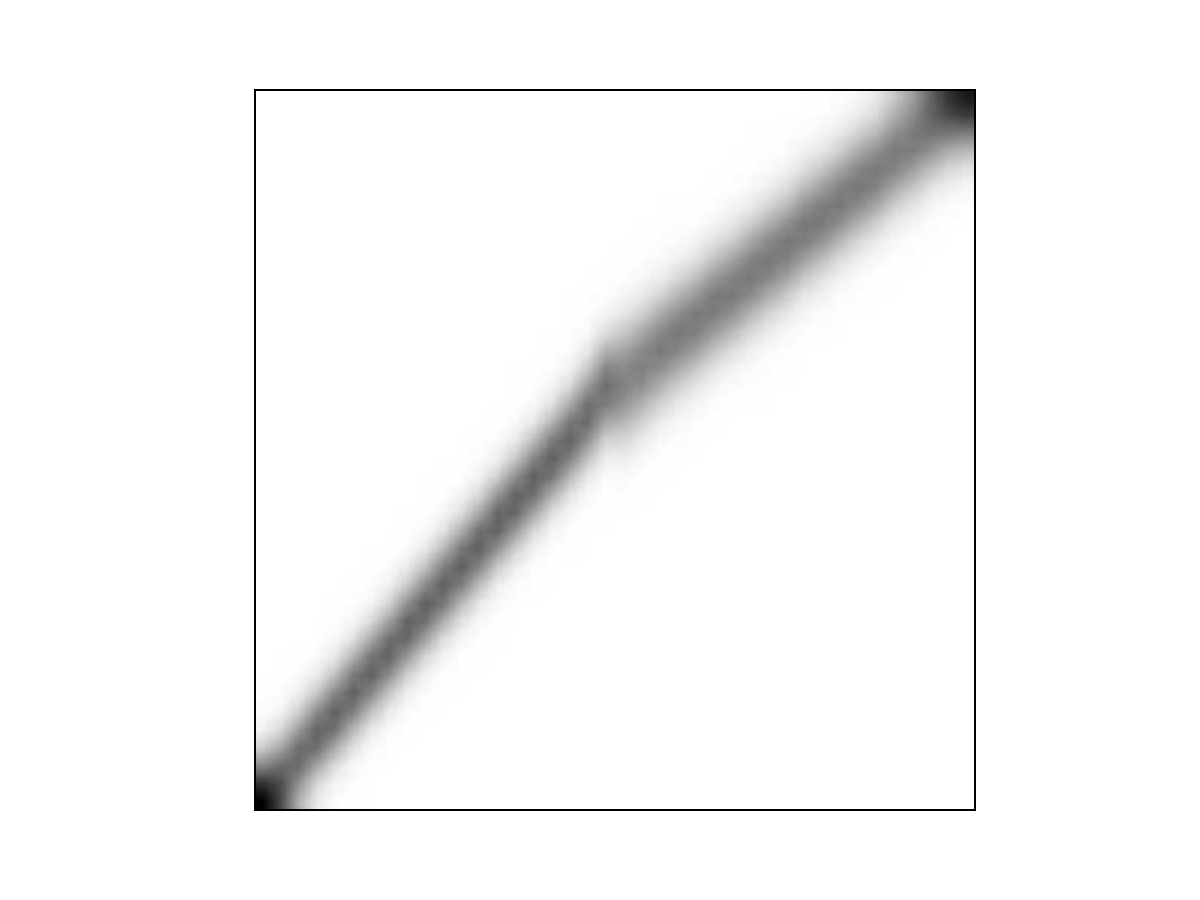}}
\subfigure[$k=20$]{\includegraphics[width=0.23\textwidth]{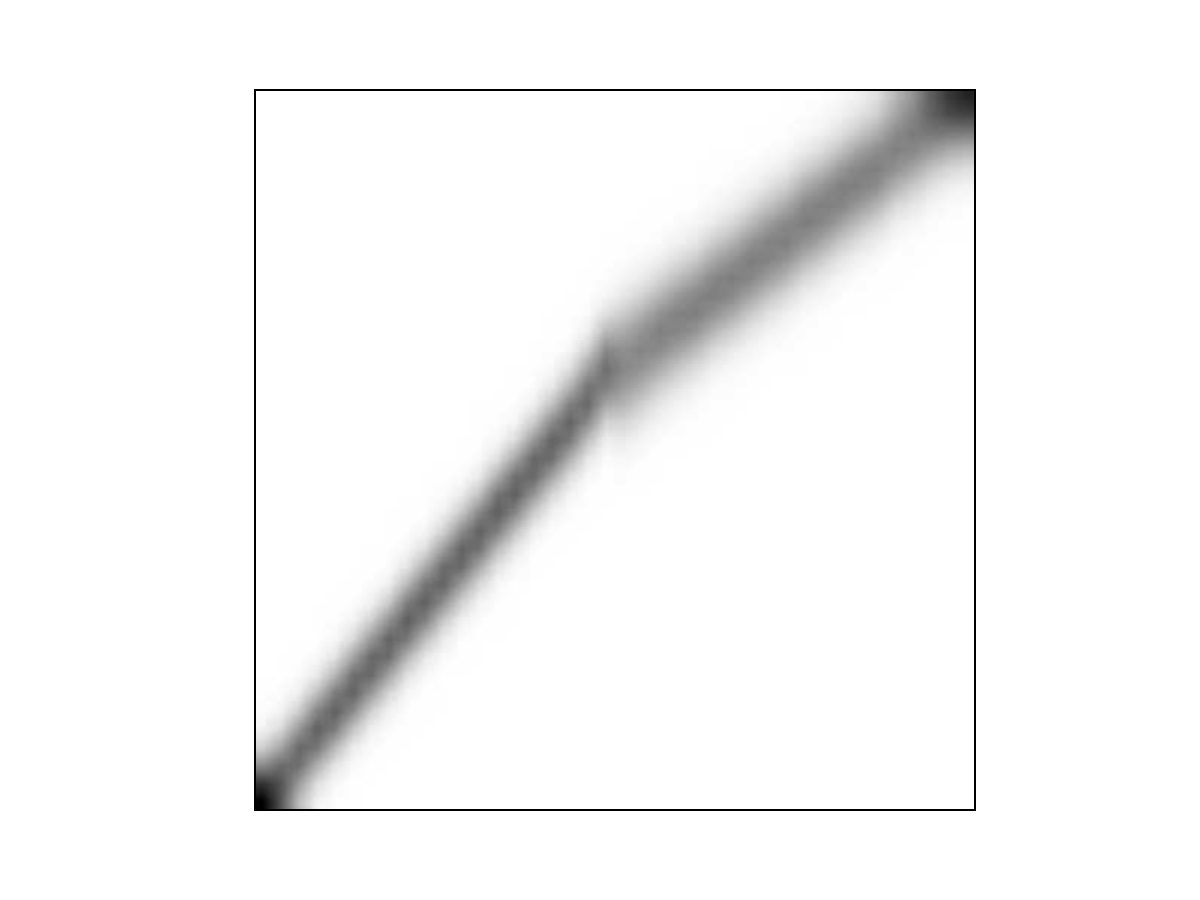}}
\subfigure[$k=25$]{\includegraphics[width=0.23\textwidth]{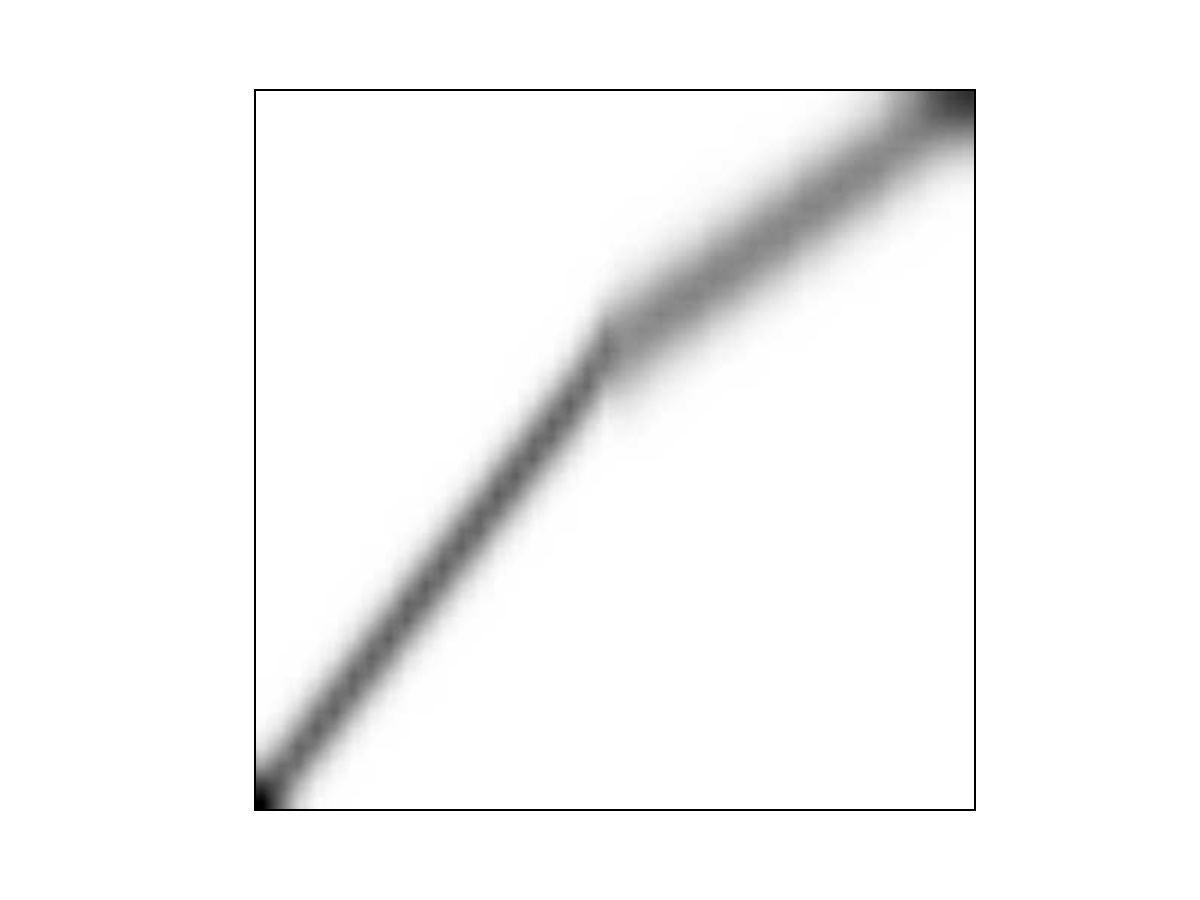}}
\subfigure[$k=30$]{\includegraphics[width=0.23\textwidth]{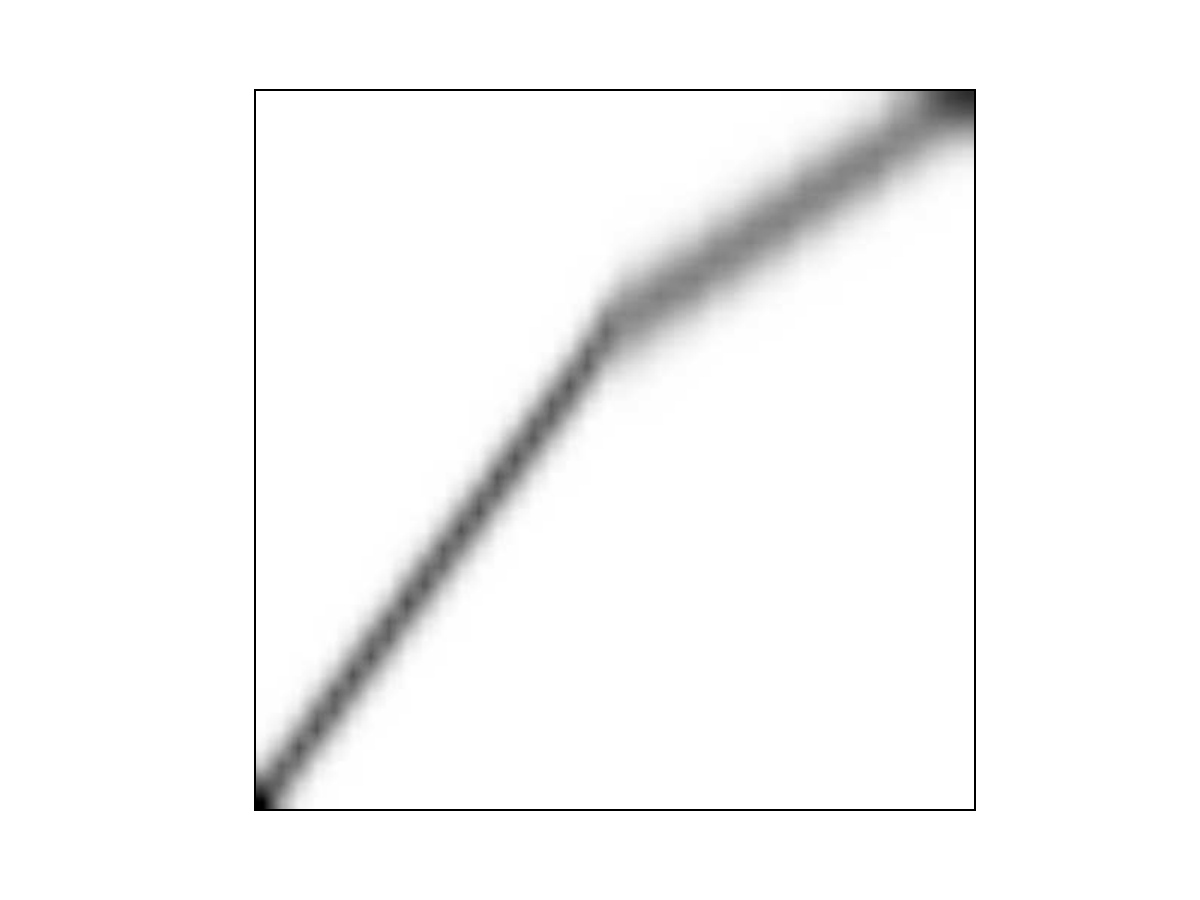}}
\subfigure[$k=35$]{\includegraphics[width=0.23\textwidth]{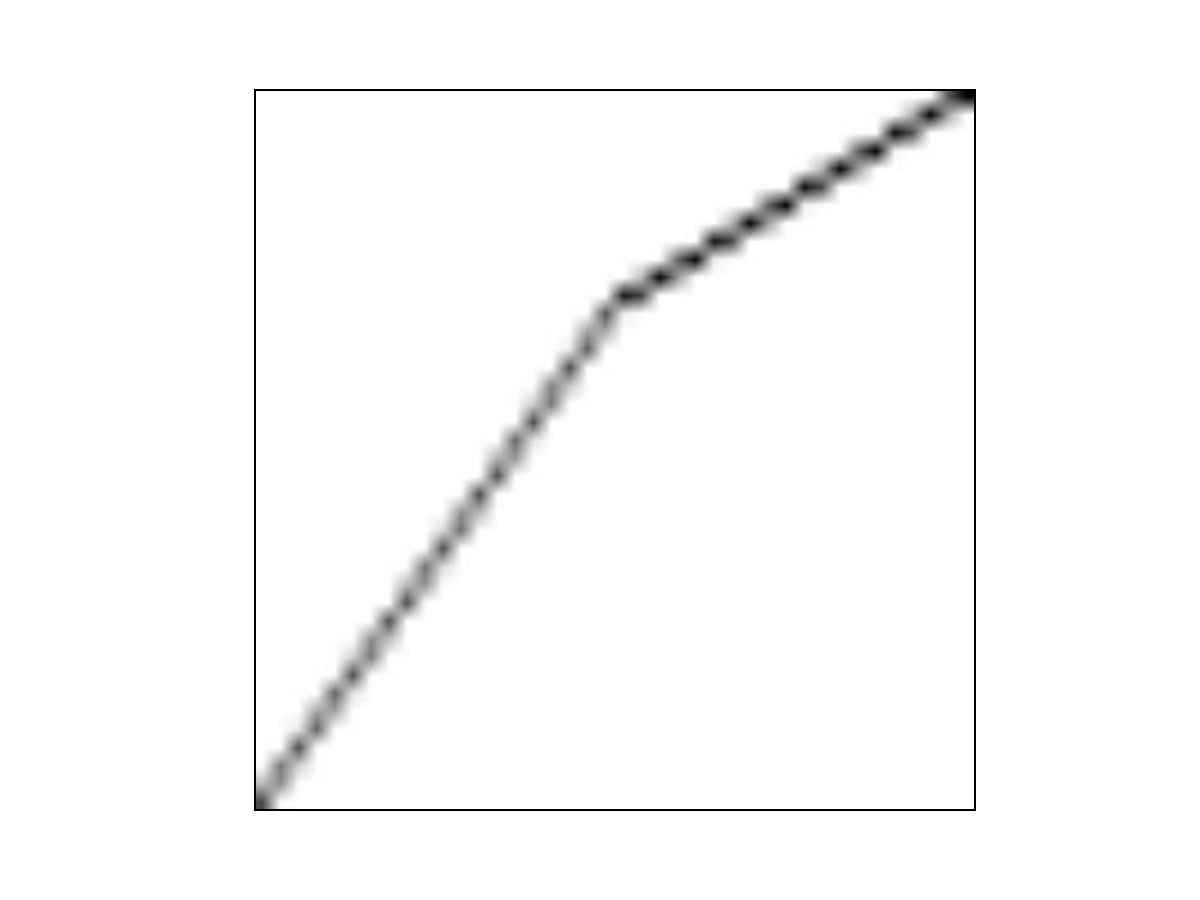}}
\caption{Transport couplings $(e^M_{0,t_k})_\#\bs{\mu}$ on $M\times M$ for the peakon-like solution associated with the boundary conditions specified by the map in equation \eqref{eq:bcpeakon}.} \label{fig:peakon1}
\end{figure}

\begin{figure}[htbp]
\centering
\subfigure[$k=1$]{\includegraphics[width=0.23\textwidth]{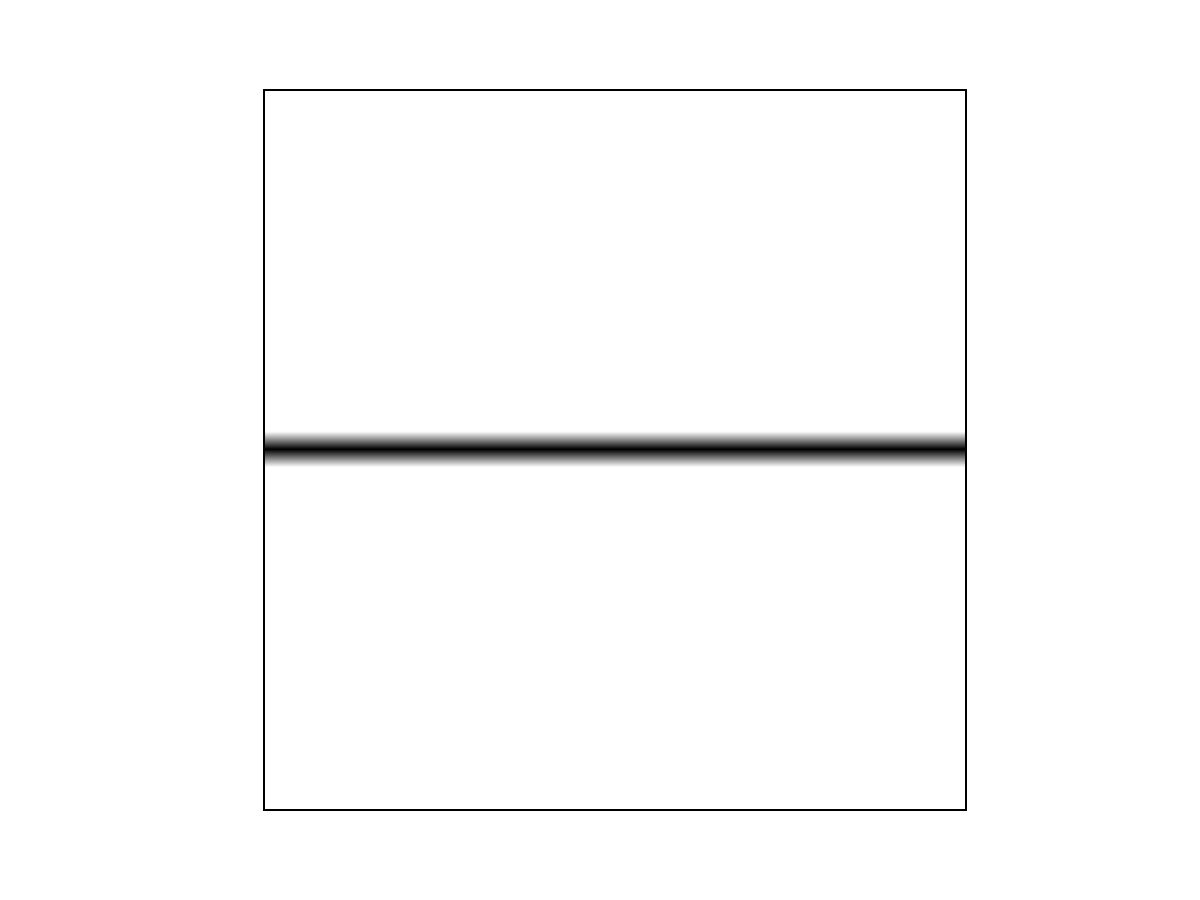}}
\subfigure[$k=6$]{\includegraphics[width=0.23\textwidth]{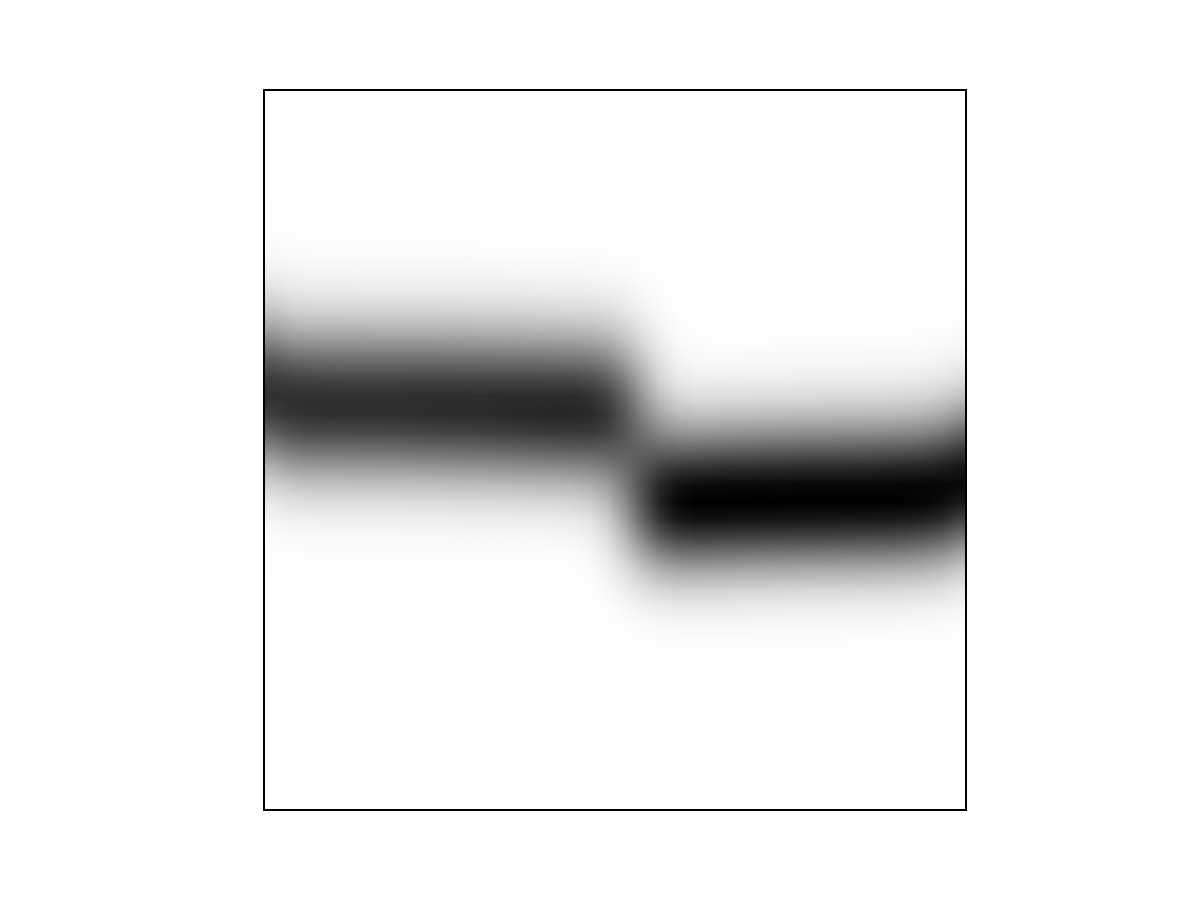}}
\subfigure[$k=11$]{\includegraphics[width=0.23\textwidth]{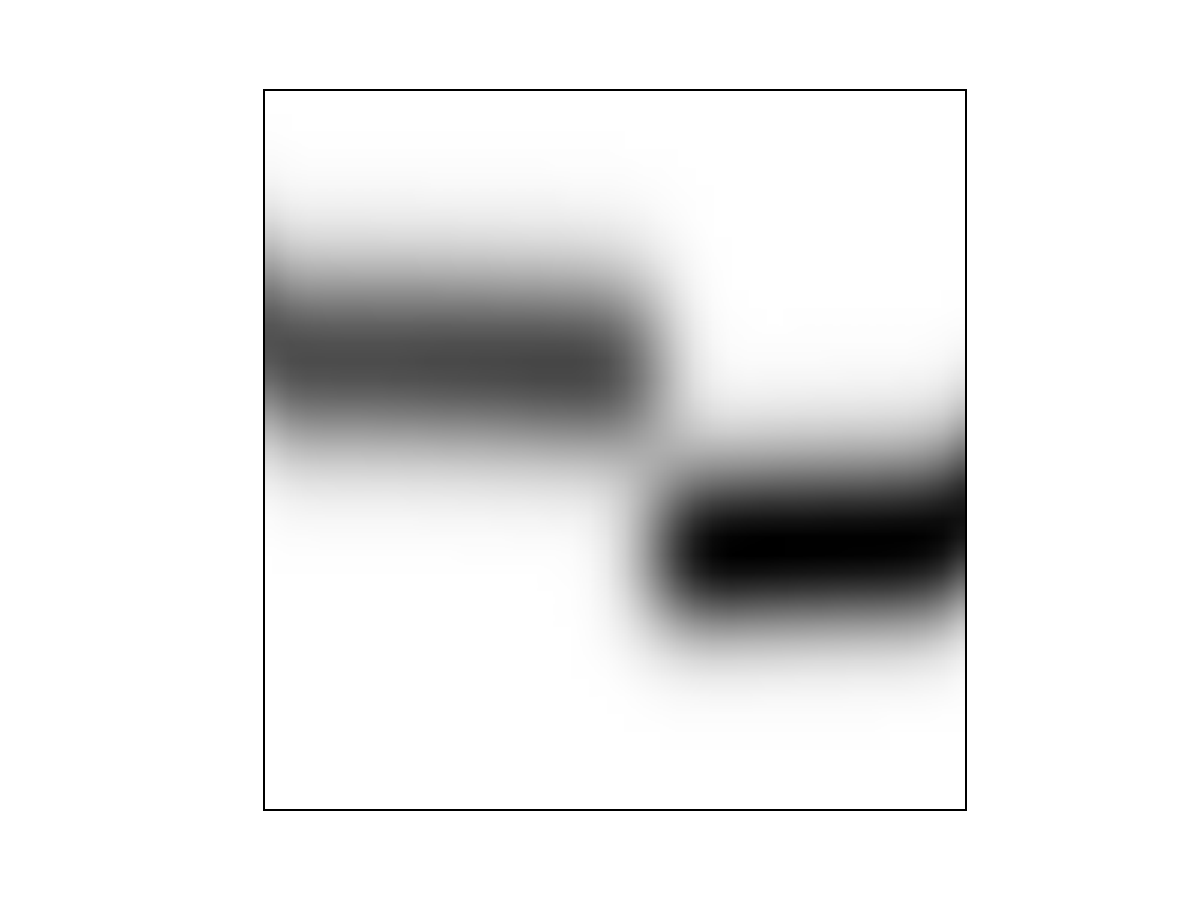}}
\subfigure[$k=16$]{\includegraphics[width=0.23\textwidth]{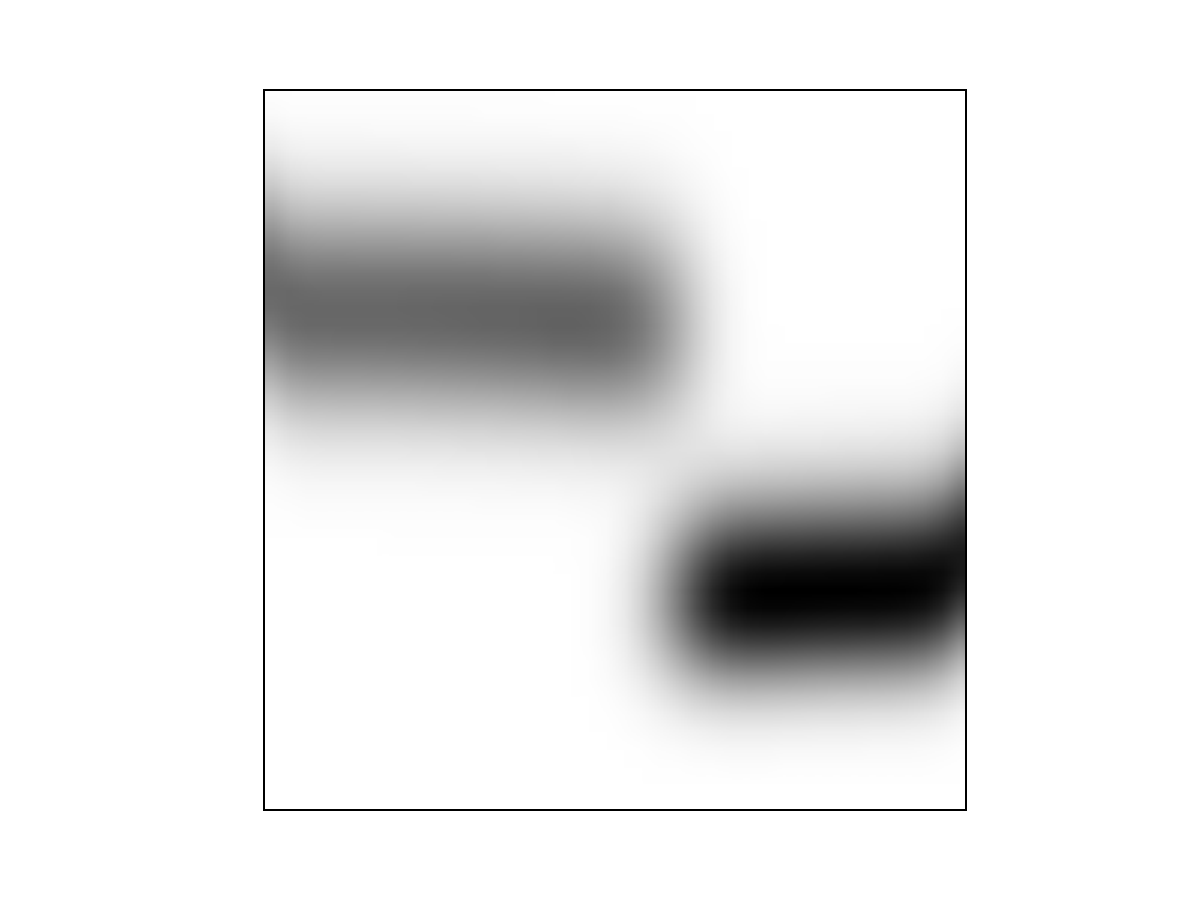}}
\subfigure[$k=20$]{\includegraphics[width=0.23\textwidth]{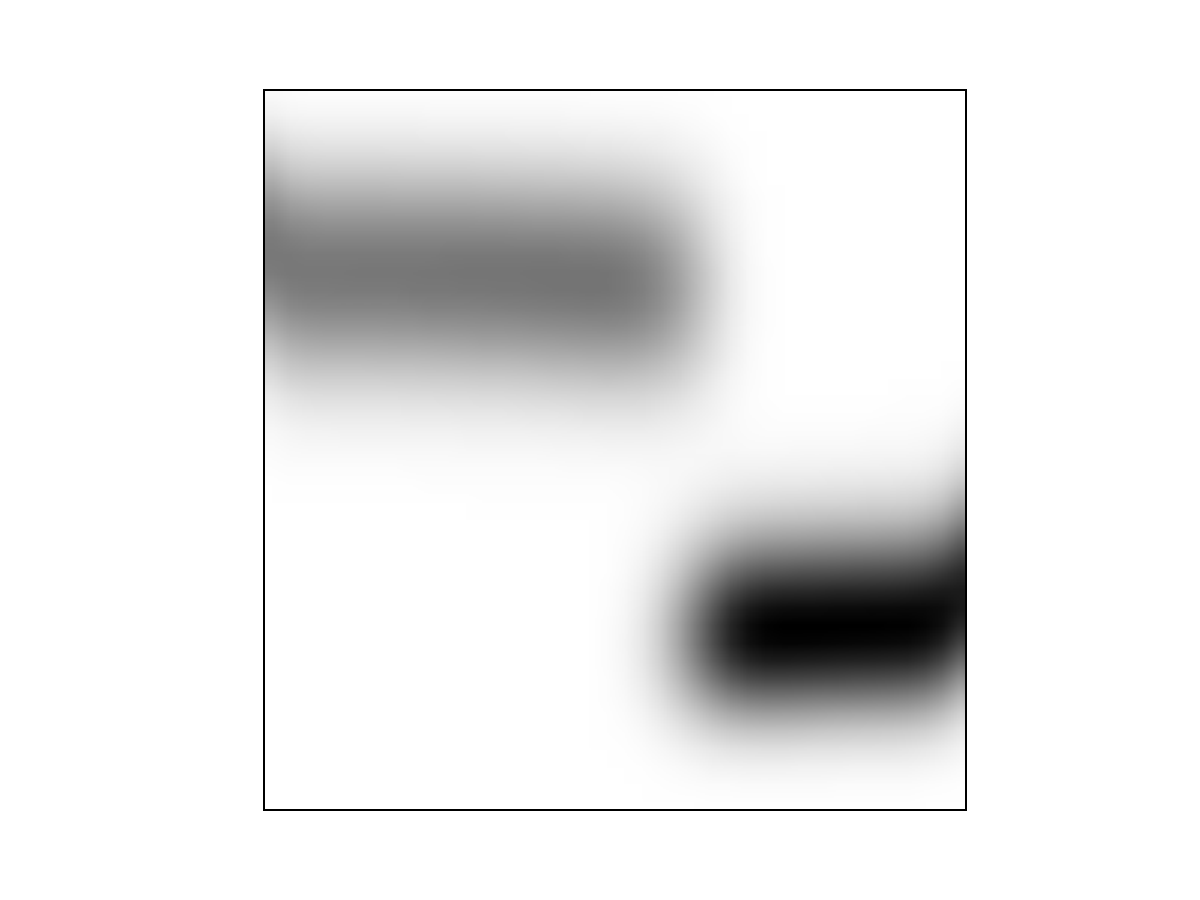}}
\subfigure[$k=25$]{\includegraphics[width=0.23\textwidth]{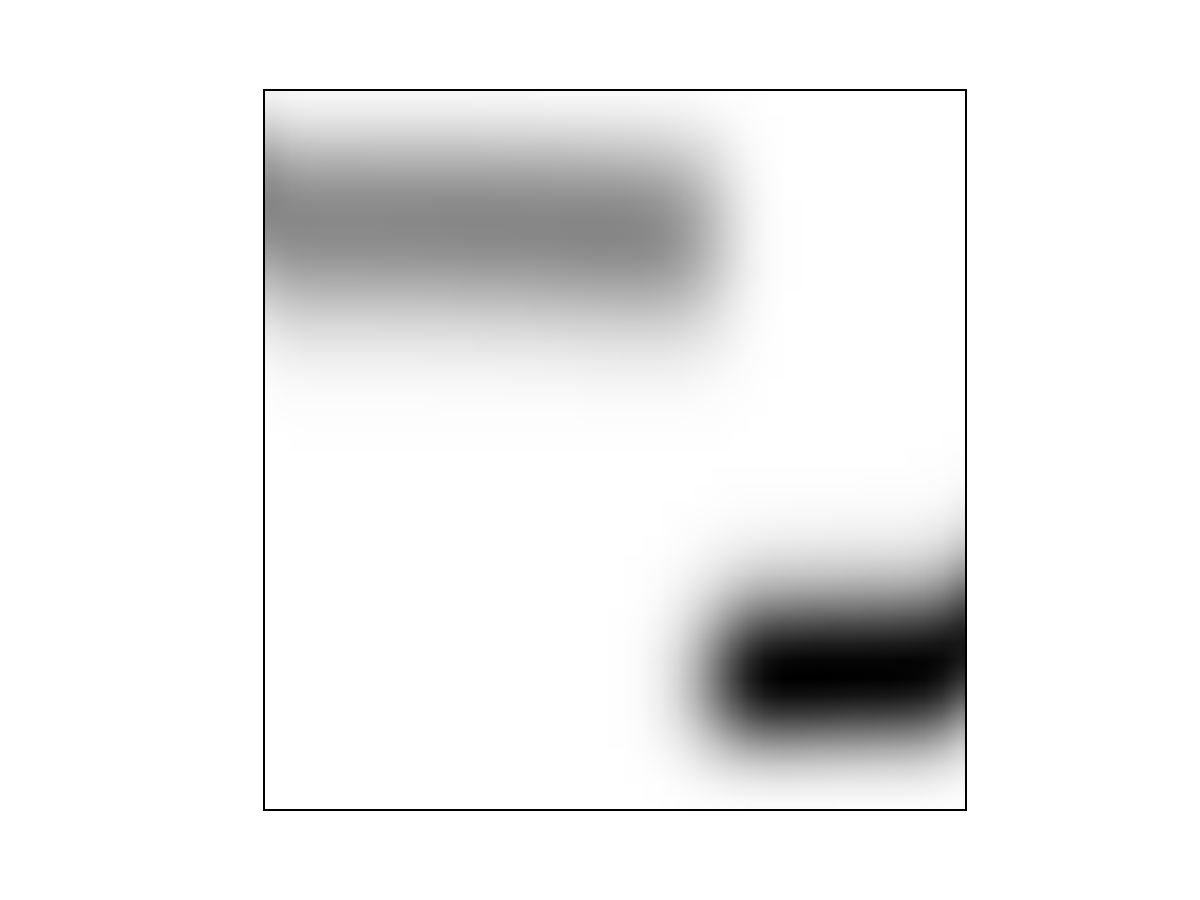}}
\subfigure[$k=30$]{\includegraphics[width=0.23\textwidth]{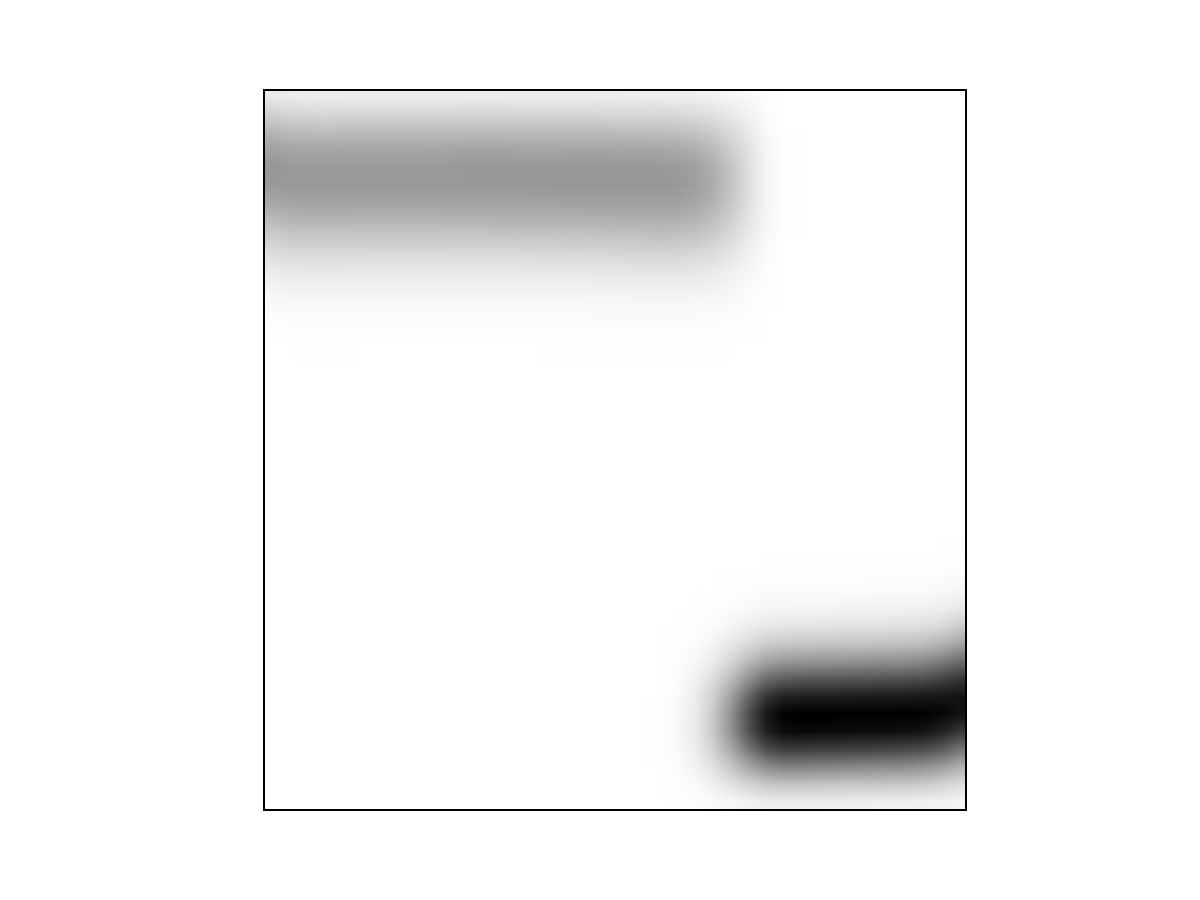}}
\subfigure[$k=35$]{\includegraphics[width=0.23\textwidth]{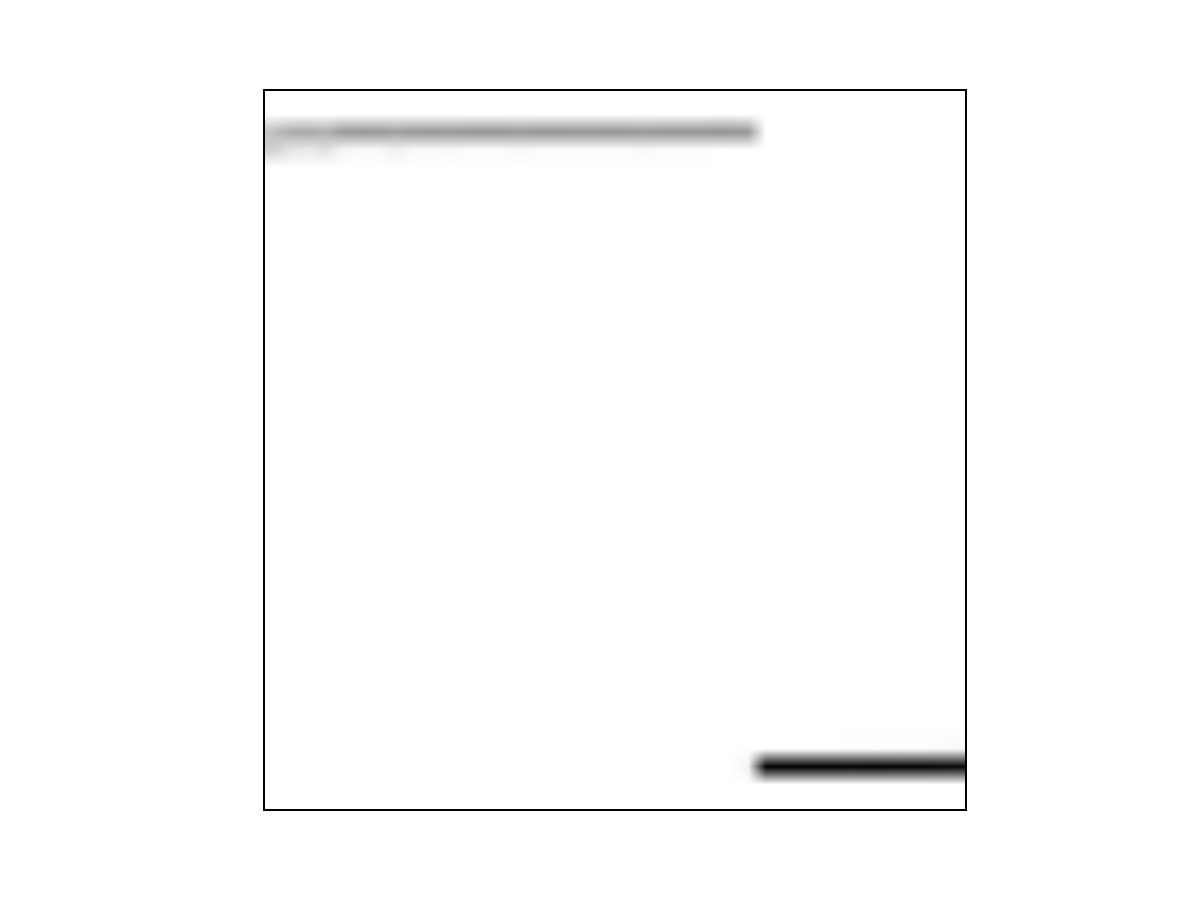}}
\caption{Fixed time marginals $(e_{t_k})_\# \bs{\mu}$ on the cone section $M\times [r_{min}, r_{max}]$ ($r_{min}=0.55$,  $r_{max}= 1.45$) for the peakon-like solution associated with the boundary conditions specified by the map  in equation \eqref{eq:bcpeakon}.} \label{fig:peakon2}
\end{figure}

\noindent \emph{A non-deterministic solution}.
The homogeneous marginal constraint allows us to consider very general couplings even defined by non-injective maps or maps that do not preserve the local orientation of the domain. Measure-preserving maps provide a special example since these were used by Brenier to define boundary conditions for generalized incompressible Euler flows.
In fact if $h$ is measure-preserving, i.e.\ $h_\# \rho_0 = \rho_0$, then we can use as coupling
\begin{equation}
\bs{\gamma} = [(\mr{Id},1),(h,1)]_\# \rho_0\,.
\end{equation}
Here, we take $h:[0,1] \rightarrow [0,1]$ to be the map 
\begin{equation}\label{eq:bcnondet}
h(x) = 1-x\,,
\end{equation} 
which can only be realized by a non-deterministic plan.  We compute the discrete solution associated with such boundary conditions with $N_x=40$, $N_r =41$, $0.6\leq r \leq 1.4$, $K=35$, $\alpha=40$, $\epsilon = 5\cdot 10^{-4}$. As before, we show the evolution of the transport plan on the domain $M$ given by $(e^M_{0,t_k})_\#\bs{\mu} \in \mc{P}(M^2)$ in figure \ref{fig:nondet1}. In figure \ref{fig:nondet2} we show the evolution of the marginals on the cone given by $(e_{t_k})_\#\bs{\mu} \in \mc{P}(\cone)$.  The transport plan evolution is remarkably similar to that of the incompressible Euler equation for the same coupling (see, e.g., \cite{benamou2017generalized}). However, the two do not coincide as it is evident from the marginals on the cone in figure \ref{fig:nondet2}. In the case of incompressible Euler, these marginals are concentrated on  $r = 1$ for every time, i.e.\ the transport plan remains measure-preserving during the evolution. This is clearly not the case for the generalized CH solution, for which also the Jacobian appears to be non-deterministic.

\begin{figure}[htbp]
\centering
\subfigure[$k=1$]{\includegraphics[width=0.23\textwidth]{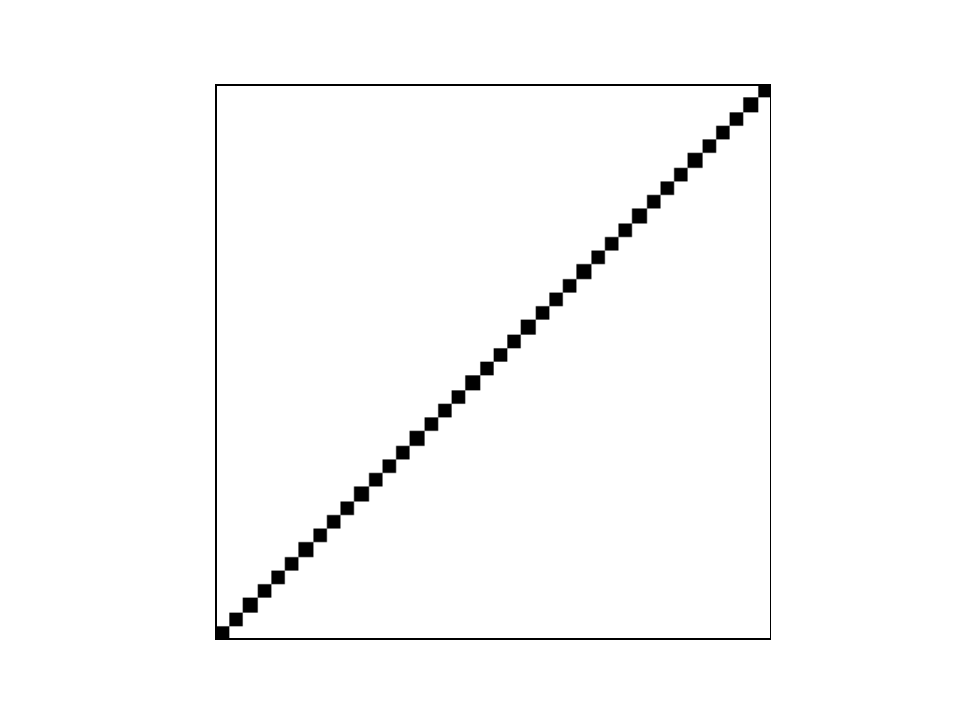}}
\subfigure[$k=6$]{\includegraphics[width=0.23\textwidth]{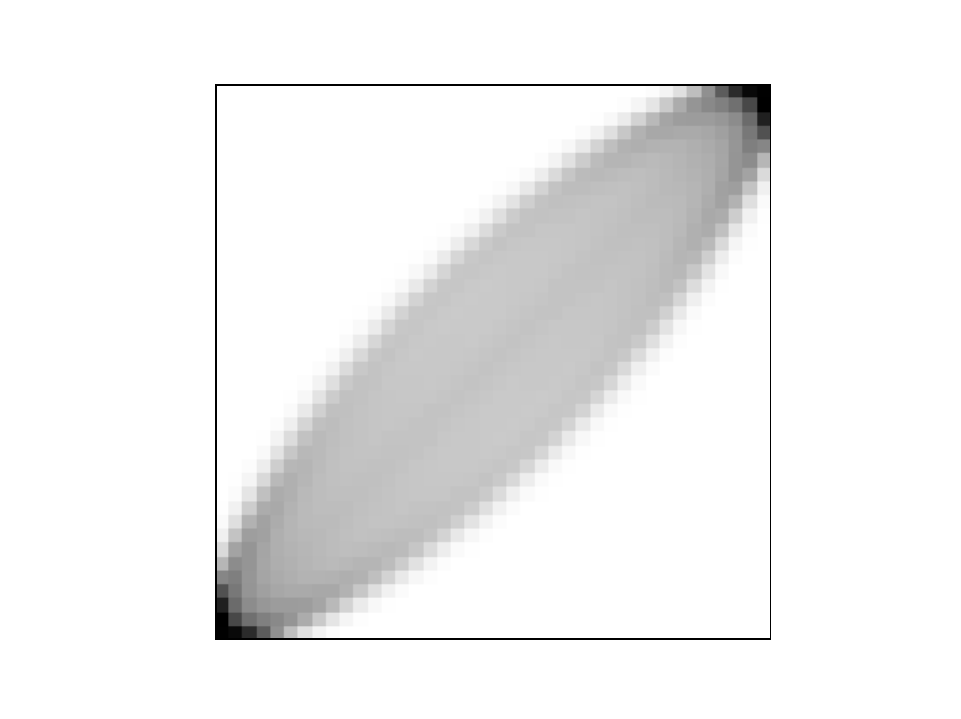}}
\subfigure[$k=11$]{\includegraphics[width=0.23\textwidth]{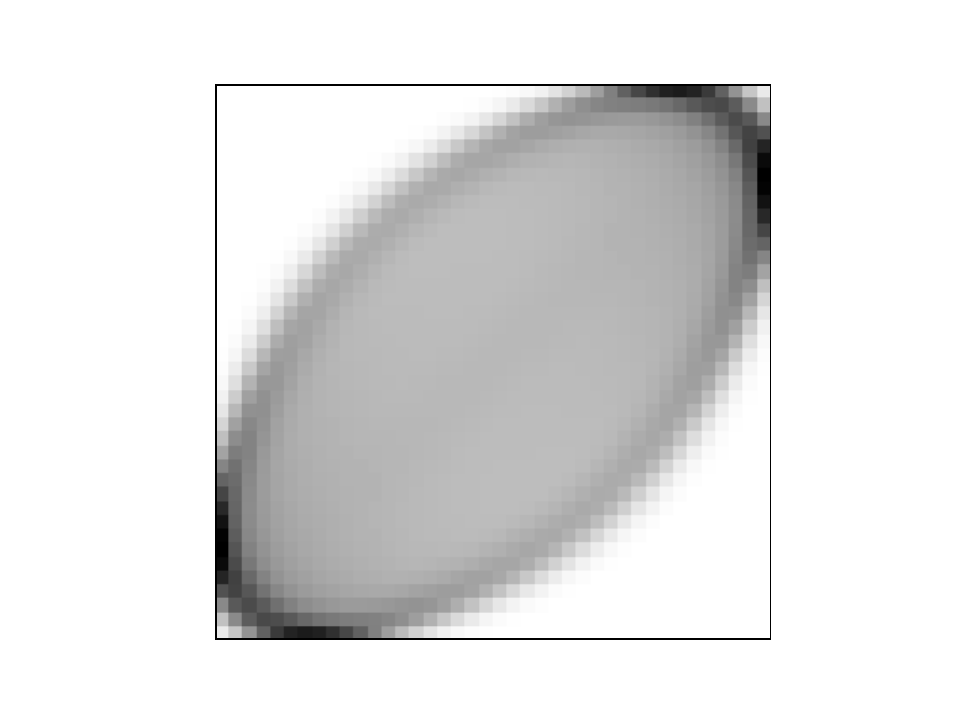}}
\subfigure[$k=16$]{\includegraphics[width=0.23\textwidth]{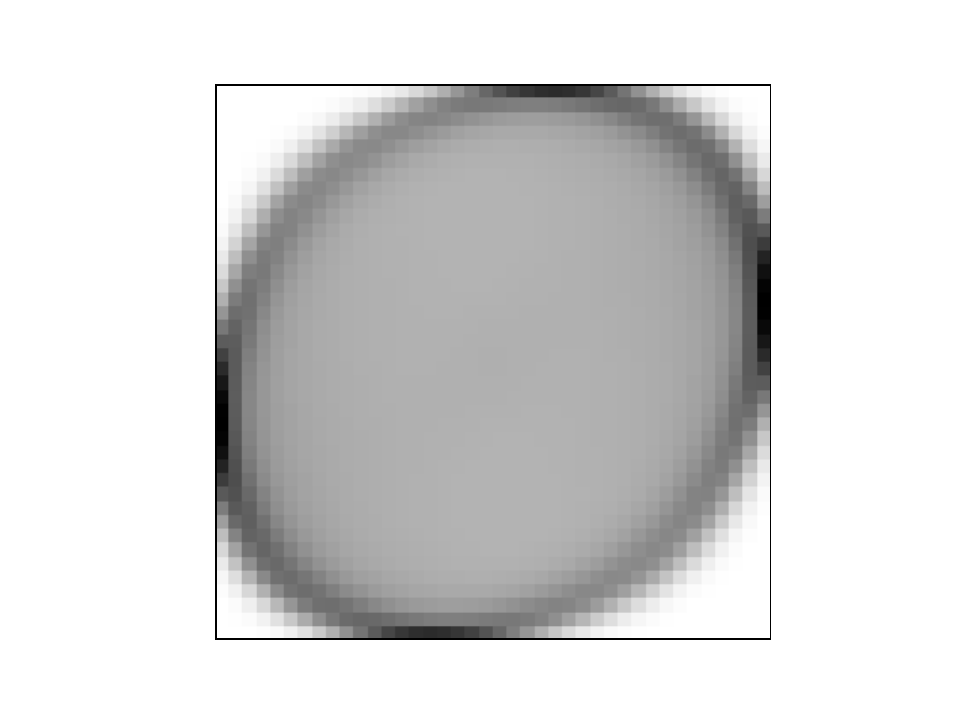}}
\subfigure[$k=20$]{\includegraphics[width=0.23\textwidth]{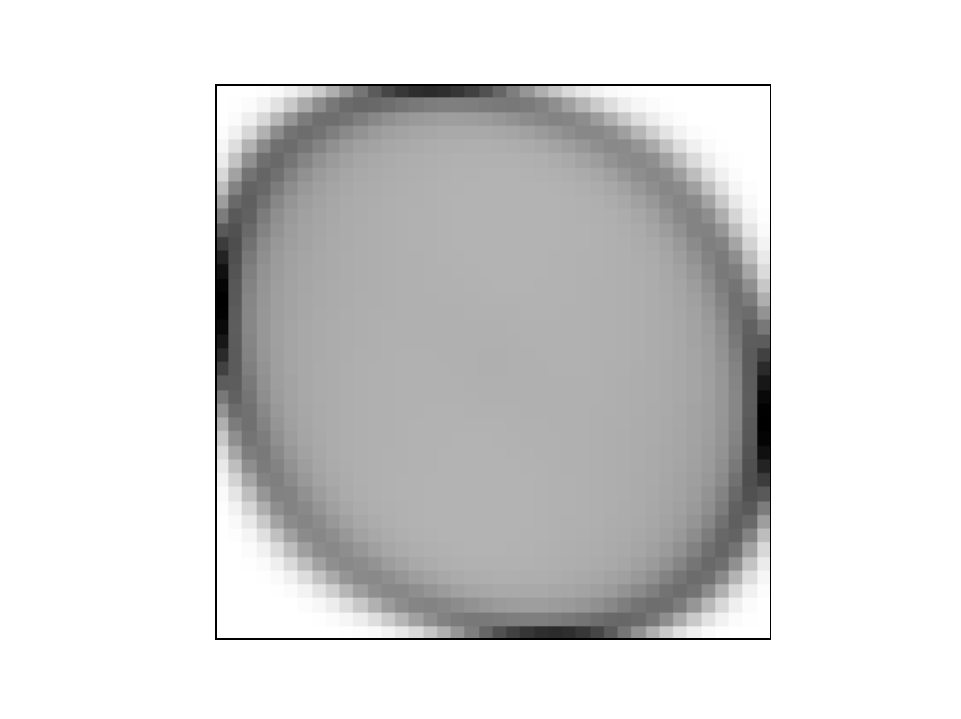}}
\subfigure[$k=25$]{\includegraphics[width=0.23\textwidth]{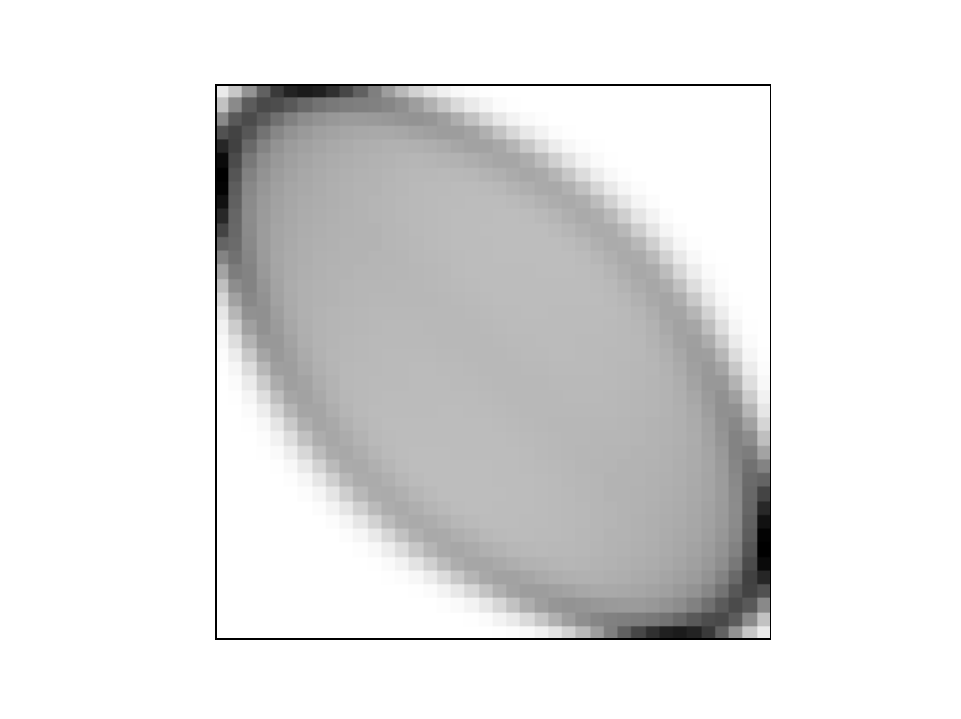}}
\subfigure[$k=30$]{\includegraphics[width=0.23\textwidth]{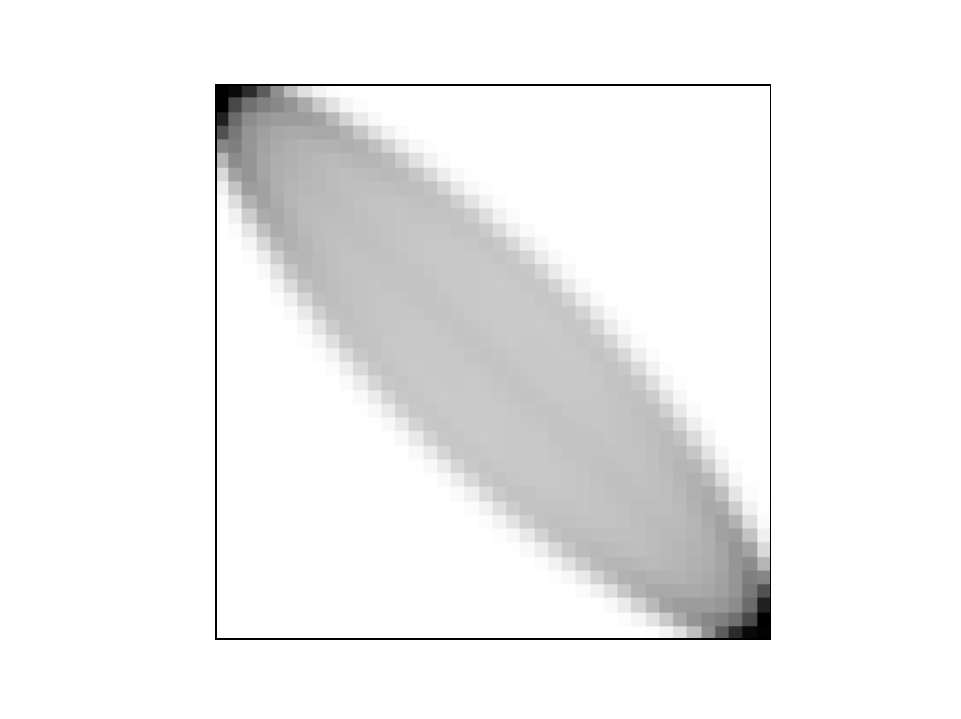}}
\subfigure[$k=35$]{\includegraphics[width=0.23\textwidth]{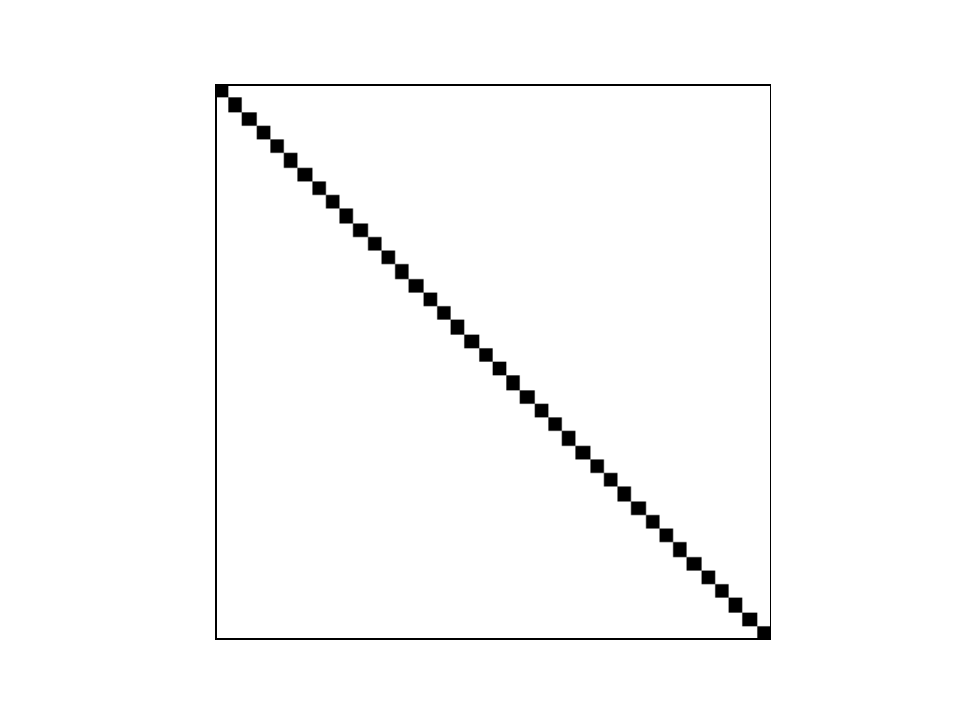}}
\caption{Transport couplings $(e^M_{0,t_k})_\#\bs{\mu}$ on $M\times M$  for the non-deterministic solution associated  to the boundary conditions specified by the map  in equation \eqref{eq:bcnondet}.} \label{fig:nondet1}
\end{figure}

\begin{figure}[htbp]
\centering
\subfigure[$k=1$]{\includegraphics[width=0.23\textwidth]{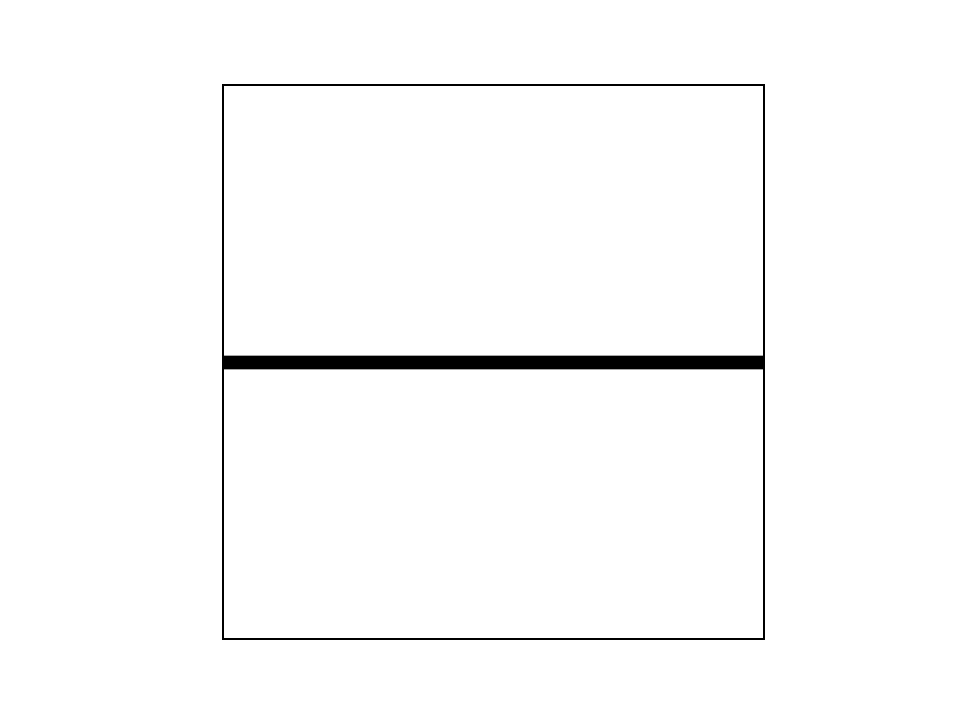}}
\subfigure[$k=6$]{\includegraphics[width=0.23\textwidth]{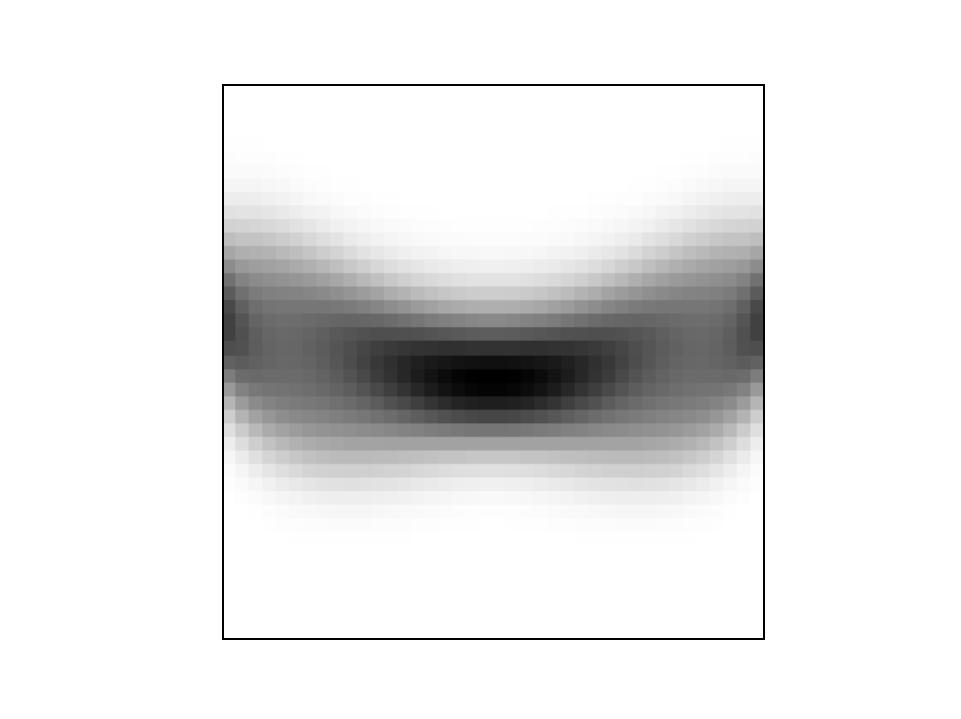}}
\subfigure[$k=11$]{\includegraphics[width=0.23\textwidth]{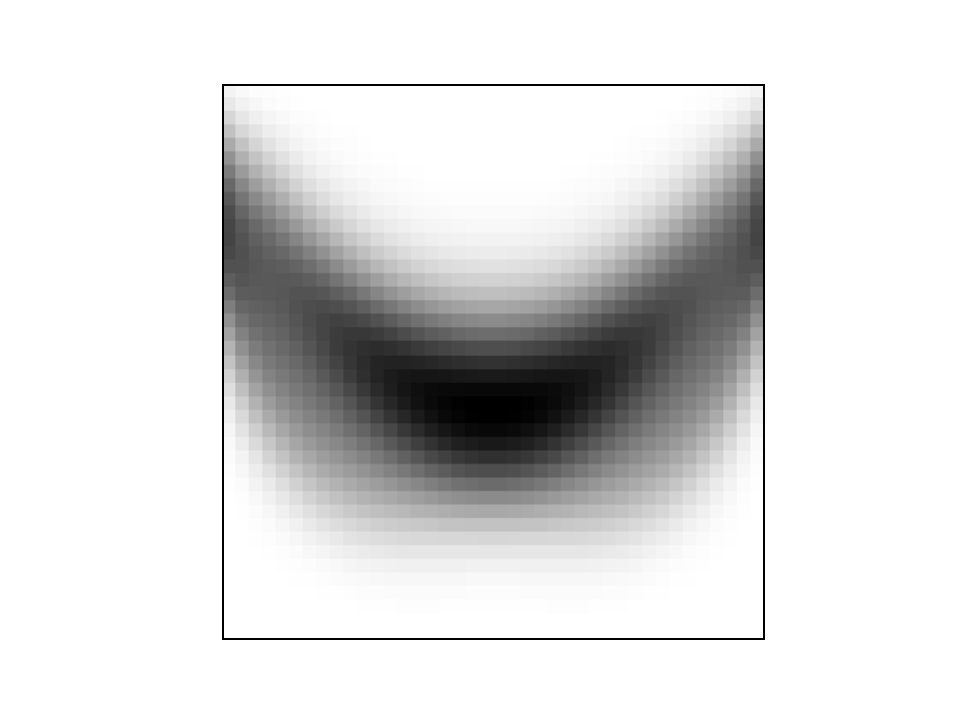}}
\subfigure[$k=16$]{\includegraphics[width=0.23\textwidth]{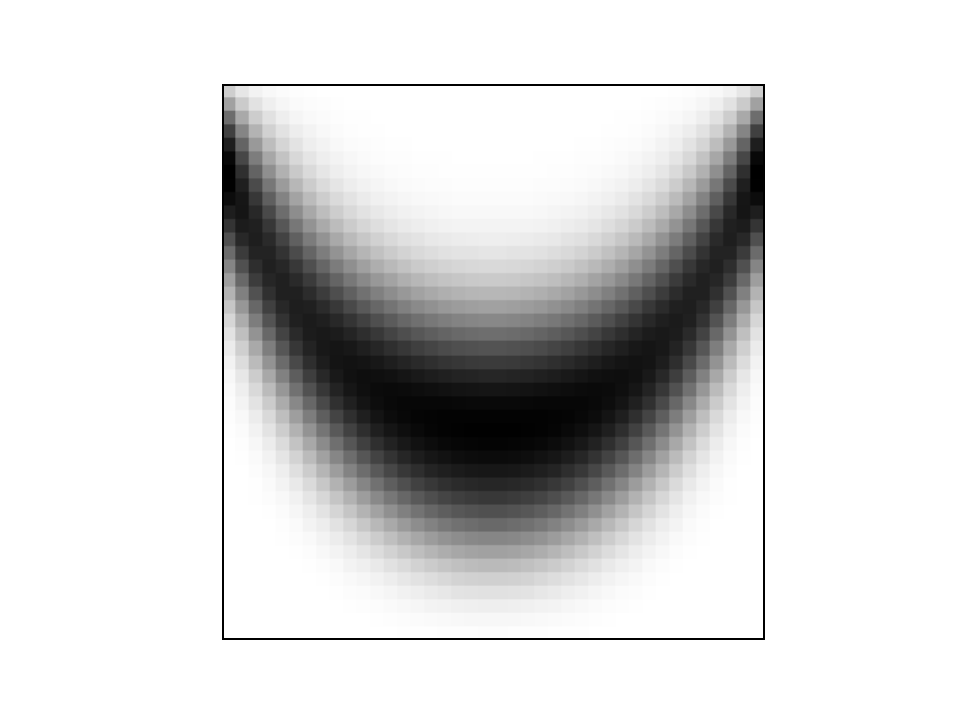}}
\subfigure[$k=20$]{\includegraphics[width=0.23\textwidth]{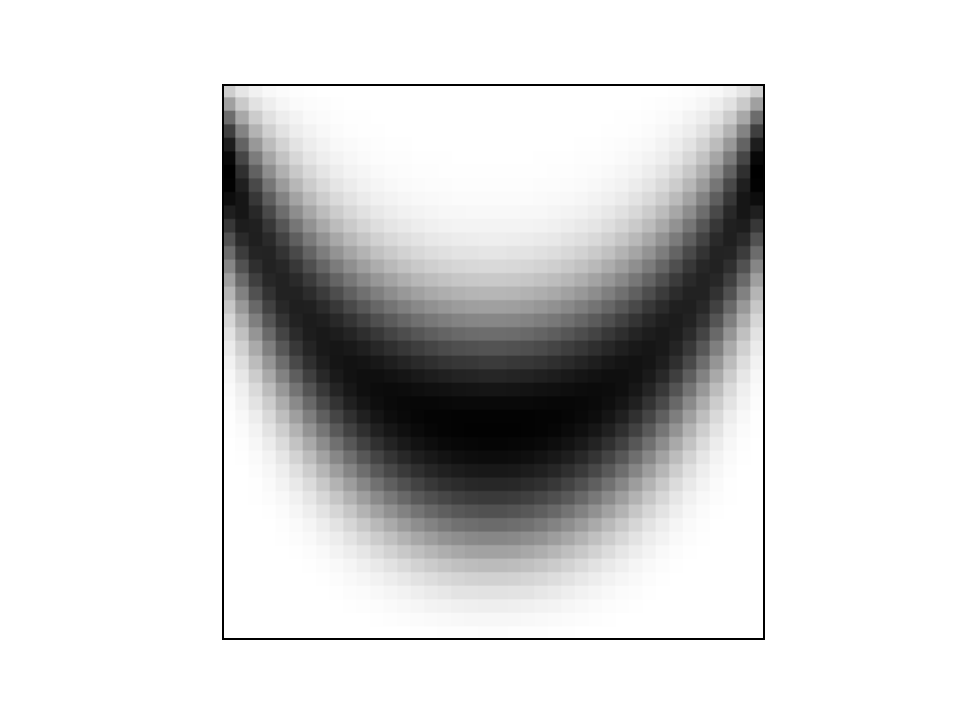}}
\subfigure[$k=25$]{\includegraphics[width=0.23\textwidth]{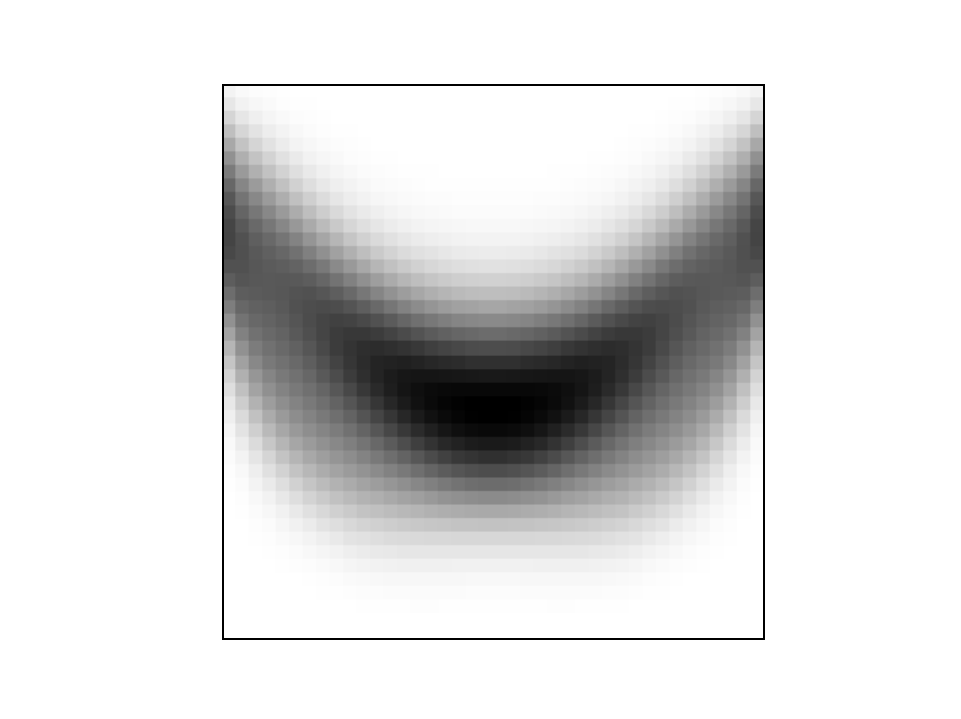}}
\subfigure[$k=30$]{\includegraphics[width=0.23\textwidth]{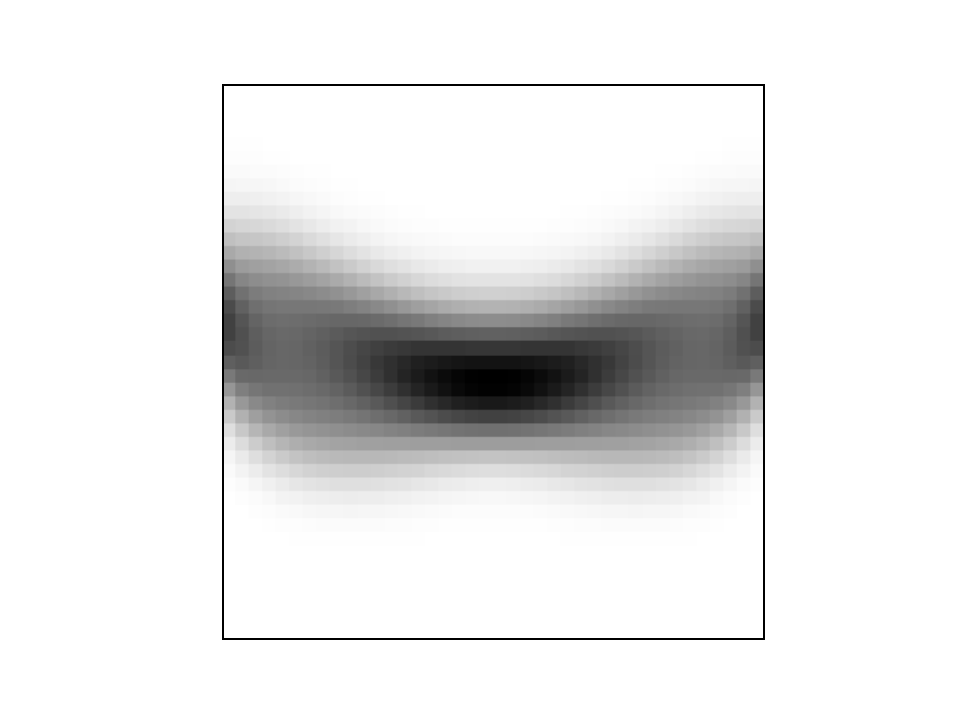}}
\subfigure[$k=35$]{\includegraphics[width=0.23\textwidth]{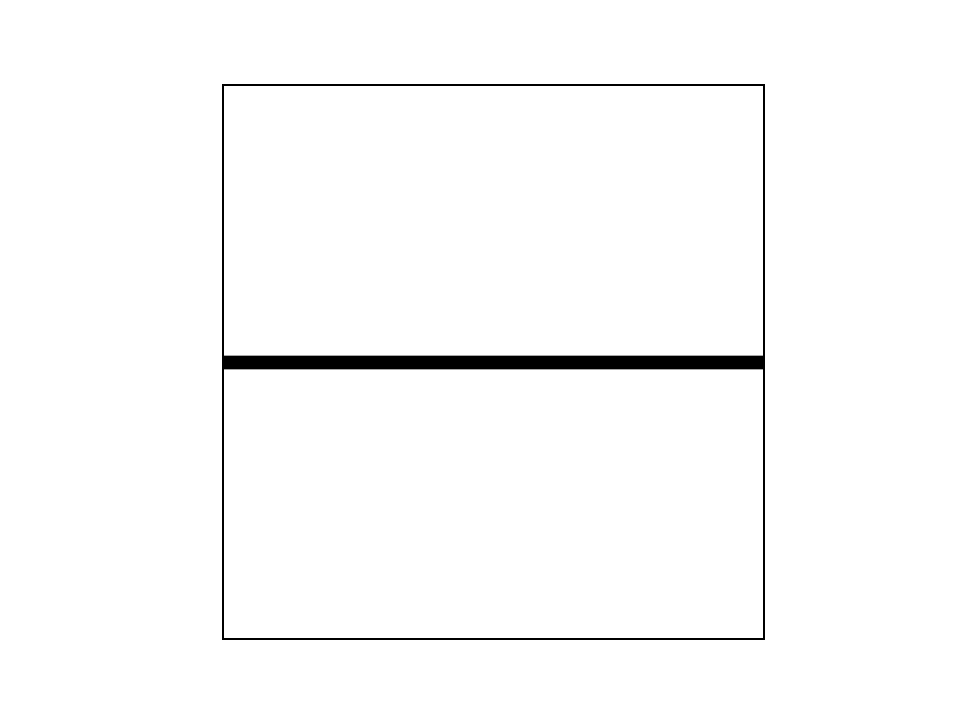}}
\caption{Fixed time marginals $(e_{t_k})_\# \bs{\mu}$ on the cone section $M\times [r_{min}, r_{max}]$ ($r_{min}=0.6$,  $r_{max}= 1.4$) for the non-deterministic associated with the boundary conditions specified by the map  in equation \eqref{eq:bcnondet}.} \label{fig:nondet2}
\end{figure}

\section{Outlook}\label{sec:conclusions}

There are several natural questions that were not addressed in this paper and that we reserve to future work: 
\begin{itemize}
\item \emph{Tight relaxation}. Brenier's relaxation of incompressible Euler is not tight in two dimensions but it is in three dimensions due to the work of Shnirelman \cite{shnirelman1994generalized}. It is an open question whether a similar result holds for the generalized problem studied in this paper.
The approximation results in section \ref{sec:examples} suggest that this is the case. In particular, we conjecture that our formulation is a tight relaxation of the $H(\mr{div})$ geodesic  problem in dimension $d \geq 2$.   
\end{itemize}

As for the generalized Euler solutions, a better understanding of the structure of minimizing generalized flows is of theoretical interest:
\begin{itemize}
\item \emph{Occurrence of singular solutions}. In this paper we did not fully characterize the emergence of singular solutions. Even for the case of rotation on the circle or on the torus, for example, we did not prove that these are the unique minimizers for the problem. In addition, such examples suggest that singular solutions appear whenever particles' displacement is sufficiently large. It would be interesting to give a full characterization in this direction, specifying when solutions are singular in terms of the boundary conditions and the dimension and geometry of the base space $M$;
\item \emph{Regularity of the pressure}. Brenier's result on the existence and uniqueness of the pressure in incompressible Euler was subsequently improved by Ambrosio and Figalli \cite{ambrosio2008regularity} in terms of regularity of the pressure field. It is natural to ask whether such a result can be extended to the generalized $H(\mr{div})$ geodesic problem. This question is related to the previous one, due to the fact that a sufficiently regular pressure field can prevent the occurrence of singular solutions as it can be deduced from the proofs in section \ref{sec:correspondence}. 
\end{itemize}

Addressing these theoretical questions will also guide the development of numerical schemes which are better suited to the formulation considered in this paper than methods based on entropic regularization. A viable alternative in this context is given by semi-discrete methods \cite{merigot2011multiscale} (see also the schemes developed for the incompressible Euler equations in \cite{merigot2016minimal,gallouet2017lagrangian}), whose use for the generalized $H(\mr{div})$ geodesic problem will also be studied in future work. 

\section*{Acknowledgments}
The research leading to these results has received funding from the People Programme (Marie
Curie Actions) of the European Union’s Seventh Framework Programme (FP7/2007-2013)
under REA grant agreement n.\
PCOFUND-GA-2013-609102, through the PRESTIGE
programme coordinated by Campus France.
The authors would also like to acknowledge the support from the project MAGA ANR-16-CE40-0014 (2016-2020).

\appendix

\section{Proof of lemma \ref{lem:equivconstr}}

\begin{proof}
Here we prove that the homogeneous marginal constraint can be enforced at each time rather than in integral form as in equation \eqref{eq:constraintsunb}. 

First, we prove that the constraint in equation \eqref{eq:constraintsunb} implies the one in  equation \eqref{eq:constraintsunbs}. In order to show this, for any fixed $t^*\in[0,T]$  and $f\in C^0(M)$, consider the following functionals 
\begin{equation}
\mc{F}(\mr{z}) \coloneqq r_{t^*}^2 f(x_{t^*}) \,, \qquad \mc{F}_n(\mr{z}) \coloneqq \int_0^T r_t^2 f(x_t) \delta_{n,t^*}(t)\,\ed t\,,
\end{equation}
where $\delta_{n,t^*}:[0,T] \rightarrow \mathbb{R}$, $n\in \mathbb{N}$, is a Dirac sequence of continuous functions converging to $\delta_{t^*}$. Then for any $\mr{z}\in \Omega$,  $\mc{F}_n(\mr{z})\rightarrow \mc{F}(\mr{z})$ as $n\rightarrow +\infty$. Moreover, using Jensen's inequality,
\begin{equation}
\begin{aligned}
\mc{F}_n(\mr{z}) & \leq \|f \|_{C^0} \int_0^T r_t^2 \delta_{n,t^*} \, \ed t \\ 
&\leq 2 \|f \|_{C^0} \left( r_0^2 + \int_0^T (r_t-r_0)^2\delta_{n,t^*} \, \ed t \right )\\
&\leq 2 \|f \|_{C^0} \left( r_0^2 + \int_0^T \dot{r}_t^2 \, \ed t\, \int_0^T t\,\delta_{n,t^*} \, \ed t \right )\\
&\leq 2 \|f \|_{C^0} \left( r_0^2 +  T\mc{A}(\mr{z})  \right )\,.
\end{aligned}
\end{equation}
The right-hand side is $\bs{\mu}$-integrable since $\mc{A}(\bs{\mu})<+\infty$ and because of the coupling constraint. Hence, we get the result by the dominated convergence theorem.

Similarly, if $f\in C^0([0,T]\times M)$, we take 
\begin{equation}
\mc{F}(\mr{z}) \coloneqq \int_0^T f(t,x_t) r_t^2 \, \ed t \,, \qquad \mc{F}_n(\mr{z}) \coloneqq \frac{T}{K} \sum_{k=0}^K f(t_k,x_{t_k}) r_{t_k}^2\,, 
\end{equation}
where $t_k \coloneqq kT/K$. Then for any $\mr{z}\in \Omega$,  $\mc{F}_n(\mr{z})\rightarrow \mc{F}(\mr{z})$ as $n\rightarrow +\infty$. Moreover,
\begin{equation}
\begin{aligned}
\mc{F}_n(\mr{z}) & \leq 2 \|f \|_{C^0} \left(r_0^2 + \frac{T}{K} \sum_{k=1}^K (r_{t_k}-r_0)^2\right) \\
& \leq 2 \|f \|_{C^0} \left(r_0^2 + \frac{T}{K}  \sum_{k=1}^K t_k\int_0^{t_k} \dot{r}_t^2 \, \ed t \right) \\
& \leq 2 \|f \|_{C^0} \left(r_0^2 + {T}^2  \mc{A}(\mr{z}) \right)\,, \\
\end{aligned}
\end{equation} 
and we can apply again the dominated convergence theorem to conclude the proof.
\end{proof}

\section{Proof of lemma \ref{lem:gv}}
\begin{proof}
Throughout this proof, most metric operations are performed with respect to the cone metric $g_\cone$, so to simplify the notation we will use $|\cdot|$ and $\langle \cdot, \cdot \rangle$ to denote, respectively, the norm and the inner product both on $T\cone$ and $\mathbb{R}^d$ according to the context. Moreover, given a vector field $u$ on the cone and a curve $t\mapsto p(t) \in \cone$, $\nabla_t u(p(t)) \coloneqq \nabla_{\dot{p}(t)} u(p(t))$ is the covariant derivative of $u$ at $p(t)$ with respect to the vector $\dot{p}(t)$.

Given a smooth solution $(\varphi,\lambda)$ and a fixed $x\in M$, let $\mr{z}^*=[\mr{x}^*,\mr{r}^*] \in \Omega$ be the curve defined by
$\mr{x}^*: t\rightarrow x^*_t \coloneqq \varphi_t(x)$ and $\mr{r}^*: t \rightarrow r^*_t \coloneqq \lambda_t(x)$.  We want to show that for any curve $\mr{z} \in AC^2([0,T];\cone)$ such that $\mr{z} \neq \mr{z}^*$,  $z_0 = z^*_0$ and $z_T = z^*_T$, we have $\mathcal{B}(\mr{z}) > \mathcal{B}(\mr{z}^*)$. We proceed in two steps: first we show that the inequality holds when $\mr{z}$ is smooth and when the geodesics between $z^*_t$ and  $z_t$ are smooth for all $t\in[0,T]$; then we derive sufficient conditions for which the inequality holds also for curves $\mr{z}$ which are farther away from $\mr{z}^*$.

Let $s\in[0,1] \mapsto c(t,s)\in \cone$ be a family of geodesics parameterized by $t\in[0,T]$ such that $c(t,0) = z^*_t$ and $c(t,1) = z_t$. In order for such geodesics to be smooth we need to assume
\begin{equation}\label{eq:dmcondition}
|x^*_t-x_t|<\pi\,, \qquad \forall\, t \in [0,T]\,. 
\end{equation}
Let $J(t,s) \coloneqq \partial_{t} c(t,s)$, which is a Jacobi field when restricted to any geodesic $c(t,\cdot)$ for any fixed $t\in[0,T]$.
Moreover, $J(t,0) = \dot{z}^*_t$ and $J(t,1) = \dot{z}_t$. Hence we want to show that 
\begin{equation} 
\int_0^T | J(t,0)|^2 - \Psi_p(t,c(t,0)) \,\ed t  \leq \int_0^T | J(t,1)|^2 - \Psi_p(t,c(t,1)) \,\ed t \,.
\end{equation}
Let $C \coloneqq \sup_{t\in[0,T]} \sup_{x\in M} | \mr{Hess}\, \Psi_p | $. The Taylor expansion of $\Psi_p(t,c(s,t))$ with respect to $s$ at $s=0$ yields
\begin{equation}\label{eq:taylor}
\Psi_p(t,c(t,1)) - \Psi_p(t,c(t,0)) - \langle \nabla \Psi_p(t,c(t,0)), \partial_s c(t,0) \rangle \leq \frac{C}{2} \int_0^1 | \partial_s c(t,s) |^2 \, \ed s \,.
\end{equation}
Since $\partial_s c(t,s) = 0$ at $t=0$ and $t=T$, by the Poincar\'e inequality we also have 
\begin{equation}
\int_0^T | \partial_s c(t,s)|^2 \, \ed t \leq \frac{T^2}{\pi^2} \int_0^T | \partial_t | \partial_s c(t,s)| |^2 \, \ed t \leq \frac{T^2}{\pi^2} \int_0^T | \nabla_t  \partial_s c(t,s)|^2  \, \ed t \,.
\end{equation}
Let $\dot{J}(t,s) \coloneqq \nabla_s  \partial_t c(t,s)$ and exchanging the order of derivatives in the equation above we obtain
\begin{equation}\label{eq:poincare}
\int_0^T | \partial_s c(t,s)|^2 \, \ed t \leq \frac{T^2}{\pi^2} \int_0^T |\dot{J}(t,s)|^2  \, \ed t \,.
\end{equation}
Integrating over $[0,T]$ equation~\eqref{eq:taylor} and using equation~\eqref{eq:poincare} we get
\begin{equation}\label{eq:taylor1}
\int_0^T \Psi_p(t,c(t,1)) - \Psi_p(t,c(t,0)) - \langle \nabla \Psi_p(t,c(t,0)), \partial_s c(t,0) \rangle \, \ed t \leq \frac{C T^2}{2 \pi^2} \int_0^1 | \dot{J}(t,s) |^2 \, \ed s \,.
\end{equation}
Consider the term involving the gradient of $\Psi_p$. Substituting $\nabla \Psi_p(t,c(t,0)) = - 2\nabla_t \dot{z}^*_t = - 2\nabla_t J(t,0)$, integrating by parts in $t$, and exchanging the order of derivatives for this term yields
\begin{equation}\label{eq:taylor2}
\int_0^T \Psi_p(t,c(t,1)) - \Psi_p(t,c(t,0)) -2 \langle J(t,0), \dot{J}(t,0) \rangle \, \ed t \leq \frac{C T^2}{2 \pi^2} \int_0^1 | \dot{J}(t,s) |^2_{g_\cone} \, \ed s \,.
\end{equation}
Let $f(s) \coloneqq \int_0^T |J(t,s)|^2 \, \ed t$, then
\begin{equation}
f'(0) = \int_0^T 2\langle J(t,0), \dot{J}(t,0) \rangle \, \ed t \, ,
\end{equation}
and
\begin{equation}\label{eq:ftaylor}
\begin{aligned}
f(1) - f(0) - f'(0) &= \int_0^1 (1-s) f''(s) \, \ed s\\
& = \int_0^1 \int _0^T 2 (1-s) (| \dot{J}(t,s) |^2 + \langle J(t,s), \nabla_s \dot{J}(t,s)  \rangle ) \, \ed t \ed s \\
& \geq \int_0^1 \int _0^T 2 (1-s) | \dot{J}(t,s) |^2  \, \ed t \ed s
 \, ,
\end{aligned}
\end{equation}
where the last inequality is due to the fact that for a Jacobi field $J(t,s)$, 
\begin{equation}
\nabla_s \dot{J}(t,s) = - R(J(t,s), \partial_s c(t,s) ) \partial_s c(t,s) \,,  
\end{equation}
where $R$ is the Riemann tensor, which for any tangent vectors $X$ and $Y$ at the same point on the cone over a flat manifold satisfies $\langle X,  R(X, Y )Y \rangle\leq 0$ . Moreover since the Jacobi fields are finite dimensional and $[0,T]\times M$ is compact, there exists a constant $C_0>0$ such that
\begin{equation}
f(1) - f(0) - f'(0) \geq \frac{C_0}{2} \int_0^1 \int _0^T  | \dot{J}(t,s) |^2  \, \ed t \ed s
 \, .
\end{equation}
Combining this with \eqref{eq:taylor2} and rearranging terms we obtain
\begin{equation}
\begin{aligned}
\left(\frac{C_0}{2} - \frac{CT^2}{2\pi^2}  \right) \int_0^1 \int _0^T  | \dot{J}(t,s) |^2  \, \ed t \ed s &+   \int_0^T | J(t,0)|^2- \Psi_p(t,c(t,0)) \,\ed t  \\ & \leq \int_0^T | J(t,1)|^2 - \Psi_p(t,c(t,1)) \,\ed t \,.
\end{aligned}
\end{equation}
Because of the inequality \eqref{eq:hessassumption}, shows that $\mr{z}^*$ is minimizing among all paths $\mr{z}\in \Omega$ which satisfy \eqref{eq:dmcondition} and it is unique when the inequality is strict. 
Note that when $M=S^1_1$, the circle of unit radius, we can identify $\cone$ with $\mathbb{R}^2$ and condition \eqref{eq:dmcondition} is not necessary. Furthermore, since geodesics are straight lines with constant speed, from equation \eqref{eq:ftaylor} we find $C_0 = 2$. This concludes the proof for the case $M=S^1_1$.

Now, assume that for all $x\in M$, $d_\cone(z_{t_0}, z_{t_1}) \leq \epsilon$, for all $t_0,t_1\in[0,T]$. Let
\begin{equation}
B_\delta \coloneqq \bigcap_{t\in [0,T]} \{ q \in \cone \,;\, d_{\cone}(q,z^*_t) \leq \delta \}\,,
\end{equation}
and take $\epsilon < \delta \coloneqq  \frac{r_{min}}{2}$, where $r_{min} \coloneqq \min_{(t,x)\in [0,T]\times M } \lambda_t(x)$. For any $q\in B_\delta$ and any $t\in[0,T]$ the geodesic path between $q$ and $z_t^*$ cannot pass through the apex, since otherwise the distance between the two points should be at least equal to $r_{min}$. In other words, we must have $d_\cone(q,z_t^*)<\pi$ and the path $\mr{z}^*$ is minimizing among all paths $\mr{z}\in \Omega$ contained in $B_\delta$. Moreover, the geodesic path from $z^*_0$ to $z^*_T$ is also included in $B_\delta$. 
Consider the following quantity
\begin{equation}
E(\delta,q, T^{*}) \coloneqq \inf_{p\in \partial B_\delta / \mathcal{C}(\partial M)}  \left\{\inf_{\mr{z}\in AC^2([0,T^*];\cone)} \left\{ \int_0^{T^{*}} | \dot{z}_t |^2 - \Psi_p(t,z_t) \,\ed t \, ;  \, z_0 = q \in B_\delta \,, \, z_T = p \right\}\right\}\,,
\end{equation}
which is the infimum action over the interval $[0,T^{*}]$ among paths starting at a point $q \in B_\delta$ and reaching its boundary $\partial B_\delta$ (but not points on $\partial M$) at time $T^{*}$. Given any path $\mr{z}$ such that $z_0 = z^*_0$ and $z_T = z^*_T$ not contained in $B_\delta$, we have
\begin{equation}\label{eq:lowerbound}
\mathcal{B}(\mr{z}) \geq \inf_{T_1 +T_2 \leq T} ( E(\delta, z^*_0 , T_1) + E(\delta, z^*_T , T_2) ) \,,
\end{equation}
and we want to show that $\mathcal{B}(\mr{z}) > \mathcal{B}(\mr{z}^*)$. We have
\begin{equation}
\begin{aligned}
E(\delta,z^*_0,T_1) & \geq \inf_p \inf_{\mr{z}} \int_0^{T_1} | \dot{z}_t |^2 \,\ed t - (r_{max} + \delta)^2 C T_1 \\
& \geq \frac{(\delta - \epsilon)^2}{T_1} - (r_{max} + \delta)^2 C T_1\,,
\end{aligned}
\end{equation}
where $C \coloneqq \sup_{(t,x)\in [0,T]\times M} |P(t,x)|$ and $r_{max} \coloneqq \max_{(t,x)\in [0,T]\times M} \lambda_t(x)$. Hence, by equation \eqref{eq:lowerbound},
\begin{equation}
\mathcal{B}(\mr{z}) \geq  \frac{4 (\delta - \epsilon)^2}{T} - (r_{max} + \delta)^2 C T\,.
\end{equation}
On the other hand, we can deduce an upper bound for $\mathcal{B}(\mr{z}^*)$ using the geodesic path $\mr{z}^g$ between $z_0^*$ and $z_T^*$, yielding
 \begin{equation}
\mathcal{B}(\mr{z}) \leq \int | \dot{z}^g_t |^2 \,\ed t + r_{max}^{2} C T \leq \frac{ \epsilon^2}{T} + r_{max}^2 C T\,.
\end{equation}
Therefore we find the following sufficient condition for optimality of the path $\mr{z}^*$:
\begin{equation}\label{eq:estdelta}
[r_{max}^2 + (r_{max} + \delta)^2] C T \leq \frac{4 (\delta - \epsilon)^2}{T} - \frac{ \epsilon^2}{T}\,.
\end{equation}
The right-hand side is positive if $\epsilon<2\delta/3$. Hence taking $\epsilon = \delta/2$ and substituting $\delta = \frac{r_{min}}{2}$,
\begin{equation}\label{eq:estdelta1}
\left[r_{max}^2 + \left(r_{max} + \frac{r_{min}}{2}\right)^2\right] C T \leq \frac{3 r_{min}^2 }{8T}\,.
\end{equation}
This is the same as equation \eqref{eq:rhoassumption}. For uniqueness we only need to substitute the inequality in \eqref{eq:estdelta1} by a strict one, which concludes the proof.
\end{proof}

\bibliographystyle{plain}      
\bibliography{refs}   

\end{document}

%% file: shock.tex
\begingroup
  \makeatletter
  \providecommand\color[2][]{%
    \GenericError{(gnuplot) \space\space\space\@spaces}{%
      Package color not loaded in conjunction with
      terminal option `colourtext'%
    }{See the gnuplot documentation for explanation.%
    }{Either use 'blacktext' in gnuplot or load the package
      color.sty in LaTeX.}%
    \renewcommand\color[2][]{}%
  }%
  \providecommand\includegraphics[2][]{%
    \GenericError{(gnuplot) \space\space\space\@spaces}{%
      Package graphicx or graphics not loaded%
    }{See the gnuplot documentation for explanation.%
    }{The gnuplot epslatex terminal needs graphicx.sty or graphics.sty.}%
    \renewcommand\includegraphics[2][]{}%
  }%
  \providecommand\rotatebox[2]{#2}%
  \@ifundefined{ifGPcolor}{%
    \newif\ifGPcolor
    \GPcolortrue
  }{}%
  \@ifundefined{ifGPblacktext}{%
    \newif\ifGPblacktext
    \GPblacktextfalse
  }{}%
  \let\gplgaddtomacro\g@addto@macro
  \gdef\gplbacktext{}%
  \gdef\gplfronttext{}%
  \makeatother
  \ifGPblacktext
    \def\colorrgb#1{}%
    \def\colorgray#1{}%
  \else
    \ifGPcolor
      \def\colorrgb#1{\color[rgb]{#1}}%
      \def\colorgray#1{\color[gray]{#1}}%
      \expandafter\def\csname LTw\endcsname{\color{white}}%
      \expandafter\def\csname LTb\endcsname{\color{black}}%
      \expandafter\def\csname LTa\endcsname{\color{black}}%
      \expandafter\def\csname LT0\endcsname{\color[rgb]{1,0,0}}%
      \expandafter\def\csname LT1\endcsname{\color[rgb]{0,1,0}}%
      \expandafter\def\csname LT2\endcsname{\color[rgb]{0,0,1}}%
      \expandafter\def\csname LT3\endcsname{\color[rgb]{1,0,1}}%
      \expandafter\def\csname LT4\endcsname{\color[rgb]{0,1,1}}%
      \expandafter\def\csname LT5\endcsname{\color[rgb]{1,1,0}}%
      \expandafter\def\csname LT6\endcsname{\color[rgb]{0,0,0}}%
      \expandafter\def\csname LT7\endcsname{\color[rgb]{1,0.3,0}}%
      \expandafter\def\csname LT8\endcsname{\color[rgb]{0.5,0.5,0.5}}%
    \else
      \def\colorrgb#1{\color{black}}%
      \def\colorgray#1{\color[gray]{#1}}%
      \expandafter\def\csname LTw\endcsname{\color{white}}%
      \expandafter\def\csname LTb\endcsname{\color{black}}%
      \expandafter\def\csname LTa\endcsname{\color{black}}%
      \expandafter\def\csname LT0\endcsname{\color{black}}%
      \expandafter\def\csname LT1\endcsname{\color{black}}%
      \expandafter\def\csname LT2\endcsname{\color{black}}%
      \expandafter\def\csname LT3\endcsname{\color{black}}%
      \expandafter\def\csname LT4\endcsname{\color{black}}%
      \expandafter\def\csname LT5\endcsname{\color{black}}%
      \expandafter\def\csname LT6\endcsname{\color{black}}%
      \expandafter\def\csname LT7\endcsname{\color{black}}%
      \expandafter\def\csname LT8\endcsname{\color{black}}%
    \fi
  \fi
    \setlength{\unitlength}{0.0500bp}%
    \ifx\gptboxheight\undefined%
      \newlength{\gptboxheight}%
      \newlength{\gptboxwidth}%
      \newsavebox{\gptboxtext}%
    \fi%
    \setlength{\fboxrule}{0.5pt}%
    \setlength{\fboxsep}{1pt}%
\begin{picture}(5040.00,3772.00)%
    \gplgaddtomacro\gplbacktext{%
      \colorrgb{0.00,0.00,0.00}%
      \put(990,1210){\makebox(0,0)[r]{\strut{}$\dot{\theta}_2(0)$}}%
      \colorrgb{0.00,0.00,0.00}%
      \put(990,2129){\makebox(0,0)[r]{\strut{}$0$}}%
      \colorrgb{0.00,0.00,0.00}%
      \put(990,3048){\makebox(0,0)[r]{\strut{}$\dot{\theta}_1(0)$}}%
      \colorrgb{0.00,0.00,0.00}%
      \put(1169,484){\makebox(0,0){\strut{}$0$}}%
      \colorrgb{0.00,0.00,0.00}%
      \put(2038,484){\makebox(0,0){\strut{}$\pi/2$}}%
      \colorrgb{0.00,0.00,0.00}%
      \put(2906,484){\makebox(0,0){\strut{}$\pi$}}%
      \colorrgb{0.00,0.00,0.00}%
      \put(3775,484){\makebox(0,0){\strut{}$3\pi/2$}}%
      \colorrgb{0.00,0.00,0.00}%
      \put(4643,484){\makebox(0,0){\strut{}$2\pi$}}%
    }%
    \gplgaddtomacro\gplfronttext{%
      \csname LTb\endcsname%
      \put(352,2129){\makebox(0,0){\strut{}$\dot{\theta}$}}%
      \put(2906,154){\makebox(0,0){\strut{}$\theta$}}%
    }%
    \gplbacktext
    \put(0,0){\includegraphics{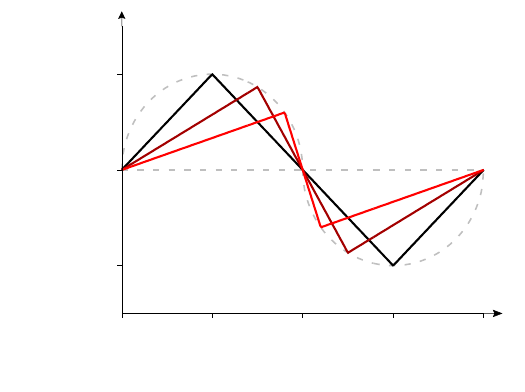}}%
    \gplfronttext
  \end{picture}%
\endgroup

%% file: shockflow.tex
\begingroup
  \makeatletter
  \providecommand\color[2][]{%
    \GenericError{(gnuplot) \space\space\space\@spaces}{%
      Package color not loaded in conjunction with
      terminal option `colourtext'%
    }{See the gnuplot documentation for explanation.%
    }{Either use 'blacktext' in gnuplot or load the package
      color.sty in LaTeX.}%
    \renewcommand\color[2][]{}%
  }%
  \providecommand\includegraphics[2][]{%
    \GenericError{(gnuplot) \space\space\space\@spaces}{%
      Package graphicx or graphics not loaded%
    }{See the gnuplot documentation for explanation.%
    }{The gnuplot epslatex terminal needs graphicx.sty or graphics.sty.}%
    \renewcommand\includegraphics[2][]{}%
  }%
  \providecommand\rotatebox[2]{#2}%
  \@ifundefined{ifGPcolor}{%
    \newif\ifGPcolor
    \GPcolortrue
  }{}%
  \@ifundefined{ifGPblacktext}{%
    \newif\ifGPblacktext
    \GPblacktextfalse
  }{}%
  \let\gplgaddtomacro\g@addto@macro
  \gdef\gplbacktext{}%
  \gdef\gplfronttext{}%
  \makeatother
  \ifGPblacktext
    \def\colorrgb#1{}%
    \def\colorgray#1{}%
  \else
    \ifGPcolor
      \def\colorrgb#1{\color[rgb]{#1}}%
      \def\colorgray#1{\color[gray]{#1}}%
      \expandafter\def\csname LTw\endcsname{\color{white}}%
      \expandafter\def\csname LTb\endcsname{\color{black}}%
      \expandafter\def\csname LTa\endcsname{\color{black}}%
      \expandafter\def\csname LT0\endcsname{\color[rgb]{1,0,0}}%
      \expandafter\def\csname LT1\endcsname{\color[rgb]{0,1,0}}%
      \expandafter\def\csname LT2\endcsname{\color[rgb]{0,0,1}}%
      \expandafter\def\csname LT3\endcsname{\color[rgb]{1,0,1}}%
      \expandafter\def\csname LT4\endcsname{\color[rgb]{0,1,1}}%
      \expandafter\def\csname LT5\endcsname{\color[rgb]{1,1,0}}%
      \expandafter\def\csname LT6\endcsname{\color[rgb]{0,0,0}}%
      \expandafter\def\csname LT7\endcsname{\color[rgb]{1,0.3,0}}%
      \expandafter\def\csname LT8\endcsname{\color[rgb]{0.5,0.5,0.5}}%
    \else
      \def\colorrgb#1{\color{black}}%
      \def\colorgray#1{\color[gray]{#1}}%
      \expandafter\def\csname LTw\endcsname{\color{white}}%
      \expandafter\def\csname LTb\endcsname{\color{black}}%
      \expandafter\def\csname LTa\endcsname{\color{black}}%
      \expandafter\def\csname LT0\endcsname{\color{black}}%
      \expandafter\def\csname LT1\endcsname{\color{black}}%
      \expandafter\def\csname LT2\endcsname{\color{black}}%
      \expandafter\def\csname LT3\endcsname{\color{black}}%
      \expandafter\def\csname LT4\endcsname{\color{black}}%
      \expandafter\def\csname LT5\endcsname{\color{black}}%
      \expandafter\def\csname LT6\endcsname{\color{black}}%
      \expandafter\def\csname LT7\endcsname{\color{black}}%
      \expandafter\def\csname LT8\endcsname{\color{black}}%
    \fi
  \fi
    \setlength{\unitlength}{0.0500bp}%
    \ifx\gptboxheight\undefined%
      \newlength{\gptboxheight}%
      \newlength{\gptboxwidth}%
      \newsavebox{\gptboxtext}%
    \fi%
    \setlength{\fboxrule}{0.5pt}%
    \setlength{\fboxsep}{1pt}%
\begin{picture}(4636.00,3772.00)%
    \gplgaddtomacro\gplbacktext{%
      \colorrgb{0.00,0.00,0.00}%
      \put(594,751){\makebox(0,0)[r]{\strut{}$0$}}%
      \colorrgb{0.00,0.00,0.00}%
      \put(594,2956){\makebox(0,0)[r]{\strut{}$T$}}%
      \colorrgb{0.00,0.00,0.00}%
      \put(594,3287){\makebox(0,0)[r]{\strut{}$T+\epsilon$}}%
      \colorrgb{0.00,0.00,0.00}%
      \put(773,484){\makebox(0,0){\strut{}$0$}}%
      \colorrgb{0.00,0.00,0.00}%
      \put(1640,484){\makebox(0,0){\strut{}$\pi/2$}}%
      \colorrgb{0.00,0.00,0.00}%
      \put(2506,484){\makebox(0,0){\strut{}$\pi$}}%
      \colorrgb{0.00,0.00,0.00}%
      \put(3373,484){\makebox(0,0){\strut{}$3\pi/2$}}%
      \colorrgb{0.00,0.00,0.00}%
      \put(4239,484){\makebox(0,0){\strut{}$2\pi$}}%
    }%
    \gplgaddtomacro\gplfronttext{%
      \csname LTb\endcsname%
      \put(220,2129){\makebox(0,0){\strut{}$t$}}%
      \put(2506,154){\makebox(0,0){\strut{}$\theta$}}%
    }%
    \gplbacktext
    \put(0,0){\includegraphics{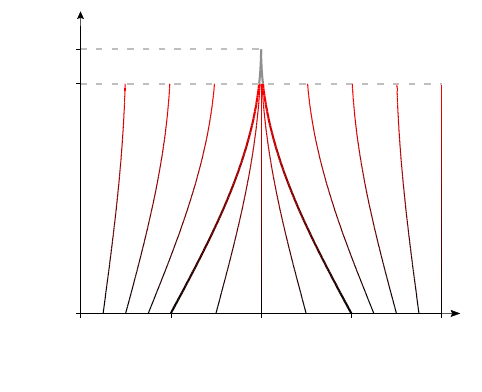}}%
    \gplfronttext
  \end{picture}%
\endgroup

%% file: shockcone.tex
\begingroup
  \makeatletter
  \providecommand\color[2][]{%
    \GenericError{(gnuplot) \space\space\space\@spaces}{%
      Package color not loaded in conjunction with
      terminal option `colourtext'%
    }{See the gnuplot documentation for explanation.%
    }{Either use 'blacktext' in gnuplot or load the package
      color.sty in LaTeX.}%
    \renewcommand\color[2][]{}%
  }%
  \providecommand\includegraphics[2][]{%
    \GenericError{(gnuplot) \space\space\space\@spaces}{%
      Package graphicx or graphics not loaded%
    }{See the gnuplot documentation for explanation.%
    }{The gnuplot epslatex terminal needs graphicx.sty or graphics.sty.}%
    \renewcommand\includegraphics[2][]{}%
  }%
  \providecommand\rotatebox[2]{#2}%
  \@ifundefined{ifGPcolor}{%
    \newif\ifGPcolor
    \GPcolortrue
  }{}%
  \@ifundefined{ifGPblacktext}{%
    \newif\ifGPblacktext
    \GPblacktextfalse
  }{}%
  \let\gplgaddtomacro\g@addto@macro
  \gdef\gplbacktext{}%
  \gdef\gplfronttext{}%
  \makeatother
  \ifGPblacktext
    \def\colorrgb#1{}%
    \def\colorgray#1{}%
  \else
    \ifGPcolor
      \def\colorrgb#1{\color[rgb]{#1}}%
      \def\colorgray#1{\color[gray]{#1}}%
      \expandafter\def\csname LTw\endcsname{\color{white}}%
      \expandafter\def\csname LTb\endcsname{\color{black}}%
      \expandafter\def\csname LTa\endcsname{\color{black}}%
      \expandafter\def\csname LT0\endcsname{\color[rgb]{1,0,0}}%
      \expandafter\def\csname LT1\endcsname{\color[rgb]{0,1,0}}%
      \expandafter\def\csname LT2\endcsname{\color[rgb]{0,0,1}}%
      \expandafter\def\csname LT3\endcsname{\color[rgb]{1,0,1}}%
      \expandafter\def\csname LT4\endcsname{\color[rgb]{0,1,1}}%
      \expandafter\def\csname LT5\endcsname{\color[rgb]{1,1,0}}%
      \expandafter\def\csname LT6\endcsname{\color[rgb]{0,0,0}}%
      \expandafter\def\csname LT7\endcsname{\color[rgb]{1,0.3,0}}%
      \expandafter\def\csname LT8\endcsname{\color[rgb]{0.5,0.5,0.5}}%
    \else
      \def\colorrgb#1{\color{black}}%
      \def\colorgray#1{\color[gray]{#1}}%
      \expandafter\def\csname LTw\endcsname{\color{white}}%
      \expandafter\def\csname LTb\endcsname{\color{black}}%
      \expandafter\def\csname LTa\endcsname{\color{black}}%
      \expandafter\def\csname LT0\endcsname{\color{black}}%
      \expandafter\def\csname LT1\endcsname{\color{black}}%
      \expandafter\def\csname LT2\endcsname{\color{black}}%
      \expandafter\def\csname LT3\endcsname{\color{black}}%
      \expandafter\def\csname LT4\endcsname{\color{black}}%
      \expandafter\def\csname LT5\endcsname{\color{black}}%
      \expandafter\def\csname LT6\endcsname{\color{black}}%
      \expandafter\def\csname LT7\endcsname{\color{black}}%
      \expandafter\def\csname LT8\endcsname{\color{black}}%
    \fi
  \fi
    \setlength{\unitlength}{0.0500bp}%
    \ifx\gptboxheight\undefined%
      \newlength{\gptboxheight}%
      \newlength{\gptboxwidth}%
      \newsavebox{\gptboxtext}%
    \fi%
    \setlength{\fboxrule}{0.5pt}%
    \setlength{\fboxsep}{1pt}%
\begin{picture}(3340.00,3772.00)%
    \gplgaddtomacro\gplbacktext{%
      \csname LTb\endcsname%
      \put(2943,1732){\makebox(0,0)[l]{\strut{}$r$}}%
      \put(2098,2078){\makebox(0,0)[l]{\strut{}$\theta$}}%
    }%
    \gplgaddtomacro\gplfronttext{%
    }%
    \gplbacktext
    \put(0,0){\includegraphics{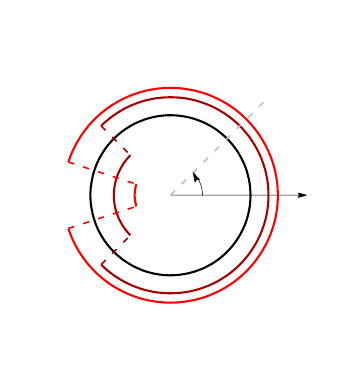}}%
    \gplfronttext
  \end{picture}%
\endgroup

%% file: shockflowm.tex
\begingroup
  \makeatletter
  \providecommand\color[2][]{%
    \GenericError{(gnuplot) \space\space\space\@spaces}{%
      Package color not loaded in conjunction with
      terminal option `colourtext'%
    }{See the gnuplot documentation for explanation.%
    }{Either use 'blacktext' in gnuplot or load the package
      color.sty in LaTeX.}%
    \renewcommand\color[2][]{}%
  }%
  \providecommand\includegraphics[2][]{%
    \GenericError{(gnuplot) \space\space\space\@spaces}{%
      Package graphicx or graphics not loaded%
    }{See the gnuplot documentation for explanation.%
    }{The gnuplot epslatex terminal needs graphicx.sty or graphics.sty.}%
    \renewcommand\includegraphics[2][]{}%
  }%
  \providecommand\rotatebox[2]{#2}%
  \@ifundefined{ifGPcolor}{%
    \newif\ifGPcolor
    \GPcolortrue
  }{}%
  \@ifundefined{ifGPblacktext}{%
    \newif\ifGPblacktext
    \GPblacktextfalse
  }{}%
  \let\gplgaddtomacro\g@addto@macro
  \gdef\gplbacktext{}%
  \gdef\gplfronttext{}%
  \makeatother
  \ifGPblacktext
    \def\colorrgb#1{}%
    \def\colorgray#1{}%
  \else
    \ifGPcolor
      \def\colorrgb#1{\color[rgb]{#1}}%
      \def\colorgray#1{\color[gray]{#1}}%
      \expandafter\def\csname LTw\endcsname{\color{white}}%
      \expandafter\def\csname LTb\endcsname{\color{black}}%
      \expandafter\def\csname LTa\endcsname{\color{black}}%
      \expandafter\def\csname LT0\endcsname{\color[rgb]{1,0,0}}%
      \expandafter\def\csname LT1\endcsname{\color[rgb]{0,1,0}}%
      \expandafter\def\csname LT2\endcsname{\color[rgb]{0,0,1}}%
      \expandafter\def\csname LT3\endcsname{\color[rgb]{1,0,1}}%
      \expandafter\def\csname LT4\endcsname{\color[rgb]{0,1,1}}%
      \expandafter\def\csname LT5\endcsname{\color[rgb]{1,1,0}}%
      \expandafter\def\csname LT6\endcsname{\color[rgb]{0,0,0}}%
      \expandafter\def\csname LT7\endcsname{\color[rgb]{1,0.3,0}}%
      \expandafter\def\csname LT8\endcsname{\color[rgb]{0.5,0.5,0.5}}%
    \else
      \def\colorrgb#1{\color{black}}%
      \def\colorgray#1{\color[gray]{#1}}%
      \expandafter\def\csname LTw\endcsname{\color{white}}%
      \expandafter\def\csname LTb\endcsname{\color{black}}%
      \expandafter\def\csname LTa\endcsname{\color{black}}%
      \expandafter\def\csname LT0\endcsname{\color{black}}%
      \expandafter\def\csname LT1\endcsname{\color{black}}%
      \expandafter\def\csname LT2\endcsname{\color{black}}%
      \expandafter\def\csname LT3\endcsname{\color{black}}%
      \expandafter\def\csname LT4\endcsname{\color{black}}%
      \expandafter\def\csname LT5\endcsname{\color{black}}%
      \expandafter\def\csname LT6\endcsname{\color{black}}%
      \expandafter\def\csname LT7\endcsname{\color{black}}%
      \expandafter\def\csname LT8\endcsname{\color{black}}%
    \fi
  \fi
    \setlength{\unitlength}{0.0500bp}%
    \ifx\gptboxheight\undefined%
      \newlength{\gptboxheight}%
      \newlength{\gptboxwidth}%
      \newsavebox{\gptboxtext}%
    \fi%
    \setlength{\fboxrule}{0.5pt}%
    \setlength{\fboxsep}{1pt}%
\begin{picture}(4608.00,3772.00)%
    \gplgaddtomacro\gplbacktext{%
      \colorrgb{0.00,0.00,0.00}%
      \put(330,751){\makebox(0,0)[r]{\strut{}$0$}}%
      \colorrgb{0.00,0.00,0.00}%
      \put(330,3256){\makebox(0,0)[r]{\strut{}$T$}}%
      \colorrgb{0.00,0.00,0.00}%
      \put(509,484){\makebox(0,0){\strut{}$0$}}%
      \colorrgb{0.00,0.00,0.00}%
      \put(1435,484){\makebox(0,0){\strut{}$\pi/2$}}%
      \colorrgb{0.00,0.00,0.00}%
      \put(2360,484){\makebox(0,0){\strut{}$\pi$}}%
      \colorrgb{0.00,0.00,0.00}%
      \put(3286,484){\makebox(0,0){\strut{}$3\pi/2$}}%
      \colorrgb{0.00,0.00,0.00}%
      \put(4211,484){\makebox(0,0){\strut{}$2\pi$}}%
    }%
    \gplgaddtomacro\gplfronttext{%
      \csname LTb\endcsname%
      \put(220,2129){\makebox(0,0){\strut{}$t$}}%
      \put(2360,154){\makebox(0,0){\strut{}$\theta$}}%
    }%
    \gplbacktext
    \put(0,0){\includegraphics{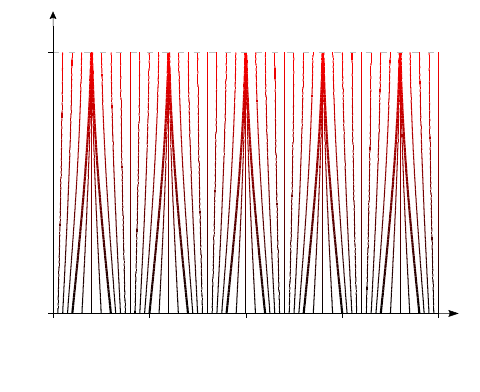}}%
    \gplfronttext
  \end{picture}%
\endgroup

%% file: shockconem.tex
\begingroup
  \makeatletter
  \providecommand\color[2][]{%
    \GenericError{(gnuplot) \space\space\space\@spaces}{%
      Package color not loaded in conjunction with
      terminal option `colourtext'%
    }{See the gnuplot documentation for explanation.%
    }{Either use 'blacktext' in gnuplot or load the package
      color.sty in LaTeX.}%
    \renewcommand\color[2][]{}%
  }%
  \providecommand\includegraphics[2][]{%
    \GenericError{(gnuplot) \space\space\space\@spaces}{%
      Package graphicx or graphics not loaded%
    }{See the gnuplot documentation for explanation.%
    }{The gnuplot epslatex terminal needs graphicx.sty or graphics.sty.}%
    \renewcommand\includegraphics[2][]{}%
  }%
  \providecommand\rotatebox[2]{#2}%
  \@ifundefined{ifGPcolor}{%
    \newif\ifGPcolor
    \GPcolortrue
  }{}%
  \@ifundefined{ifGPblacktext}{%
    \newif\ifGPblacktext
    \GPblacktextfalse
  }{}%
  \let\gplgaddtomacro\g@addto@macro
  \gdef\gplbacktext{}%
  \gdef\gplfronttext{}%
  \makeatother
  \ifGPblacktext
    \def\colorrgb#1{}%
    \def\colorgray#1{}%
  \else
    \ifGPcolor
      \def\colorrgb#1{\color[rgb]{#1}}%
      \def\colorgray#1{\color[gray]{#1}}%
      \expandafter\def\csname LTw\endcsname{\color{white}}%
      \expandafter\def\csname LTb\endcsname{\color{black}}%
      \expandafter\def\csname LTa\endcsname{\color{black}}%
      \expandafter\def\csname LT0\endcsname{\color[rgb]{1,0,0}}%
      \expandafter\def\csname LT1\endcsname{\color[rgb]{0,1,0}}%
      \expandafter\def\csname LT2\endcsname{\color[rgb]{0,0,1}}%
      \expandafter\def\csname LT3\endcsname{\color[rgb]{1,0,1}}%
      \expandafter\def\csname LT4\endcsname{\color[rgb]{0,1,1}}%
      \expandafter\def\csname LT5\endcsname{\color[rgb]{1,1,0}}%
      \expandafter\def\csname LT6\endcsname{\color[rgb]{0,0,0}}%
      \expandafter\def\csname LT7\endcsname{\color[rgb]{1,0.3,0}}%
      \expandafter\def\csname LT8\endcsname{\color[rgb]{0.5,0.5,0.5}}%
    \else
      \def\colorrgb#1{\color{black}}%
      \def\colorgray#1{\color[gray]{#1}}%
      \expandafter\def\csname LTw\endcsname{\color{white}}%
      \expandafter\def\csname LTb\endcsname{\color{black}}%
      \expandafter\def\csname LTa\endcsname{\color{black}}%
      \expandafter\def\csname LT0\endcsname{\color{black}}%
      \expandafter\def\csname LT1\endcsname{\color{black}}%
      \expandafter\def\csname LT2\endcsname{\color{black}}%
      \expandafter\def\csname LT3\endcsname{\color{black}}%
      \expandafter\def\csname LT4\endcsname{\color{black}}%
      \expandafter\def\csname LT5\endcsname{\color{black}}%
      \expandafter\def\csname LT6\endcsname{\color{black}}%
      \expandafter\def\csname LT7\endcsname{\color{black}}%
      \expandafter\def\csname LT8\endcsname{\color{black}}%
    \fi
  \fi
    \setlength{\unitlength}{0.0500bp}%
    \ifx\gptboxheight\undefined%
      \newlength{\gptboxheight}%
      \newlength{\gptboxwidth}%
      \newsavebox{\gptboxtext}%
    \fi%
    \setlength{\fboxrule}{0.5pt}%
    \setlength{\fboxsep}{1pt}%
\begin{picture}(3340.00,3772.00)%
    \gplgaddtomacro\gplbacktext{%
    }%
    \gplgaddtomacro\gplfronttext{%
    }%
    \gplbacktext
    \put(0,0){\includegraphics{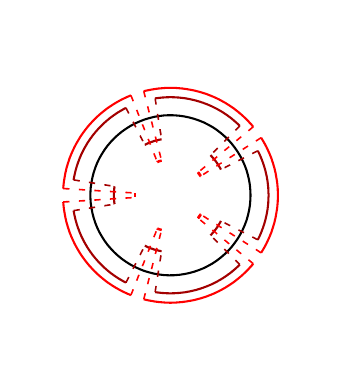}}%
    \gplfronttext
  \end{picture}%
\endgroup

%% file: shockflowmm.tex
\begingroup
  \makeatletter
  \providecommand\color[2][]{%
    \GenericError{(gnuplot) \space\space\space\@spaces}{%
      Package color not loaded in conjunction with
      terminal option `colourtext'%
    }{See the gnuplot documentation for explanation.%
    }{Either use 'blacktext' in gnuplot or load the package
      color.sty in LaTeX.}%
    \renewcommand\color[2][]{}%
  }%
  \providecommand\includegraphics[2][]{%
    \GenericError{(gnuplot) \space\space\space\@spaces}{%
      Package graphicx or graphics not loaded%
    }{See the gnuplot documentation for explanation.%
    }{The gnuplot epslatex terminal needs graphicx.sty or graphics.sty.}%
    \renewcommand\includegraphics[2][]{}%
  }%
  \providecommand\rotatebox[2]{#2}%
  \@ifundefined{ifGPcolor}{%
    \newif\ifGPcolor
    \GPcolortrue
  }{}%
  \@ifundefined{ifGPblacktext}{%
    \newif\ifGPblacktext
    \GPblacktextfalse
  }{}%
  \let\gplgaddtomacro\g@addto@macro
  \gdef\gplbacktext{}%
  \gdef\gplfronttext{}%
  \makeatother
  \ifGPblacktext
    \def\colorrgb#1{}%
    \def\colorgray#1{}%
  \else
    \ifGPcolor
      \def\colorrgb#1{\color[rgb]{#1}}%
      \def\colorgray#1{\color[gray]{#1}}%
      \expandafter\def\csname LTw\endcsname{\color{white}}%
      \expandafter\def\csname LTb\endcsname{\color{black}}%
      \expandafter\def\csname LTa\endcsname{\color{black}}%
      \expandafter\def\csname LT0\endcsname{\color[rgb]{1,0,0}}%
      \expandafter\def\csname LT1\endcsname{\color[rgb]{0,1,0}}%
      \expandafter\def\csname LT2\endcsname{\color[rgb]{0,0,1}}%
      \expandafter\def\csname LT3\endcsname{\color[rgb]{1,0,1}}%
      \expandafter\def\csname LT4\endcsname{\color[rgb]{0,1,1}}%
      \expandafter\def\csname LT5\endcsname{\color[rgb]{1,1,0}}%
      \expandafter\def\csname LT6\endcsname{\color[rgb]{0,0,0}}%
      \expandafter\def\csname LT7\endcsname{\color[rgb]{1,0.3,0}}%
      \expandafter\def\csname LT8\endcsname{\color[rgb]{0.5,0.5,0.5}}%
    \else
      \def\colorrgb#1{\color{black}}%
      \def\colorgray#1{\color[gray]{#1}}%
      \expandafter\def\csname LTw\endcsname{\color{white}}%
      \expandafter\def\csname LTb\endcsname{\color{black}}%
      \expandafter\def\csname LTa\endcsname{\color{black}}%
      \expandafter\def\csname LT0\endcsname{\color{black}}%
      \expandafter\def\csname LT1\endcsname{\color{black}}%
      \expandafter\def\csname LT2\endcsname{\color{black}}%
      \expandafter\def\csname LT3\endcsname{\color{black}}%
      \expandafter\def\csname LT4\endcsname{\color{black}}%
      \expandafter\def\csname LT5\endcsname{\color{black}}%
      \expandafter\def\csname LT6\endcsname{\color{black}}%
      \expandafter\def\csname LT7\endcsname{\color{black}}%
      \expandafter\def\csname LT8\endcsname{\color{black}}%
    \fi
  \fi
    \setlength{\unitlength}{0.0500bp}%
    \ifx\gptboxheight\undefined%
      \newlength{\gptboxheight}%
      \newlength{\gptboxwidth}%
      \newsavebox{\gptboxtext}%
    \fi%
    \setlength{\fboxrule}{0.5pt}%
    \setlength{\fboxsep}{1pt}%
\begin{picture}(4608.00,3772.00)%
    \gplgaddtomacro\gplbacktext{%
      \colorrgb{0.00,0.00,0.00}%
      \put(330,751){\makebox(0,0)[r]{\strut{}$0$}}%
      \colorrgb{0.00,0.00,0.00}%
      \put(330,3256){\makebox(0,0)[r]{\strut{}$T$}}%
      \colorrgb{0.00,0.00,0.00}%
      \put(509,484){\makebox(0,0){\strut{}$0$}}%
      \colorrgb{0.00,0.00,0.00}%
      \put(1435,484){\makebox(0,0){\strut{}$\pi/2$}}%
      \colorrgb{0.00,0.00,0.00}%
      \put(2360,484){\makebox(0,0){\strut{}$\pi$}}%
      \colorrgb{0.00,0.00,0.00}%
      \put(3286,484){\makebox(0,0){\strut{}$3\pi/2$}}%
      \colorrgb{0.00,0.00,0.00}%
      \put(4211,484){\makebox(0,0){\strut{}$2\pi$}}%
    }%
    \gplgaddtomacro\gplfronttext{%
      \csname LTb\endcsname%
      \put(220,2129){\makebox(0,0){\strut{}$t$}}%
      \put(2360,154){\makebox(0,0){\strut{}$\theta$}}%
    }%
    \gplbacktext
    \put(0,0){\includegraphics{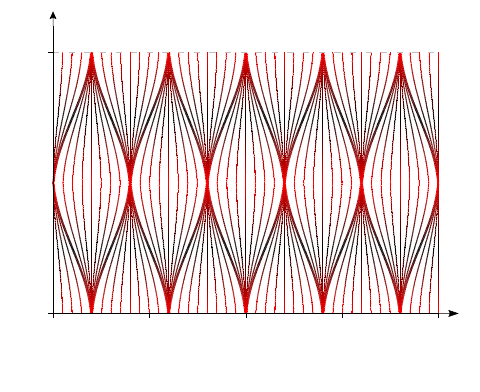}}%
    \gplfronttext
  \end{picture}%
\endgroup

%% file: shockflow2t0.tex
\begingroup
  \makeatletter
  \providecommand\color[2][]{%
    \GenericError{(gnuplot) \space\space\space\@spaces}{%
      Package color not loaded in conjunction with
      terminal option `colourtext'%
    }{See the gnuplot documentation for explanation.%
    }{Either use 'blacktext' in gnuplot or load the package
      color.sty in LaTeX.}%
    \renewcommand\color[2][]{}%
  }%
  \providecommand\includegraphics[2][]{%
    \GenericError{(gnuplot) \space\space\space\@spaces}{%
      Package graphicx or graphics not loaded%
    }{See the gnuplot documentation for explanation.%
    }{The gnuplot epslatex terminal needs graphicx.sty or graphics.sty.}%
    \renewcommand\includegraphics[2][]{}%
  }%
  \providecommand\rotatebox[2]{#2}%
  \@ifundefined{ifGPcolor}{%
    \newif\ifGPcolor
    \GPcolortrue
  }{}%
  \@ifundefined{ifGPblacktext}{%
    \newif\ifGPblacktext
    \GPblacktextfalse
  }{}%
  \let\gplgaddtomacro\g@addto@macro
  \gdef\gplbacktext{}%
  \gdef\gplfronttext{}%
  \makeatother
  \ifGPblacktext
    \def\colorrgb#1{}%
    \def\colorgray#1{}%
  \else
    \ifGPcolor
      \def\colorrgb#1{\color[rgb]{#1}}%
      \def\colorgray#1{\color[gray]{#1}}%
      \expandafter\def\csname LTw\endcsname{\color{white}}%
      \expandafter\def\csname LTb\endcsname{\color{black}}%
      \expandafter\def\csname LTa\endcsname{\color{black}}%
      \expandafter\def\csname LT0\endcsname{\color[rgb]{1,0,0}}%
      \expandafter\def\csname LT1\endcsname{\color[rgb]{0,1,0}}%
      \expandafter\def\csname LT2\endcsname{\color[rgb]{0,0,1}}%
      \expandafter\def\csname LT3\endcsname{\color[rgb]{1,0,1}}%
      \expandafter\def\csname LT4\endcsname{\color[rgb]{0,1,1}}%
      \expandafter\def\csname LT5\endcsname{\color[rgb]{1,1,0}}%
      \expandafter\def\csname LT6\endcsname{\color[rgb]{0,0,0}}%
      \expandafter\def\csname LT7\endcsname{\color[rgb]{1,0.3,0}}%
      \expandafter\def\csname LT8\endcsname{\color[rgb]{0.5,0.5,0.5}}%
    \else
      \def\colorrgb#1{\color{black}}%
      \def\colorgray#1{\color[gray]{#1}}%
      \expandafter\def\csname LTw\endcsname{\color{white}}%
      \expandafter\def\csname LTb\endcsname{\color{black}}%
      \expandafter\def\csname LTa\endcsname{\color{black}}%
      \expandafter\def\csname LT0\endcsname{\color{black}}%
      \expandafter\def\csname LT1\endcsname{\color{black}}%
      \expandafter\def\csname LT2\endcsname{\color{black}}%
      \expandafter\def\csname LT3\endcsname{\color{black}}%
      \expandafter\def\csname LT4\endcsname{\color{black}}%
      \expandafter\def\csname LT5\endcsname{\color{black}}%
      \expandafter\def\csname LT6\endcsname{\color{black}}%
      \expandafter\def\csname LT7\endcsname{\color{black}}%
      \expandafter\def\csname LT8\endcsname{\color{black}}%
    \fi
  \fi
    \setlength{\unitlength}{0.0500bp}%
    \ifx\gptboxheight\undefined%
      \newlength{\gptboxheight}%
      \newlength{\gptboxwidth}%
      \newsavebox{\gptboxtext}%
    \fi%
    \setlength{\fboxrule}{0.5pt}%
    \setlength{\fboxsep}{1pt}%
\begin{picture}(2736.00,2590.00)%
    \gplgaddtomacro\gplbacktext{%
      \colorrgb{0.00,0.00,0.00}%
      \put(462,751){\makebox(0,0)[r]{\strut{}$0$}}%
      \colorrgb{0.00,0.00,0.00}%
      \put(462,1538){\makebox(0,0)[r]{\strut{}$\pi$}}%
      \colorrgb{0.00,0.00,0.00}%
      \put(462,2325){\makebox(0,0)[r]{\strut{}$2\pi$}}%
      \colorrgb{0.00,0.00,0.00}%
      \put(641,484){\makebox(0,0){\strut{}$0$}}%
      \colorrgb{0.00,0.00,0.00}%
      \put(1490,484){\makebox(0,0){\strut{}$\pi$}}%
      \colorrgb{0.00,0.00,0.00}%
      \put(2339,484){\makebox(0,0){\strut{}$2\pi$}}%
    }%
    \gplgaddtomacro\gplfronttext{%
      \csname LTb\endcsname%
      \put(220,1538){\makebox(0,0){\strut{}$\theta$}}%
      \put(1490,154){\makebox(0,0){\strut{}$\phi$}}%
    }%
    \gplbacktext
    \put(0,0){\includegraphics{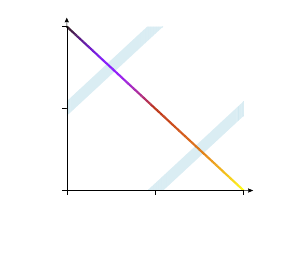}}%
    \gplfronttext
  \end{picture}%
\endgroup

%% file: shockflow2t1.tex
\begingroup
  \makeatletter
  \providecommand\color[2][]{%
    \GenericError{(gnuplot) \space\space\space\@spaces}{%
      Package color not loaded in conjunction with
      terminal option `colourtext'%
    }{See the gnuplot documentation for explanation.%
    }{Either use 'blacktext' in gnuplot or load the package
      color.sty in LaTeX.}%
    \renewcommand\color[2][]{}%
  }%
  \providecommand\includegraphics[2][]{%
    \GenericError{(gnuplot) \space\space\space\@spaces}{%
      Package graphicx or graphics not loaded%
    }{See the gnuplot documentation for explanation.%
    }{The gnuplot epslatex terminal needs graphicx.sty or graphics.sty.}%
    \renewcommand\includegraphics[2][]{}%
  }%
  \providecommand\rotatebox[2]{#2}%
  \@ifundefined{ifGPcolor}{%
    \newif\ifGPcolor
    \GPcolortrue
  }{}%
  \@ifundefined{ifGPblacktext}{%
    \newif\ifGPblacktext
    \GPblacktextfalse
  }{}%
  \let\gplgaddtomacro\g@addto@macro
  \gdef\gplbacktext{}%
  \gdef\gplfronttext{}%
  \makeatother
  \ifGPblacktext
    \def\colorrgb#1{}%
    \def\colorgray#1{}%
  \else
    \ifGPcolor
      \def\colorrgb#1{\color[rgb]{#1}}%
      \def\colorgray#1{\color[gray]{#1}}%
      \expandafter\def\csname LTw\endcsname{\color{white}}%
      \expandafter\def\csname LTb\endcsname{\color{black}}%
      \expandafter\def\csname LTa\endcsname{\color{black}}%
      \expandafter\def\csname LT0\endcsname{\color[rgb]{1,0,0}}%
      \expandafter\def\csname LT1\endcsname{\color[rgb]{0,1,0}}%
      \expandafter\def\csname LT2\endcsname{\color[rgb]{0,0,1}}%
      \expandafter\def\csname LT3\endcsname{\color[rgb]{1,0,1}}%
      \expandafter\def\csname LT4\endcsname{\color[rgb]{0,1,1}}%
      \expandafter\def\csname LT5\endcsname{\color[rgb]{1,1,0}}%
      \expandafter\def\csname LT6\endcsname{\color[rgb]{0,0,0}}%
      \expandafter\def\csname LT7\endcsname{\color[rgb]{1,0.3,0}}%
      \expandafter\def\csname LT8\endcsname{\color[rgb]{0.5,0.5,0.5}}%
    \else
      \def\colorrgb#1{\color{black}}%
      \def\colorgray#1{\color[gray]{#1}}%
      \expandafter\def\csname LTw\endcsname{\color{white}}%
      \expandafter\def\csname LTb\endcsname{\color{black}}%
      \expandafter\def\csname LTa\endcsname{\color{black}}%
      \expandafter\def\csname LT0\endcsname{\color{black}}%
      \expandafter\def\csname LT1\endcsname{\color{black}}%
      \expandafter\def\csname LT2\endcsname{\color{black}}%
      \expandafter\def\csname LT3\endcsname{\color{black}}%
      \expandafter\def\csname LT4\endcsname{\color{black}}%
      \expandafter\def\csname LT5\endcsname{\color{black}}%
      \expandafter\def\csname LT6\endcsname{\color{black}}%
      \expandafter\def\csname LT7\endcsname{\color{black}}%
      \expandafter\def\csname LT8\endcsname{\color{black}}%
    \fi
  \fi
    \setlength{\unitlength}{0.0500bp}%
    \ifx\gptboxheight\undefined%
      \newlength{\gptboxheight}%
      \newlength{\gptboxwidth}%
      \newsavebox{\gptboxtext}%
    \fi%
    \setlength{\fboxrule}{0.5pt}%
    \setlength{\fboxsep}{1pt}%
\begin{picture}(2736.00,2590.00)%
    \gplgaddtomacro\gplbacktext{%
      \colorrgb{0.00,0.00,0.00}%
      \put(462,751){\makebox(0,0)[r]{\strut{}$0$}}%
      \colorrgb{0.00,0.00,0.00}%
      \put(462,1538){\makebox(0,0)[r]{\strut{}$\pi$}}%
      \colorrgb{0.00,0.00,0.00}%
      \put(462,2325){\makebox(0,0)[r]{\strut{}$2\pi$}}%
      \colorrgb{0.00,0.00,0.00}%
      \put(641,484){\makebox(0,0){\strut{}$0$}}%
      \colorrgb{0.00,0.00,0.00}%
      \put(1490,484){\makebox(0,0){\strut{}$\pi$}}%
      \colorrgb{0.00,0.00,0.00}%
      \put(2339,484){\makebox(0,0){\strut{}$2\pi$}}%
    }%
    \gplgaddtomacro\gplfronttext{%
      \csname LTb\endcsname%
      \put(220,1538){\makebox(0,0){\strut{}$\theta$}}%
      \put(1490,154){\makebox(0,0){\strut{}$\phi$}}%
    }%
    \gplbacktext
    \put(0,0){\includegraphics{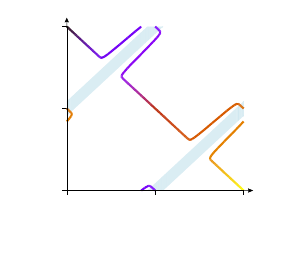}}%
    \gplfronttext
  \end{picture}%
\endgroup

%% file: shockflow2t2.tex
\begingroup
  \makeatletter
  \providecommand\color[2][]{%
    \GenericError{(gnuplot) \space\space\space\@spaces}{%
      Package color not loaded in conjunction with
      terminal option `colourtext'%
    }{See the gnuplot documentation for explanation.%
    }{Either use 'blacktext' in gnuplot or load the package
      color.sty in LaTeX.}%
    \renewcommand\color[2][]{}%
  }%
  \providecommand\includegraphics[2][]{%
    \GenericError{(gnuplot) \space\space\space\@spaces}{%
      Package graphicx or graphics not loaded%
    }{See the gnuplot documentation for explanation.%
    }{The gnuplot epslatex terminal needs graphicx.sty or graphics.sty.}%
    \renewcommand\includegraphics[2][]{}%
  }%
  \providecommand\rotatebox[2]{#2}%
  \@ifundefined{ifGPcolor}{%
    \newif\ifGPcolor
    \GPcolortrue
  }{}%
  \@ifundefined{ifGPblacktext}{%
    \newif\ifGPblacktext
    \GPblacktextfalse
  }{}%
  \let\gplgaddtomacro\g@addto@macro
  \gdef\gplbacktext{}%
  \gdef\gplfronttext{}%
  \makeatother
  \ifGPblacktext
    \def\colorrgb#1{}%
    \def\colorgray#1{}%
  \else
    \ifGPcolor
      \def\colorrgb#1{\color[rgb]{#1}}%
      \def\colorgray#1{\color[gray]{#1}}%
      \expandafter\def\csname LTw\endcsname{\color{white}}%
      \expandafter\def\csname LTb\endcsname{\color{black}}%
      \expandafter\def\csname LTa\endcsname{\color{black}}%
      \expandafter\def\csname LT0\endcsname{\color[rgb]{1,0,0}}%
      \expandafter\def\csname LT1\endcsname{\color[rgb]{0,1,0}}%
      \expandafter\def\csname LT2\endcsname{\color[rgb]{0,0,1}}%
      \expandafter\def\csname LT3\endcsname{\color[rgb]{1,0,1}}%
      \expandafter\def\csname LT4\endcsname{\color[rgb]{0,1,1}}%
      \expandafter\def\csname LT5\endcsname{\color[rgb]{1,1,0}}%
      \expandafter\def\csname LT6\endcsname{\color[rgb]{0,0,0}}%
      \expandafter\def\csname LT7\endcsname{\color[rgb]{1,0.3,0}}%
      \expandafter\def\csname LT8\endcsname{\color[rgb]{0.5,0.5,0.5}}%
    \else
      \def\colorrgb#1{\color{black}}%
      \def\colorgray#1{\color[gray]{#1}}%
      \expandafter\def\csname LTw\endcsname{\color{white}}%
      \expandafter\def\csname LTb\endcsname{\color{black}}%
      \expandafter\def\csname LTa\endcsname{\color{black}}%
      \expandafter\def\csname LT0\endcsname{\color{black}}%
      \expandafter\def\csname LT1\endcsname{\color{black}}%
      \expandafter\def\csname LT2\endcsname{\color{black}}%
      \expandafter\def\csname LT3\endcsname{\color{black}}%
      \expandafter\def\csname LT4\endcsname{\color{black}}%
      \expandafter\def\csname LT5\endcsname{\color{black}}%
      \expandafter\def\csname LT6\endcsname{\color{black}}%
      \expandafter\def\csname LT7\endcsname{\color{black}}%
      \expandafter\def\csname LT8\endcsname{\color{black}}%
    \fi
  \fi
    \setlength{\unitlength}{0.0500bp}%
    \ifx\gptboxheight\undefined%
      \newlength{\gptboxheight}%
      \newlength{\gptboxwidth}%
      \newsavebox{\gptboxtext}%
    \fi%
    \setlength{\fboxrule}{0.5pt}%
    \setlength{\fboxsep}{1pt}%
\begin{picture}(2736.00,2590.00)%
    \gplgaddtomacro\gplbacktext{%
      \colorrgb{0.00,0.00,0.00}%
      \put(462,751){\makebox(0,0)[r]{\strut{}$0$}}%
      \colorrgb{0.00,0.00,0.00}%
      \put(462,1538){\makebox(0,0)[r]{\strut{}$\pi$}}%
      \colorrgb{0.00,0.00,0.00}%
      \put(462,2325){\makebox(0,0)[r]{\strut{}$2\pi$}}%
      \colorrgb{0.00,0.00,0.00}%
      \put(641,484){\makebox(0,0){\strut{}$0$}}%
      \colorrgb{0.00,0.00,0.00}%
      \put(1490,484){\makebox(0,0){\strut{}$\pi$}}%
      \colorrgb{0.00,0.00,0.00}%
      \put(2339,484){\makebox(0,0){\strut{}$2\pi$}}%
    }%
    \gplgaddtomacro\gplfronttext{%
      \csname LTb\endcsname%
      \put(220,1538){\makebox(0,0){\strut{}$\theta$}}%
      \put(1490,154){\makebox(0,0){\strut{}$\phi$}}%
    }%
    \gplbacktext
    \put(0,0){\includegraphics{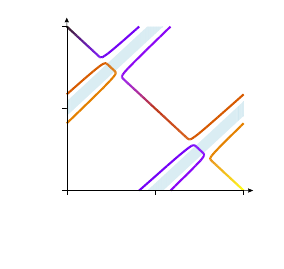}}%
    \gplfronttext
  \end{picture}%
\endgroup

%% file: shockflow2t3.tex
\begingroup
  \makeatletter
  \providecommand\color[2][]{%
    \GenericError{(gnuplot) \space\space\space\@spaces}{%
      Package color not loaded in conjunction with
      terminal option `colourtext'%
    }{See the gnuplot documentation for explanation.%
    }{Either use 'blacktext' in gnuplot or load the package
      color.sty in LaTeX.}%
    \renewcommand\color[2][]{}%
  }%
  \providecommand\includegraphics[2][]{%
    \GenericError{(gnuplot) \space\space\space\@spaces}{%
      Package graphicx or graphics not loaded%
    }{See the gnuplot documentation for explanation.%
    }{The gnuplot epslatex terminal needs graphicx.sty or graphics.sty.}%
    \renewcommand\includegraphics[2][]{}%
  }%
  \providecommand\rotatebox[2]{#2}%
  \@ifundefined{ifGPcolor}{%
    \newif\ifGPcolor
    \GPcolortrue
  }{}%
  \@ifundefined{ifGPblacktext}{%
    \newif\ifGPblacktext
    \GPblacktextfalse
  }{}%
  \let\gplgaddtomacro\g@addto@macro
  \gdef\gplbacktext{}%
  \gdef\gplfronttext{}%
  \makeatother
  \ifGPblacktext
    \def\colorrgb#1{}%
    \def\colorgray#1{}%
  \else
    \ifGPcolor
      \def\colorrgb#1{\color[rgb]{#1}}%
      \def\colorgray#1{\color[gray]{#1}}%
      \expandafter\def\csname LTw\endcsname{\color{white}}%
      \expandafter\def\csname LTb\endcsname{\color{black}}%
      \expandafter\def\csname LTa\endcsname{\color{black}}%
      \expandafter\def\csname LT0\endcsname{\color[rgb]{1,0,0}}%
      \expandafter\def\csname LT1\endcsname{\color[rgb]{0,1,0}}%
      \expandafter\def\csname LT2\endcsname{\color[rgb]{0,0,1}}%
      \expandafter\def\csname LT3\endcsname{\color[rgb]{1,0,1}}%
      \expandafter\def\csname LT4\endcsname{\color[rgb]{0,1,1}}%
      \expandafter\def\csname LT5\endcsname{\color[rgb]{1,1,0}}%
      \expandafter\def\csname LT6\endcsname{\color[rgb]{0,0,0}}%
      \expandafter\def\csname LT7\endcsname{\color[rgb]{1,0.3,0}}%
      \expandafter\def\csname LT8\endcsname{\color[rgb]{0.5,0.5,0.5}}%
    \else
      \def\colorrgb#1{\color{black}}%
      \def\colorgray#1{\color[gray]{#1}}%
      \expandafter\def\csname LTw\endcsname{\color{white}}%
      \expandafter\def\csname LTb\endcsname{\color{black}}%
      \expandafter\def\csname LTa\endcsname{\color{black}}%
      \expandafter\def\csname LT0\endcsname{\color{black}}%
      \expandafter\def\csname LT1\endcsname{\color{black}}%
      \expandafter\def\csname LT2\endcsname{\color{black}}%
      \expandafter\def\csname LT3\endcsname{\color{black}}%
      \expandafter\def\csname LT4\endcsname{\color{black}}%
      \expandafter\def\csname LT5\endcsname{\color{black}}%
      \expandafter\def\csname LT6\endcsname{\color{black}}%
      \expandafter\def\csname LT7\endcsname{\color{black}}%
      \expandafter\def\csname LT8\endcsname{\color{black}}%
    \fi
  \fi
    \setlength{\unitlength}{0.0500bp}%
    \ifx\gptboxheight\undefined%
      \newlength{\gptboxheight}%
      \newlength{\gptboxwidth}%
      \newsavebox{\gptboxtext}%
    \fi%
    \setlength{\fboxrule}{0.5pt}%
    \setlength{\fboxsep}{1pt}%
\begin{picture}(2736.00,2590.00)%
    \gplgaddtomacro\gplbacktext{%
      \colorrgb{0.00,0.00,0.00}%
      \put(462,751){\makebox(0,0)[r]{\strut{}$0$}}%
      \colorrgb{0.00,0.00,0.00}%
      \put(462,1538){\makebox(0,0)[r]{\strut{}$\pi$}}%
      \colorrgb{0.00,0.00,0.00}%
      \put(462,2325){\makebox(0,0)[r]{\strut{}$2\pi$}}%
      \colorrgb{0.00,0.00,0.00}%
      \put(641,484){\makebox(0,0){\strut{}$0$}}%
      \colorrgb{0.00,0.00,0.00}%
      \put(1490,484){\makebox(0,0){\strut{}$\pi$}}%
      \colorrgb{0.00,0.00,0.00}%
      \put(2339,484){\makebox(0,0){\strut{}$2\pi$}}%
    }%
    \gplgaddtomacro\gplfronttext{%
      \csname LTb\endcsname%
      \put(220,1538){\makebox(0,0){\strut{}$\theta$}}%
      \put(1490,154){\makebox(0,0){\strut{}$\phi$}}%
    }%
    \gplbacktext
    \put(0,0){\includegraphics{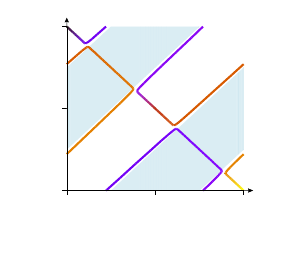}}%
    \gplfronttext
  \end{picture}%
\endgroup

%% file: shockflow2t4.tex
\begingroup
  \makeatletter
  \providecommand\color[2][]{%
    \GenericError{(gnuplot) \space\space\space\@spaces}{%
      Package color not loaded in conjunction with
      terminal option `colourtext'%
    }{See the gnuplot documentation for explanation.%
    }{Either use 'blacktext' in gnuplot or load the package
      color.sty in LaTeX.}%
    \renewcommand\color[2][]{}%
  }%
  \providecommand\includegraphics[2][]{%
    \GenericError{(gnuplot) \space\space\space\@spaces}{%
      Package graphicx or graphics not loaded%
    }{See the gnuplot documentation for explanation.%
    }{The gnuplot epslatex terminal needs graphicx.sty or graphics.sty.}%
    \renewcommand\includegraphics[2][]{}%
  }%
  \providecommand\rotatebox[2]{#2}%
  \@ifundefined{ifGPcolor}{%
    \newif\ifGPcolor
    \GPcolortrue
  }{}%
  \@ifundefined{ifGPblacktext}{%
    \newif\ifGPblacktext
    \GPblacktextfalse
  }{}%
  \let\gplgaddtomacro\g@addto@macro
  \gdef\gplbacktext{}%
  \gdef\gplfronttext{}%
  \makeatother
  \ifGPblacktext
    \def\colorrgb#1{}%
    \def\colorgray#1{}%
  \else
    \ifGPcolor
      \def\colorrgb#1{\color[rgb]{#1}}%
      \def\colorgray#1{\color[gray]{#1}}%
      \expandafter\def\csname LTw\endcsname{\color{white}}%
      \expandafter\def\csname LTb\endcsname{\color{black}}%
      \expandafter\def\csname LTa\endcsname{\color{black}}%
      \expandafter\def\csname LT0\endcsname{\color[rgb]{1,0,0}}%
      \expandafter\def\csname LT1\endcsname{\color[rgb]{0,1,0}}%
      \expandafter\def\csname LT2\endcsname{\color[rgb]{0,0,1}}%
      \expandafter\def\csname LT3\endcsname{\color[rgb]{1,0,1}}%
      \expandafter\def\csname LT4\endcsname{\color[rgb]{0,1,1}}%
      \expandafter\def\csname LT5\endcsname{\color[rgb]{1,1,0}}%
      \expandafter\def\csname LT6\endcsname{\color[rgb]{0,0,0}}%
      \expandafter\def\csname LT7\endcsname{\color[rgb]{1,0.3,0}}%
      \expandafter\def\csname LT8\endcsname{\color[rgb]{0.5,0.5,0.5}}%
    \else
      \def\colorrgb#1{\color{black}}%
      \def\colorgray#1{\color[gray]{#1}}%
      \expandafter\def\csname LTw\endcsname{\color{white}}%
      \expandafter\def\csname LTb\endcsname{\color{black}}%
      \expandafter\def\csname LTa\endcsname{\color{black}}%
      \expandafter\def\csname LT0\endcsname{\color{black}}%
      \expandafter\def\csname LT1\endcsname{\color{black}}%
      \expandafter\def\csname LT2\endcsname{\color{black}}%
      \expandafter\def\csname LT3\endcsname{\color{black}}%
      \expandafter\def\csname LT4\endcsname{\color{black}}%
      \expandafter\def\csname LT5\endcsname{\color{black}}%
      \expandafter\def\csname LT6\endcsname{\color{black}}%
      \expandafter\def\csname LT7\endcsname{\color{black}}%
      \expandafter\def\csname LT8\endcsname{\color{black}}%
    \fi
  \fi
    \setlength{\unitlength}{0.0500bp}%
    \ifx\gptboxheight\undefined%
      \newlength{\gptboxheight}%
      \newlength{\gptboxwidth}%
      \newsavebox{\gptboxtext}%
    \fi%
    \setlength{\fboxrule}{0.5pt}%
    \setlength{\fboxsep}{1pt}%
\begin{picture}(2736.00,2590.00)%
    \gplgaddtomacro\gplbacktext{%
      \colorrgb{0.00,0.00,0.00}%
      \put(462,751){\makebox(0,0)[r]{\strut{}$0$}}%
      \colorrgb{0.00,0.00,0.00}%
      \put(462,1538){\makebox(0,0)[r]{\strut{}$\pi$}}%
      \colorrgb{0.00,0.00,0.00}%
      \put(462,2325){\makebox(0,0)[r]{\strut{}$2\pi$}}%
      \colorrgb{0.00,0.00,0.00}%
      \put(641,484){\makebox(0,0){\strut{}$0$}}%
      \colorrgb{0.00,0.00,0.00}%
      \put(1490,484){\makebox(0,0){\strut{}$\pi$}}%
      \colorrgb{0.00,0.00,0.00}%
      \put(2339,484){\makebox(0,0){\strut{}$2\pi$}}%
    }%
    \gplgaddtomacro\gplfronttext{%
      \csname LTb\endcsname%
      \put(220,1538){\makebox(0,0){\strut{}$\theta$}}%
      \put(1490,154){\makebox(0,0){\strut{}$\phi$}}%
    }%
    \gplbacktext
    \put(0,0){\includegraphics{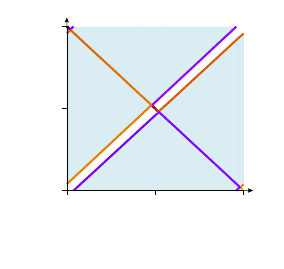}}%
    \gplfronttext
  \end{picture}%
\endgroup

%% file: shockflow2t5.tex
\begingroup
  \makeatletter
  \providecommand\color[2][]{%
    \GenericError{(gnuplot) \space\space\space\@spaces}{%
      Package color not loaded in conjunction with
      terminal option `colourtext'%
    }{See the gnuplot documentation for explanation.%
    }{Either use 'blacktext' in gnuplot or load the package
      color.sty in LaTeX.}%
    \renewcommand\color[2][]{}%
  }%
  \providecommand\includegraphics[2][]{%
    \GenericError{(gnuplot) \space\space\space\@spaces}{%
      Package graphicx or graphics not loaded%
    }{See the gnuplot documentation for explanation.%
    }{The gnuplot epslatex terminal needs graphicx.sty or graphics.sty.}%
    \renewcommand\includegraphics[2][]{}%
  }%
  \providecommand\rotatebox[2]{#2}%
  \@ifundefined{ifGPcolor}{%
    \newif\ifGPcolor
    \GPcolortrue
  }{}%
  \@ifundefined{ifGPblacktext}{%
    \newif\ifGPblacktext
    \GPblacktextfalse
  }{}%
  \let\gplgaddtomacro\g@addto@macro
  \gdef\gplbacktext{}%
  \gdef\gplfronttext{}%
  \makeatother
  \ifGPblacktext
    \def\colorrgb#1{}%
    \def\colorgray#1{}%
  \else
    \ifGPcolor
      \def\colorrgb#1{\color[rgb]{#1}}%
      \def\colorgray#1{\color[gray]{#1}}%
      \expandafter\def\csname LTw\endcsname{\color{white}}%
      \expandafter\def\csname LTb\endcsname{\color{black}}%
      \expandafter\def\csname LTa\endcsname{\color{black}}%
      \expandafter\def\csname LT0\endcsname{\color[rgb]{1,0,0}}%
      \expandafter\def\csname LT1\endcsname{\color[rgb]{0,1,0}}%
      \expandafter\def\csname LT2\endcsname{\color[rgb]{0,0,1}}%
      \expandafter\def\csname LT3\endcsname{\color[rgb]{1,0,1}}%
      \expandafter\def\csname LT4\endcsname{\color[rgb]{0,1,1}}%
      \expandafter\def\csname LT5\endcsname{\color[rgb]{1,1,0}}%
      \expandafter\def\csname LT6\endcsname{\color[rgb]{0,0,0}}%
      \expandafter\def\csname LT7\endcsname{\color[rgb]{1,0.3,0}}%
      \expandafter\def\csname LT8\endcsname{\color[rgb]{0.5,0.5,0.5}}%
    \else
      \def\colorrgb#1{\color{black}}%
      \def\colorgray#1{\color[gray]{#1}}%
      \expandafter\def\csname LTw\endcsname{\color{white}}%
      \expandafter\def\csname LTb\endcsname{\color{black}}%
      \expandafter\def\csname LTa\endcsname{\color{black}}%
      \expandafter\def\csname LT0\endcsname{\color{black}}%
      \expandafter\def\csname LT1\endcsname{\color{black}}%
      \expandafter\def\csname LT2\endcsname{\color{black}}%
      \expandafter\def\csname LT3\endcsname{\color{black}}%
      \expandafter\def\csname LT4\endcsname{\color{black}}%
      \expandafter\def\csname LT5\endcsname{\color{black}}%
      \expandafter\def\csname LT6\endcsname{\color{black}}%
      \expandafter\def\csname LT7\endcsname{\color{black}}%
      \expandafter\def\csname LT8\endcsname{\color{black}}%
    \fi
  \fi
    \setlength{\unitlength}{0.0500bp}%
    \ifx\gptboxheight\undefined%
      \newlength{\gptboxheight}%
      \newlength{\gptboxwidth}%
      \newsavebox{\gptboxtext}%
    \fi%
    \setlength{\fboxrule}{0.5pt}%
    \setlength{\fboxsep}{1pt}%
\begin{picture}(2736.00,2590.00)%
    \gplgaddtomacro\gplbacktext{%
      \colorrgb{0.00,0.00,0.00}%
      \put(462,751){\makebox(0,0)[r]{\strut{}$0$}}%
      \colorrgb{0.00,0.00,0.00}%
      \put(462,1538){\makebox(0,0)[r]{\strut{}$\pi$}}%
      \colorrgb{0.00,0.00,0.00}%
      \put(462,2325){\makebox(0,0)[r]{\strut{}$2\pi$}}%
      \colorrgb{0.00,0.00,0.00}%
      \put(641,484){\makebox(0,0){\strut{}$0$}}%
      \colorrgb{0.00,0.00,0.00}%
      \put(1490,484){\makebox(0,0){\strut{}$\pi$}}%
      \colorrgb{0.00,0.00,0.00}%
      \put(2339,484){\makebox(0,0){\strut{}$2\pi$}}%
    }%
    \gplgaddtomacro\gplfronttext{%
      \csname LTb\endcsname%
      \put(220,1538){\makebox(0,0){\strut{}$\theta$}}%
      \put(1490,154){\makebox(0,0){\strut{}$\phi$}}%
    }%
    \gplbacktext
    \put(0,0){\includegraphics{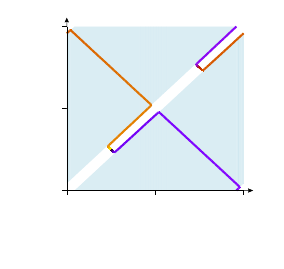}}%
    \gplfronttext
  \end{picture}%
\endgroup

%% file: shockflow2t6.tex
\begingroup
  \makeatletter
  \providecommand\color[2][]{%
    \GenericError{(gnuplot) \space\space\space\@spaces}{%
      Package color not loaded in conjunction with
      terminal option `colourtext'%
    }{See the gnuplot documentation for explanation.%
    }{Either use 'blacktext' in gnuplot or load the package
      color.sty in LaTeX.}%
    \renewcommand\color[2][]{}%
  }%
  \providecommand\includegraphics[2][]{%
    \GenericError{(gnuplot) \space\space\space\@spaces}{%
      Package graphicx or graphics not loaded%
    }{See the gnuplot documentation for explanation.%
    }{The gnuplot epslatex terminal needs graphicx.sty or graphics.sty.}%
    \renewcommand\includegraphics[2][]{}%
  }%
  \providecommand\rotatebox[2]{#2}%
  \@ifundefined{ifGPcolor}{%
    \newif\ifGPcolor
    \GPcolortrue
  }{}%
  \@ifundefined{ifGPblacktext}{%
    \newif\ifGPblacktext
    \GPblacktextfalse
  }{}%
  \let\gplgaddtomacro\g@addto@macro
  \gdef\gplbacktext{}%
  \gdef\gplfronttext{}%
  \makeatother
  \ifGPblacktext
    \def\colorrgb#1{}%
    \def\colorgray#1{}%
  \else
    \ifGPcolor
      \def\colorrgb#1{\color[rgb]{#1}}%
      \def\colorgray#1{\color[gray]{#1}}%
      \expandafter\def\csname LTw\endcsname{\color{white}}%
      \expandafter\def\csname LTb\endcsname{\color{black}}%
      \expandafter\def\csname LTa\endcsname{\color{black}}%
      \expandafter\def\csname LT0\endcsname{\color[rgb]{1,0,0}}%
      \expandafter\def\csname LT1\endcsname{\color[rgb]{0,1,0}}%
      \expandafter\def\csname LT2\endcsname{\color[rgb]{0,0,1}}%
      \expandafter\def\csname LT3\endcsname{\color[rgb]{1,0,1}}%
      \expandafter\def\csname LT4\endcsname{\color[rgb]{0,1,1}}%
      \expandafter\def\csname LT5\endcsname{\color[rgb]{1,1,0}}%
      \expandafter\def\csname LT6\endcsname{\color[rgb]{0,0,0}}%
      \expandafter\def\csname LT7\endcsname{\color[rgb]{1,0.3,0}}%
      \expandafter\def\csname LT8\endcsname{\color[rgb]{0.5,0.5,0.5}}%
    \else
      \def\colorrgb#1{\color{black}}%
      \def\colorgray#1{\color[gray]{#1}}%
      \expandafter\def\csname LTw\endcsname{\color{white}}%
      \expandafter\def\csname LTb\endcsname{\color{black}}%
      \expandafter\def\csname LTa\endcsname{\color{black}}%
      \expandafter\def\csname LT0\endcsname{\color{black}}%
      \expandafter\def\csname LT1\endcsname{\color{black}}%
      \expandafter\def\csname LT2\endcsname{\color{black}}%
      \expandafter\def\csname LT3\endcsname{\color{black}}%
      \expandafter\def\csname LT4\endcsname{\color{black}}%
      \expandafter\def\csname LT5\endcsname{\color{black}}%
      \expandafter\def\csname LT6\endcsname{\color{black}}%
      \expandafter\def\csname LT7\endcsname{\color{black}}%
      \expandafter\def\csname LT8\endcsname{\color{black}}%
    \fi
  \fi
    \setlength{\unitlength}{0.0500bp}%
    \ifx\gptboxheight\undefined%
      \newlength{\gptboxheight}%
      \newlength{\gptboxwidth}%
      \newsavebox{\gptboxtext}%
    \fi%
    \setlength{\fboxrule}{0.5pt}%
    \setlength{\fboxsep}{1pt}%
\begin{picture}(2736.00,2590.00)%
    \gplgaddtomacro\gplbacktext{%
      \colorrgb{0.00,0.00,0.00}%
      \put(462,751){\makebox(0,0)[r]{\strut{}$0$}}%
      \colorrgb{0.00,0.00,0.00}%
      \put(462,1538){\makebox(0,0)[r]{\strut{}$\pi$}}%
      \colorrgb{0.00,0.00,0.00}%
      \put(462,2325){\makebox(0,0)[r]{\strut{}$2\pi$}}%
      \colorrgb{0.00,0.00,0.00}%
      \put(641,484){\makebox(0,0){\strut{}$0$}}%
      \colorrgb{0.00,0.00,0.00}%
      \put(1490,484){\makebox(0,0){\strut{}$\pi$}}%
      \colorrgb{0.00,0.00,0.00}%
      \put(2339,484){\makebox(0,0){\strut{}$2\pi$}}%
    }%
    \gplgaddtomacro\gplfronttext{%
      \csname LTb\endcsname%
      \put(220,1538){\makebox(0,0){\strut{}$\theta$}}%
      \put(1490,154){\makebox(0,0){\strut{}$\phi$}}%
    }%
    \gplbacktext
    \put(0,0){\includegraphics{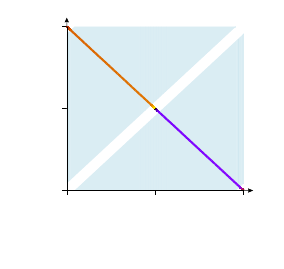}}%
    \gplfronttext
  \end{picture}%
\endgroup

%% file: shockflow2t7.tex
\begingroup
  \makeatletter
  \providecommand\color[2][]{%
    \GenericError{(gnuplot) \space\space\space\@spaces}{%
      Package color not loaded in conjunction with
      terminal option `colourtext'%
    }{See the gnuplot documentation for explanation.%
    }{Either use 'blacktext' in gnuplot or load the package
      color.sty in LaTeX.}%
    \renewcommand\color[2][]{}%
  }%
  \providecommand\includegraphics[2][]{%
    \GenericError{(gnuplot) \space\space\space\@spaces}{%
      Package graphicx or graphics not loaded%
    }{See the gnuplot documentation for explanation.%
    }{The gnuplot epslatex terminal needs graphicx.sty or graphics.sty.}%
    \renewcommand\includegraphics[2][]{}%
  }%
  \providecommand\rotatebox[2]{#2}%
  \@ifundefined{ifGPcolor}{%
    \newif\ifGPcolor
    \GPcolortrue
  }{}%
  \@ifundefined{ifGPblacktext}{%
    \newif\ifGPblacktext
    \GPblacktextfalse
  }{}%
  \let\gplgaddtomacro\g@addto@macro
  \gdef\gplbacktext{}%
  \gdef\gplfronttext{}%
  \makeatother
  \ifGPblacktext
    \def\colorrgb#1{}%
    \def\colorgray#1{}%
  \else
    \ifGPcolor
      \def\colorrgb#1{\color[rgb]{#1}}%
      \def\colorgray#1{\color[gray]{#1}}%
      \expandafter\def\csname LTw\endcsname{\color{white}}%
      \expandafter\def\csname LTb\endcsname{\color{black}}%
      \expandafter\def\csname LTa\endcsname{\color{black}}%
      \expandafter\def\csname LT0\endcsname{\color[rgb]{1,0,0}}%
      \expandafter\def\csname LT1\endcsname{\color[rgb]{0,1,0}}%
      \expandafter\def\csname LT2\endcsname{\color[rgb]{0,0,1}}%
      \expandafter\def\csname LT3\endcsname{\color[rgb]{1,0,1}}%
      \expandafter\def\csname LT4\endcsname{\color[rgb]{0,1,1}}%
      \expandafter\def\csname LT5\endcsname{\color[rgb]{1,1,0}}%
      \expandafter\def\csname LT6\endcsname{\color[rgb]{0,0,0}}%
      \expandafter\def\csname LT7\endcsname{\color[rgb]{1,0.3,0}}%
      \expandafter\def\csname LT8\endcsname{\color[rgb]{0.5,0.5,0.5}}%
    \else
      \def\colorrgb#1{\color{black}}%
      \def\colorgray#1{\color[gray]{#1}}%
      \expandafter\def\csname LTw\endcsname{\color{white}}%
      \expandafter\def\csname LTb\endcsname{\color{black}}%
      \expandafter\def\csname LTa\endcsname{\color{black}}%
      \expandafter\def\csname LT0\endcsname{\color{black}}%
      \expandafter\def\csname LT1\endcsname{\color{black}}%
      \expandafter\def\csname LT2\endcsname{\color{black}}%
      \expandafter\def\csname LT3\endcsname{\color{black}}%
      \expandafter\def\csname LT4\endcsname{\color{black}}%
      \expandafter\def\csname LT5\endcsname{\color{black}}%
      \expandafter\def\csname LT6\endcsname{\color{black}}%
      \expandafter\def\csname LT7\endcsname{\color{black}}%
      \expandafter\def\csname LT8\endcsname{\color{black}}%
    \fi
  \fi
    \setlength{\unitlength}{0.0500bp}%
    \ifx\gptboxheight\undefined%
      \newlength{\gptboxheight}%
      \newlength{\gptboxwidth}%
      \newsavebox{\gptboxtext}%
    \fi%
    \setlength{\fboxrule}{0.5pt}%
    \setlength{\fboxsep}{1pt}%
\begin{picture}(2736.00,2590.00)%
    \gplgaddtomacro\gplbacktext{%
      \colorrgb{0.00,0.00,0.00}%
      \put(462,751){\makebox(0,0)[r]{\strut{}$0$}}%
      \colorrgb{0.00,0.00,0.00}%
      \put(462,1538){\makebox(0,0)[r]{\strut{}$\pi$}}%
      \colorrgb{0.00,0.00,0.00}%
      \put(462,2325){\makebox(0,0)[r]{\strut{}$2\pi$}}%
      \colorrgb{0.00,0.00,0.00}%
      \put(641,484){\makebox(0,0){\strut{}$0$}}%
      \colorrgb{0.00,0.00,0.00}%
      \put(1490,484){\makebox(0,0){\strut{}$\pi$}}%
      \colorrgb{0.00,0.00,0.00}%
      \put(2339,484){\makebox(0,0){\strut{}$2\pi$}}%
    }%
    \gplgaddtomacro\gplfronttext{%
      \csname LTb\endcsname%
      \put(220,1538){\makebox(0,0){\strut{}$\theta$}}%
      \put(1490,154){\makebox(0,0){\strut{}$\phi$}}%
    }%
    \gplbacktext
    \put(0,0){\includegraphics{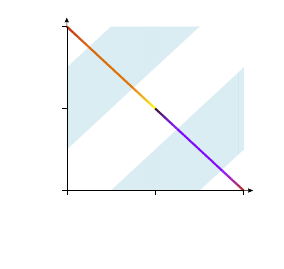}}%
    \gplfronttext
  \end{picture}%
\endgroup

%% file: shockflow2t8.tex
\begingroup
  \makeatletter
  \providecommand\color[2][]{%
    \GenericError{(gnuplot) \space\space\space\@spaces}{%
      Package color not loaded in conjunction with
      terminal option `colourtext'%
    }{See the gnuplot documentation for explanation.%
    }{Either use 'blacktext' in gnuplot or load the package
      color.sty in LaTeX.}%
    \renewcommand\color[2][]{}%
  }%
  \providecommand\includegraphics[2][]{%
    \GenericError{(gnuplot) \space\space\space\@spaces}{%
      Package graphicx or graphics not loaded%
    }{See the gnuplot documentation for explanation.%
    }{The gnuplot epslatex terminal needs graphicx.sty or graphics.sty.}%
    \renewcommand\includegraphics[2][]{}%
  }%
  \providecommand\rotatebox[2]{#2}%
  \@ifundefined{ifGPcolor}{%
    \newif\ifGPcolor
    \GPcolortrue
  }{}%
  \@ifundefined{ifGPblacktext}{%
    \newif\ifGPblacktext
    \GPblacktextfalse
  }{}%
  \let\gplgaddtomacro\g@addto@macro
  \gdef\gplbacktext{}%
  \gdef\gplfronttext{}%
  \makeatother
  \ifGPblacktext
    \def\colorrgb#1{}%
    \def\colorgray#1{}%
  \else
    \ifGPcolor
      \def\colorrgb#1{\color[rgb]{#1}}%
      \def\colorgray#1{\color[gray]{#1}}%
      \expandafter\def\csname LTw\endcsname{\color{white}}%
      \expandafter\def\csname LTb\endcsname{\color{black}}%
      \expandafter\def\csname LTa\endcsname{\color{black}}%
      \expandafter\def\csname LT0\endcsname{\color[rgb]{1,0,0}}%
      \expandafter\def\csname LT1\endcsname{\color[rgb]{0,1,0}}%
      \expandafter\def\csname LT2\endcsname{\color[rgb]{0,0,1}}%
      \expandafter\def\csname LT3\endcsname{\color[rgb]{1,0,1}}%
      \expandafter\def\csname LT4\endcsname{\color[rgb]{0,1,1}}%
      \expandafter\def\csname LT5\endcsname{\color[rgb]{1,1,0}}%
      \expandafter\def\csname LT6\endcsname{\color[rgb]{0,0,0}}%
      \expandafter\def\csname LT7\endcsname{\color[rgb]{1,0.3,0}}%
      \expandafter\def\csname LT8\endcsname{\color[rgb]{0.5,0.5,0.5}}%
    \else
      \def\colorrgb#1{\color{black}}%
      \def\colorgray#1{\color[gray]{#1}}%
      \expandafter\def\csname LTw\endcsname{\color{white}}%
      \expandafter\def\csname LTb\endcsname{\color{black}}%
      \expandafter\def\csname LTa\endcsname{\color{black}}%
      \expandafter\def\csname LT0\endcsname{\color{black}}%
      \expandafter\def\csname LT1\endcsname{\color{black}}%
      \expandafter\def\csname LT2\endcsname{\color{black}}%
      \expandafter\def\csname LT3\endcsname{\color{black}}%
      \expandafter\def\csname LT4\endcsname{\color{black}}%
      \expandafter\def\csname LT5\endcsname{\color{black}}%
      \expandafter\def\csname LT6\endcsname{\color{black}}%
      \expandafter\def\csname LT7\endcsname{\color{black}}%
      \expandafter\def\csname LT8\endcsname{\color{black}}%
    \fi
  \fi
    \setlength{\unitlength}{0.0500bp}%
    \ifx\gptboxheight\undefined%
      \newlength{\gptboxheight}%
      \newlength{\gptboxwidth}%
      \newsavebox{\gptboxtext}%
    \fi%
    \setlength{\fboxrule}{0.5pt}%
    \setlength{\fboxsep}{1pt}%
\begin{picture}(2736.00,2590.00)%
    \gplgaddtomacro\gplbacktext{%
      \colorrgb{0.00,0.00,0.00}%
      \put(462,751){\makebox(0,0)[r]{\strut{}$0$}}%
      \colorrgb{0.00,0.00,0.00}%
      \put(462,1538){\makebox(0,0)[r]{\strut{}$\pi$}}%
      \colorrgb{0.00,0.00,0.00}%
      \put(462,2325){\makebox(0,0)[r]{\strut{}$2\pi$}}%
      \colorrgb{0.00,0.00,0.00}%
      \put(641,484){\makebox(0,0){\strut{}$0$}}%
      \colorrgb{0.00,0.00,0.00}%
      \put(1490,484){\makebox(0,0){\strut{}$\pi$}}%
      \colorrgb{0.00,0.00,0.00}%
      \put(2339,484){\makebox(0,0){\strut{}$2\pi$}}%
    }%
    \gplgaddtomacro\gplfronttext{%
      \csname LTb\endcsname%
      \put(220,1538){\makebox(0,0){\strut{}$\theta$}}%
      \put(1490,154){\makebox(0,0){\strut{}$\phi$}}%
    }%
    \gplbacktext
    \put(0,0){\includegraphics{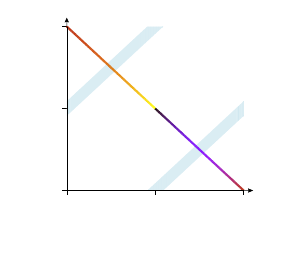}}%
    \gplfronttext
  \end{picture}%
\endgroup